\documentclass[a4paper,12pt,intlimits]{amsart}
\usepackage{amssymb}

\usepackage{graphicx}  \usepackage{epstopdf}

\usepackage[english]{babel}

\usepackage{amsmath}
\usepackage{comment}

\usepackage{amsthm, amsfonts, amssymb, latexsym, graphics, showidx}
\usepackage{color}

\numberwithin{equation}{section}
\newtheorem{de}{Definition}[section]
\newtheorem{thm}{Theorem}[section]
\newtheorem{rem}[thm]{Remark}
\newtheorem{cor}[thm]{Corollary}
\newtheorem{prop}[thm]{Proposition}
\newtheorem{lem}[thm]{Lemma}

\renewcommand{\dim}{\noindent\textbf{Proof.} }

\newcommand{\finedim}{{\unskip\nobreak\hfil\penalty50
   \hskip2em\hbox{}\nobreak\hfil\mbox{$\Box$ \qquad}
   \parfillskip=0pt \finalhyphendemerits=0\par\medskip}}
\newcommand{\R}{\mathbb{R}}

\newcommand{\Z}{\mathbb{Z}}
\newcommand{\Om}{\Omega}

\newcommand{\al}{\alpha}

\newcommand{\xs}{\overline{x}}

\newcommand{\ts}{\overline{t}}

\newcommand{\p}{\partial}

\renewcommand{\theta}{\vartheta}
\renewcommand{\epsilon}{\varepsilon}
\newcommand{\ep}{\epsilon}
\newcommand{\z}{\zeta}
\newcommand{\I}{\mathcal{I}_1}

\newcommand{\us}{\overline{u}}

\renewcommand{\leq}{\leqslant}
\renewcommand{\le}{\leqslant}
\renewcommand{\geq}{\geqslant}
\renewcommand{\ge}{\geqslant}

\newcommand{\beq}{\begin{equation}}
\newcommand{\eeq}{\end{equation}}
\newcommand{\beqs}{\begin{equation*}}
\newcommand{\eeqs}{\end{equation*}}
\newcommand{\beqa}{\begin{eqnarray}}
\newcommand{\eeqa}{\end{eqnarray}}
\newcommand{\beqas}{\begin{eqnarray*}}
\newcommand{\eeqas}{\end{eqnarray*}}

\title[]{Derivation of the  1-D Groma-Balogh equations from the Peierls-Nabarro model}

\author{Stefania Patrizi and Tharathep Sangsawang}
\thanks{The first  author has been supported by the
NSF Grant DMS-2155156 " Nonlinear PDE methods in the study of interphases"}

\address[Stefania Patrizi and Tharathep Sangsawang]{
Department of Mathematics,
University of Texas at Austin,
2515 Speedway Stop C1200,
Austin, Texas 78712-1202, USA}

\email{spatrizi@math.utexas.edu} 
\email{tsangsaw@math.utexas.edu}

\subjclass[2010]{82D25, 35R09, 74E15, 35R11, 47G20.}
\keywords{Peierls-Nabarro model, nonlocal integro-differential equations,
dislocation dynamics, fractional Allen Cahn, phase tranditions.}

\begin{document}

\begin{abstract}
We consider a semi-linear integro-differential equation  in dimension one associated to the half Laplacian
whose solution represents the atom dislocation in a crystal.
The equation  comprises the evolutive version of the classical
Peierls-Nabarro model. 
We show that for  a large number of dislocations, the solution, properly rescaled, converges to the  solution of a fully nonlinear 
integro-differential equation 
which is a model for the macroscopic 
 crystal plasticity  with density of dislocations.   This leads to the formal derivation of the 1-D  Groma-Balogh equations \cite{groma},   
a  popular model 
describing the evolution of the density of positive and negative oriented parallel straight dislocation lines. 
This paper completes the work of \cite{patsan}. The main novelty here is that we allow dislocations to have different orientation and so we have to deal with collisions of them. 
\end{abstract}

\maketitle
\section{Introduction}
The goal of this paper is to complete the study started by the authors in \cite{patsan} of  the behavior as $(\ep,\delta)\to(0,0)$ of the solution $u^{\ep,\delta}$ of the following  evolutionary partial-integro-differential equation
\beq\label{uepeq}\begin{cases}
\delta\partial_t u^{\ep,\delta}=\I [u^{\ep,\delta}]-\displaystyle\frac{1}{\delta}W'\left(\frac{u^{\ep,\delta}}{\ep}\right)&\text{in }(0,+\infty)\times\R\\
u^{\ep,\delta}(0,\cdot)=u_0(\cdot)&\text{on }\R
\end{cases}\eeq
where $\ep,\,\delta>0$ are small scale parameters, $W$ is a multi-well periodic potential and we denote by  $\I$ the so-called fractional Laplacian 
 of order 1, $-(-\Delta)^\frac{1}{2}$, 
defined by 
$$ \I[v](x)=\frac{1}{\pi}PV
\displaystyle\int_{\R}\displaystyle\frac{v(y)-v(x)}{(y-x)^{2}}dy,$$ where PV stands for principal value. 
We refer to~\cite{s} or~\cite{dnpv} for a basic introduction
to the fractional Laplace operator.

Equation \eqref{uepeq} with $\ep=\delta=1$ arises in the Peierls-Nabarro model to describe at microscopic scale the motion of dislocation lines in crystals.
Dislocations are line defects in crystalline materials whose motion is responsible of the plastic  behavior of metals. Dislocations  can be described at several scales by different models:
\begin{itemize}
\item[a)] atomic scale (Frenkel-Kontorova model), 
\item [b)]microscopic scale (Peierls-Nabarro model), 
\item [c)] mesoscopic scale (Discrete dislocation dynamics), 
\item[d)] macroscopic scale (Elasto-visco-plasticity with density of dislocations). 
\end{itemize}
We refer the reader to the book \cite{hl} for a tour in the theory of dislocations. 
The 1-D Peierls-Nabarro model describes 
the  microscopic effect of an {\em ensemble of straight edge dislocation lines all lying in the same plane}. 
After a cross section, dislocation lines can be identified by points on a line. 
Every  dislocation  is associated to either a {\em positive} or a {\em negative orientation}, depending on the direction of the Burgers' vector (a fixed vector associated to the dislocation).  
Equation \eqref{uepeq} with $\ep=1$, which  is obtained after a parabolic rescaling of the original model, has been investigated in a series of papers \cite{gonzalezmonneau,dpv, dfv, pv2,pv4,pv3}.
The solution here, $u^{1,\delta}$,  is a phase transition function which represents   the atom displacement,
in terms of~$\delta$, which in turn
represents the size of the crystal scale. 
Starting from an initial configuration where the transitions occurs at some given points,  for small~$\delta$, the displacement  function
approaches a piecewise constant function. 
The plateaus of this asymptotic limit correspond to
the periodic sites induced by the crystalline structure,
but its jump points evolve in time, according to a singular potential. Roughly speaking, one can
imagine that the discontinuity points of this
limit displacement function behave like a ``particle'' system
(though no ``material'' particle is really involved), 
driven by a system of ordinary differential equations
which describe the position of the jump points~$y_1(t),\dots,y_N(t)$.
The system corresponds to the  discrete dislocation dynamics  and the convergence result is a passage from (b) to  (c). 
The physical properties of the singular potential of this ODE system depend on the orientation of the  displacement function at the jump points. Namely, if the displacement function has the same spatial monotonicity at $y_i$ and $y_{i+1}$ (i.e., $y_i$ and $y_{i+1}$ have the same orientation), then the potential induces a repulsion between the particles $y_i$ and $y_{i+1}$. Conversely, when the displacement  function has opposite spatial monotonicity at $y_i$ and $y_{i+1}$ (i.e., $y_i$ and $y_{i+1}$ have opposite orientation), then the potential becomes attractive, and the two particles may collide in a finite time.
We will give more details in Section \ref{PNsec}. Collisions create a problem in the analysis  as the dynamical system that governs the motion of the dislocation particles   ceases to be well-defined at the collision time. 
The study of the asymptotics of the  displacement function after collision time permits to understand how the dynamical law of the interphase points can be continued/extended after collisions, see \cite{pv3}. 

Different space/time scales of the original Peierls-Nabarro model  also produce homogenization results, see \cite{mp, mp2,pv}.  

  The model  can also be linked
to the classical model at the atomic scale which was introduced by 
Frenkel and Kontorova (see~\cite{fino}) (from (a) to (b)).

We refer to \cite{cddp,fim, gm, gmps, glp,mm,mmp,mora,sppg} for further related results.

\medskip
In \cite{patsan} the authors considered for the first time in this framework  the case in which the number of dislocations $N$ goes to $\infty$. We introduced a second parameter $\ep$ such that 
$N=N_\ep\to \infty $ as $\ep\to 0$. A parabolic rescaling in $\delta$ and hyperbolic rescaling in $\ep$ of the original model  leads to \eqref{uepeq}.  In \cite{patsan} we only considered the case when the dislocation points have all the same orientation, which in the model corresponds to assuming  the initial condition $u_0$ to be monotonic.
In the present paper, we remove the monotonicity assumption on $u_0$, {\em  allowing dislocations to have different orientation}.  More precisely,
on  $u_0$ we assume
\begin{equation}\label{u^0ass}
\begin{cases}
u_0\in C^{1,1}(\mathbb{R}),\\
\lim \limits_{x\to-\infty}u_0(x)=0,\\
\lim \limits_{x\to+\infty}u_0(x)=l, \text{ for some }l\in\R.\\
\end{cases}
\end{equation}
For fixed $\ep$, the dislocation  points at initial time are approximated by the points in the level sets $\{u_0=\ep i\}$, $i\in \Z$, while their  orientations are determined by the monotonicity of $u_0$ at the points. 
The limits in \eqref{u^0ass} guarantee that the dislocation points remain in a compact set for fixed $\ep$. The first limit is just a normalization, 0 could be replaced by any real number.


On the potential $W$ we assume
\begin{equation}\label{Wass}
\begin{cases}W\in C^{2,\beta}(\R)& \text{for some }0<\beta<1\\
W(u+1)=W(u)& \text{for any } u\in\R\\
W=0& \text{on }\Z\\
W>0 & \text{on }\R\setminus\Z\\
W''(0)>0.\\
\end{cases}
\end{equation}

 Our goal in this paper is to understand the large scale limit of the Peierls-Nabarro model  for  a  {\em large number of
  parallel straight edge dislocation lines lying in the same slip plane, with possibly different orientation,  moving with self-interactions}.
We perform a {\em direct passage from the model (b) to the model  (d) } and show that at macroscopic scale the density of dislocations is governed by the following evolution law:
\beq
\label{ubareq}\begin{cases}
\partial_t u=c_0|\partial_x u|\,\I[ u]&\text{in }(0,+\infty)\times\R\\
u(0,\cdot)=u_0&\text{on }\R
\end{cases}\eeq
where  $c_0>0$ is defined in the forthcoming \eqref{c0}.  Our main result is the following:
\begin{thm}\label{mainthm}
Assume  \eqref{u^0ass} and \eqref{Wass}.   Let $u^{\ep,\delta}$ be the solution of~\eqref{uepeq}.
There exists a number $A_\ep>0$ depending on $\ep$ and $u_0$ such that if $\delta A_\ep\to 0$ as $(\ep,\delta)\to(0,0)$, then $u^{\ep,\delta}$ converges locally uniformly in $[0,+\infty)\times\R$ to the viscosity solution $\us$ of~\eqref{ubareq}, as $\ep\to0$. 
\end{thm}
\begin{rem}
The quantity $A_\ep$ in Theorem \ref{mainthm}, will be made explicit later on, see Section \ref{Aepsec}. The condition  $\delta A_\ep\to 0$ as $(\ep,\delta)\to(0,0)$ is automatically satisfied  if we have some control on the number of dislocation points at time 0 with respect to $\ep$. This is the case, for example, when $u_0$ is either monotone or goes to $0$ and $l$, respectively as $x\to-\infty$ and $x\to+\infty$, faster or equal than respectively $c/x$ and $l+c/x$, for some $c>0$, see   Section \ref{Aepsec}. The latter condition is natural in this setting, see \eqref{phiinfinity}. In these two situations $\ep$ and $\delta$ are two free parameters. 
\end{rem}
\begin{rem}
For simplicity, in what follows we will assume that $\delta=\delta_\ep\to0$ as $\ep\to0$ and we will denote by 
$$u^\ep(t,x)=u^{\ep,\delta}(t,x)$$ the solution of~\eqref{uepeq}. 
\end{rem}
To prove Theorem \ref{mainthm}, the idea is to approximate the dislocation particles with points $x_i(t)$ where the limit function $\us$ attains the value $\ep i$ at time $t$, $i\in\Z$.
 We then  show that 
 $$b_i\dot{x}_i=-\frac{\partial \us_t(t,x_i(t))}{|\partial_x \us(t,x_i(t))|}\simeq -c_0\I[\us(t,\cdot)](x_i(t)),$$
 with $b_i=\text{sgn}(\partial_x \us(t,x_i(t)))$, provided $\partial_x \us(t,x_i(t))\neq 0$.

 One of the main difficulties in the proof  of Theorem \ref{mainthm} consists in proving that $\partial_t \us=0$ (in the viscosity sense) at points 
  where  $\partial_x \us$ vanishes. This result  is also the main novelty with respect to our previous work \cite{patsan}. Indeed in the monotonic case we could prove, 
   by  using an  approximation argument, that if $u_0$, and thus the limit function  $\us$, is monotone,  it is enough to test equation~\eqref{ubareq} with test functions for $\us$ with non vanishing derivative in $x$. 
That argument cannot be applied in the present setting and we have to deal with   the case of test functions with vanishing derivatives.   
Roughly speaking, points $x$ where $\partial_x \us(t,x)= 0$ correspond to the locations where  collisions occur at time $t$. The proof here is based on a new analysis of how the datum in \eqref{uepeq} is  transported along the characteristics $x_i(t)$ around a collision point.  
The strategy and the heuristic of the proofs are explained in Section \ref{strategysec}.

Differentiating equation \eqref{ubareq} formally yields the  following system of equations for the positive and negative part of $f=\partial_x \us$, 
\beq\label{gromaeq}\begin{split} &\partial_t f^+=c_0\partial_x( f^+\mathcal{H}(f^+-f^-)),\\&
 \partial_t f^-=-c_0\partial_x( f^-\mathcal{H}(f^+-f^-)),
 \end{split}\eeq
with $\mathcal{H}$  the Hilbert transform. 
Equations \eqref{gromaeq}  are  the 1-D version of the 2-D  Groma-Balogh equations \cite{groma},   
a  popular model 
describing the evolution of the density of positive and negative oriented parallel straight dislocation lines. 
This is the first time such equations are formally derived from  the microscopic Peierls-Nabarro model. 

As a by-product of the proof of Theorem \ref{mainthm} we also obtain the following asymptotic behavior of the limit function. 

\begin{prop}\label{uasymptoticprop}
The limit function $\us$ satisfies 
\beq\label{ulimits}\lim_{x\to-\infty} \us(t,x)=0,\quad
 \lim_{x\to+\infty}\us(t,x)=l,\eeq
uniformly in $t\in [0,T]$, for any $T>0$. 
Moreover,  for all $(t,x)\in (0,+\infty)\times\R$\beq\label{ulesssupinfu0}\inf u_0\leq \us(t,x)\leq \sup u_0.\eeq
\end{prop}
\begin{rem}
The limits \eqref{ulimits} can be interpreted as the property of the dislocations particles to remain in a compact set in the interval $[0,T]$. 
Property \eqref{ulesssupinfu0}, which is an easy consequence of the comparison principle, says that, while dislocations may annihilate, no 
 dislocations are created.
\end{rem}

\subsection{The Peierls-Nabarro model}\label{PNsec}
The Peierls-Nabarro model \cite{n,p} is a phase field model for dislocation
dynamics incorporating atomic features into continuum framework.
In a phase field approach, the dislocations are represented by
transition of a continuous field. 
We refer to~\cite{nabarro} for a survey of the
Peierls-Nabarro model. 
See also Section~1.1 in~\cite{patsan} for some basic physical derivation.  
After a section of the three-dimensional crystal with a plane, a straight  dislocation line can be identified with a point  on a line.
A positive oriented dislocation  located at 0 is described by the transition from 0 to 1 of the phase transition function $\phi$,  solution to 
\begin{equation}\label{phi}
\begin{cases}
\I[\phi]=W'(\phi)&\text{ in } \R\\
\phi'>0&\text{ in}\quad \R\\
\displaystyle\lim_{z\to-\infty}\displaystyle\phi(z)=0,\quad\lim_{z\to+\infty}\displaystyle\phi(z)=1, \quad \phi(0)=\displaystyle\frac{1}{2},
\end{cases}
\end{equation}
while a negative oriented dislocation  located at 0  is described by the transition from 0 to -1 of the function $\hat{\phi}$, solution to  
\begin{equation}\label{phi-}
\begin{cases}
\I[\hat{\phi}]=W'(\hat{\phi})&\text{ in } \R\\
\hat{\phi}'<0&\text{ in}\quad \R\\
\displaystyle\lim_{z\to-\infty}\displaystyle\hat{\phi}(z)=0,\quad\lim_{z\to+\infty}\displaystyle\hat{\phi}(z)=-1, \quad \hat{\phi}(0)=-\displaystyle\frac{1}{2}.
\end{cases}
\end{equation}
Under assumption  \eqref{Wass}, existence of a unique solution of \eqref{phi} has been proven in \cite{csm,psv}.  Notice that, by the periodicity of $W$,  
$\hat{\phi}(z)=\phi(-z)-1$ is the unique solution to \eqref{phi-}
Define
\beq \label{phibi}
\phi(z,b) := \begin{cases}
\phi(z) \quad &\text{if} \quad b = 1\\
\hat{\phi}(z) \quad &\text{if} \quad b = -1.
\end{cases}
\eeq
\medskip

In the face cubic structure (FCC) observed in many metals and
alloys, dislocations move at low temperature on the slip plane.
The dynamics for a collection of straight parallel edge dislocations lines, all contained in a single slip plane, moving with self-interactions (no exterior forces)
 is then
described by the evolutive version  of the Peierls-Nabarro model
 (see for instance \cite{MBW} and \cite{Denoual}):
\begin{equation}\label{nabarroevolutintro}
\p_{t} u=\I[u(t,\cdot)]-W'\left(u\right)\quad\text{in}\quad (0,+\infty)\times\R,\\
\end{equation}
with initial condition 
\beq\label{nabarroevolutintroinitial}u(0,x)=\sum_{i=1}^{N}\phi\left(x-\frac{y_i^0}{\delta}, b_i\right),\eeq
where $\phi$ is the solution of \eqref{phi}, $N$ is  the number of dislocations,
$y_i^0/\delta$ are the initial locations of the dislocation points and neighboring dislocations are at distance at microscopic scale of order  $1/\delta$, that is 
$$0\le y^0_{i+1}-y^0_i\sim 1.$$ 
The number $b_i\in\{-1,1\}$ identifies the orientation of the dislocation: when  $b_i=1$ the dislocation is positive oriented, when   $b_i=-1$ it is instead  negative oriented. 

Let $u$ be the solution of \eqref{nabarroevolutintro}  with initial condition \eqref{nabarroevolutintroinitial}. 
Then,  the rescaled function 
$$v^\delta(t,x)= u\left(\frac{t}{\delta^2},\frac{x}{\delta}\right),$$ which is solution to the integro-differential equation in \eqref{uepeq} with $\ep=1$ converges as $\delta\to0$ 
 to a sum of Heaviside functions of the form 
$\sum_{i=1}^N H(b_i(x-y_i(t)))$, where 
the interphase (jump) points $y_i(t)$,  $i=1,\ldots,N$ evolve in time   driven by  the following system of  ODE:
\beq\label{DDD}\begin{cases}\dot{y}_i=c_0\displaystyle\sum_{j\neq i}\frac{b_ib_j}{y_i-y_j}&\text{in }(0,T_c)\\
y_i(0)=y_i^0,
\end{cases}
\eeq
where $c_0$ is defined by
\beq\label{c0}c_0=\left(\int_\R (\phi')^2\right)^{-1},\eeq
see \cite{gonzalezmonneau, pv3}. 
 Here $0<T_c\le+\infty$ is the first time a collision  between opposite oriented interphase points occurs. Indeed, if $y_i$ and $y_{i+1}$ have opposite orientation, that is $b_i b_{i+1}=-1$, the equation for  $\dot y_i$ contains the term $-1/(y_i-y_{i+1})>0$ and  the equation for  $\dot y_{i+1}$ contains the term $-1/(y_{i+1}-y_{i})<0$. Since $y_i(0)<y_{i+1}(0)$,   the two points may collide in finite time.  Points with same orientation repel each other, thus never collide. System \eqref{DDD}  can be extended after collision by removing the particles that annihilate at collision,
see    \cite{pv3, mmp}. In the physical model, the ODE system \eqref{DDD} represents the discrete dynamics of $N$  dislocation  points with possibly different orientation.

In the present we want to identify at {\em large (macroscopic) scale} the evolution model for
the dynamics of a density of dislocations. We introduce a further parameter $\ep$ and consider a number of dislocations $N=N_\ep$ such that
$N_\ep\to+\infty$ as $\ep\to0$ and we send both $\delta$ and $\ep$ to 0 together. We do not specify  how $N_\ep$ goes to 0 with $\ep$ but we only require that 
$$\ep^2 N_\ep\delta\to0$$ as $\ep\to0$. 
We consider the
following rescaling
$$u^\ep(t,x)=\ep u\left(\frac{t}{\ep \delta^2},\frac{x}{\ep\delta}\right),$$
with $u$  the solution of \eqref{nabarroevolutintro}-\eqref{nabarroevolutintroinitial}. 
Then we see that $u^\ep$ is solution of \eqref{uepeq} with initial datum 
\beq\label{u0introphi}u^\ep(0,x)=\sum_{i=1}^{N_\ep}\ep \phi\left(\frac{x-x_i^0}{\ep\delta},b_i\right),\eeq
with $x_i^0=\ep y_i^0$.

In general, we consider an initial datum $u_0$ satisfying \eqref{u^0ass}.
One can actually prove (see Proposition \ref{approxpropfinal}) that any function satisfying \eqref{u^0ass},  can be approximated by a function of the form  \eqref{u0introphi}.

The parameter $\ep$ can be interpreted as the ratio between the microscopic scale and the macroscopic scale. The parameter $\delta$ instead which concerns  the passage from the microscopic scale to the mesoscopic scale (model \eqref{DDD}) is related to the density of dislocations being small at microscopic scale. 

\subsection{Organization of the paper}
The paper is organized as follows. In Section \ref{notationsec} we introduce notations and recall some general auxiliary results that  will be used in the  paper.
 The strategy and the heuristic of the proof of Theorem \ref{mainthm} are presented in Section \ref{strategysec}.  In Section 4 we prove a discrete approximation formula of the fractional Laplacian $\I$ which extends to non monotonic functions the one given in \cite{patsan}. 
In Section \ref{supersolutionssec} we construct local in time and global in space supersolutions of \eqref{uepeq}.
 Sections \ref{convergence}  and \ref{Initialconditionsection} are devoted to the proof of our main result, Theorem \ref{mainthm}.
Proposition \ref{uasymptoticprop} is proven in Section \ref{additional}.  Finally, the proofs of some auxiliary lemmas are given in Section \ref{lemmatasec}.
\section{Definitions, Notations and preliminary results}\label{notationsec}

\subsection{Definitions and Notations} 
Let $v$  be a function satisfying the following assumptions
\beq\label{vass}
\begin{cases}
v\in C^{1,1}(\mathbb{R}),\\
v \text{ not constant},\\
\lim \limits_{x\to-\infty}v(x)=0,\\
\lim \limits_{x\to+\infty}v(x)=l, \text{ for some }l\in\R.\\
\end{cases}
\eeq
For a fixed $\ep \in (0,1)$, we define
 $$\Lambda_i := \{x \, |\, i\ep< v(x) <  (i+1)\ep\},\quad i=s_\ep,\ldots, S_\ep,$$
 where $s_\ep:=\left \lceil\frac{\inf_\R v}{\ep} \right \rceil$ and $S_\ep:=\left \lfloor\frac{\sup_\R v}{\ep} \right \rfloor.$
By the  limits in \eqref{vass}, the set $\Lambda_i$ may contain two connected components of the form $(-\infty, a)$ and $(c,+\infty)$ for some $a,\,c\in\R$. 
We denote by $\tilde\Lambda_i$ the subset  of $\Lambda_i $ obtained by eventually removing those two connected components and all  connected components in which the oscillation of $v$ is smaller than $\ep$. 
Then, there exists a compact set $[-K_\ep,K_\ep]$ such that $\tilde\Lambda_i\subseteq [-K_\ep,K_\ep]$ for all $i=s_\ep,\ldots, S_\ep$.
 Moreover, any connected component of $\tilde\Lambda_i $ has measure bigger or equal than  $\ep/L$, where $L$ is the Lipschitz constant of $v$. Indeed,  if $A$ is any connected component of $\tilde\Lambda_i $, then there exists a point $x_0\in A$ such that $v(x_0)=\ep(i+1/2)$ and by the regularity of $v$, the interval $(x_0-\ep/(2L),x_0+\ep/(2L))$ 
  is contained in $A$. We infer that the number of connected components of $\tilde\Lambda_i$ is finite. In particular, the set $\bigcup_i \partial \tilde\Lambda_i $ has a finite number of points,  that is
\beq \label{xi0}
\displaystyle\bigcup_{i=s_\ep}^{S_\ep} \partial \tilde\Lambda_i = \{x_1,x_2,...,x_{N_\ep}\},
\eeq
for some positive integer $N_\ep$ depending on $\ep$, where the points $x_i$ are ordered such that $x_1<x_2<...<x_{N_\ep}$.
 Define $x_0=-\infty$, $v(x_0) = 0$ and for each $i\in \{1,2,...,N_\ep\}$,
\beq \label{bidef0}
b_i := \dfrac{v(x_i) - v(x_{i-1})}{\ep}\in \{-1,0,1\}.
\eeq
Let  $\mathcal{P}$ be  the subset of $\bigcup_{i=s_\ep}^{S_\ep} \partial \tilde\Lambda_i $ obtained by removing  the points $x_i$ for which $b_i=0$. By relabeling 
the points in 
$\mathcal{P}$,  we can write   
\beq \label{xi}
\mathcal{P}=\displaystyle \{x_1,x_2,...,x_{N_\ep}\}.
\eeq
By \eqref{bidef0}, for $x_i$ defined as in \eqref{xi}, we have 
\beq \label{bidef}
b_i := \dfrac{v(x_i) - v(x_{i-1})}{\ep}\in \{-1,1\},
\eeq
which also gives the following expression for $v(x_i)$
\beq \label{vxidef}
v(x_i) = v(x_{i-1}) + b_i\ep = \sum_{j=1}^i b_j \ep.
\eeq
We will sometimes refer to the level set points defined in \eqref{xi} as {\em particles}. Notice that  by removing the level set points corresponding to $b_i=0$ the oscillation of $v$ in the interval $[x_i,x_{i+1}]$ could be greater than $\ep$ but it is  always less 
than   $2\ep$, that is, 
\beq\label{oscvxixi+1}|v(x)-v(y)|\leq 2\ep,\quad\text{for all }x,\,y\in [x_i,x_{i+1}].\eeq

For any $x\in(x_1,x_{N_\ep})$ we will call the {\em closest particle to $x$} the biggest  point $x_{i_0}$ such that
$|x-x_{i_0}|\leq |x-x_i|$ for all $i=1,\ldots, N_\ep$. 

Given integers $M$ and $N$ such that $1\leq M\le N\le N_\ep$, we denote the number of particles with $b_i = 1$ and particles with $b_i = -1$ in $[x_M,x_N]$ by $n^+_{M,N}$ and $n^-_{M,N}$ respectively. Precisely, 
\begin{eqnarray*}
n^+_{M,N}:=| \{i \in \{M,...,N\} \ | \ b_i = 1\}|,\\
 n^-_{M,N}:= |\{i \in \{M,...,N\} \ | \ b_i = -1\}|.
 \end{eqnarray*}
Also, we define
\beq \label{nmn} n_{M,N}:=n^+_{M,N}-n^-_{M,N}.\eeq
When $M=1$ and $N=N_\ep$, we denote 
\beqs N^+_\ep:=n^+_{1,N_\ep},\quad N^-_\ep:=n^-_{1,N_\ep}.\eeqs
Note that $N_\ep= N^+_\ep+N^-_\ep$.
\begin{rem}
Using \eqref{vxidef}, \eqref{nmn} can also be expressed by
\beq\label{stronzo}
\ep n_{M,N} = \ep n^+_{M,N}-\ep n^-_{M,N}= \ep \sum_{\stackrel{i=M}{b_i = 1}}^N b_i + \ep  \sum_{\stackrel{i=M}{b_i = -1}}^N b_i = \sum_{i=M}^N b_i \ep = v(x_N) - v(x_M) + b_M\ep.
\eeq
In particular, for $M=1$, we have $v(x_1) = b_1\ep$, which yields
\beqs \label{n1m}
\ep n_{1,N} = v(x_N).
\eeqs
\end{rem}

 Similarly to Definition \eqref{phibi}, for the Heaviside function $H$, we define for any $z\in \mathbb{R}$ and $b\in\{-1,1\}$, 
\beq \label{hbi}
H(z,b) :=b H(z). \eeq

To construct sub and supersolution of \eqref{uepeq} we will often make use of   the following ODE's system
\beq \label{levelsetode}
\begin{cases}
\dot{x}_i(t) &= -c_0b_iL,\\
x_i(0) &= x_i^0,
\end{cases}
\eeq
where $x_1^0,x_2^0,...,x_{N_\ep}^0$ are the level set points of the initial condition, $L\in\R$ and $c_0$ is given by \eqref{c0}.

We denote by $B_r(x)$ the ball of radius
$r$ centered at $x$. The cylinder $(t-\tau,t+\tau)\times B_r(x)$
is denoted by $Q_{\tau,r}(t,x)$.
$\lfloor x \rfloor$ and $\lceil x\rceil$ denote respectively the
floor and the ceil integer parts of a real number $x$.

For $r>0$, we denote 
\beq\label{i1v}\I^{1,r}[v](x)=\frac{1}{\pi}PV
\displaystyle\int_{|y-x|\leq r}\displaystyle\frac{v(y)-v(x)}{(y-x)^{2}}dy,\eeq
and 
\beq\label{i2v}\I^{2,r}[v](x)=\frac{1}{\pi}
\displaystyle\int_{|y-x|>r}\displaystyle\frac{v(y)-v(x)}{(y-x)^{2}}dy.\eeq
Then we can write
$$\I[v](x)=\I^{1,r}[v](x)+\I^{2,r}[v](x).$$
Let $I\subset [0,+\infty)$ be an interval.  We denote by $USC_b(I\times\R)$ (resp.,
$LSC_b(I\times\R)$) the set of upper (resp., lower)
semicontinuous functions on $I\times\R$ which are bounded on
$([0,T]\cap I)\times\R$ for any $T>0$ and we set
$C_b(I\times\R):=USC_b(I\times\R)\cap
LSC_b(I\times\R)$.
We denote by $C_b^2([0,+\infty)\times\R)$ the subset of functions of $C_b([0,+\infty)\times\R)$ with continuous second derivatives.
Finally,  $C^{1,1}(\R)$ is the set of functions with bounded $C^{1,1}$ norm over~$\R$.

Given a sequence $\{u^\ep\}$ we denote
$${\limsup_{\ep\to0}}^*u^\ep(t,x)=\sup\Big\{\limsup_{\ep\to0} u^\ep(t_\ep,x_\ep)\,|\,t_\ep \to t\text{ and }x_\ep\to x\Big\},$$
and 
$${\liminf_{\ep\to0}}_*u^\ep(t,x)=\inf\Big\{\liminf_{\ep\to0} u^\ep(t_\ep,x_\ep)\,|\,t_\ep \to t\text{ and }x_\ep\to x\Big\}.$$

Given a quantity   $E=E(x)$, we write  $E=O(A)$ is there exists a constant $C>0$ such that, for all $x$, 
$$|E|\le C A.$$
We write $E=o_\ep(1)$ if 
$$\lim_{\ep \to0} E=0,$$
uniformly in $x$. 

\subsection{Definition of $A_\ep$}\label{Aepsec}
Since $u_0$ satisfies \eqref{u^0ass}, it is easy to see that there exist $C^{1,1}$ functions $v_1$ and $w_1$ such that
\beq\label{u_01} \begin{cases}v_1\le u_0,\quad v_1(-\infty)=0,\quad v_1\text{ is non-increasing } \\
w_1\le u_0,\quad w_1(+\infty)=l,\quad w_1\text{ is non-decreasing},
\end{cases}
\eeq
and there exist $C^{1,1}$ functions $v_2$ and $w_2$ such that
\beq\label{u_02} \begin{cases}v_2\ge u_0,\quad v_2(-\infty)=0,\quad v_2\text{ is non-decreasing} \\
w_2\ge u_0,\quad w_2(+\infty)=l,\quad w_2\text{ is non-increasing}.
\end{cases}
\eeq
Let $K_\ep>0$ be such that for $i=1,2$, 
\beqs |v_i(x)|<\frac{\ep}{4}\text{ if }x< -K_\ep\quad \text{and}\quad |w_i(x)-l|<\frac{\ep}{4}\text{ if }x> K_\ep.\eeqs 
Then, all the points in the $\ep$ level sets of  $u_0$ defined as in \eqref{xi}  must belong to the compact set $[-K_\ep,K_\ep]$ and by the forthcoming formula \eqref{proplippart1},  if $N^0_\ep$ is the number of such points, then $N^0_\ep\le CK_\ep/\ep$. 
Set 
$$A_\ep:=\ep K_\ep.$$ 
Then we choose  $\delta=o_\ep(1)$ such that 
\beq\label{Aepo0bound}\delta A_\ep=o_\ep(1).\eeq The condition guarantees that 
\beq\label{Nepo0bound}\ep^2N^0_\ep\delta=o_\ep(1).\eeq
Notice that if $u_0$ is monotonic then $N^0_\ep\leq (\sup u_0-\inf u_0)/\ep$, therefore \eqref{Nepo0bound} is always satisfied and no condition on how $\delta$ goes to 0 as $\ep\to0$ is required. 
It is easy to see that \eqref{Aepo0bound} holds true under no condition on $\delta$ also if  $u_0$ satisfies the following asymptotic estimate
$$ |u_0(x)-l H(x)|\leq \frac{C}{x}\quad \text {if }|x|>1,$$for some $C>0$, with $H$ the Heaviside function. 
\subsection{Short and long range interaction}
We start by recalling a basic fact about the operator $\I$.
Given $v\in C^{1,1}(\R)$ and $r>0$ we can split $\I[v]$ into the short and long range interaction as follows,
$$\I[v](x)=\I^{1,r}[v](x)+\I^{2,r}[v](x),$$ where $\I^{1,r}[v](x),\,\I^{2,r}[v](x)$ are defined respectively by \eqref{i1v} and \eqref{i2v}.
The short range interaction can be rewritten as 
$$ \I^{1,r}[v](x)=\frac{1}{2\pi}\int_{|y|<r}\dfrac{v(x+y)+v(x-y)-2v(x)}{y^2}dy,$$ 
Therefore, 
\beq\label{I11boundprelim}| \I^{1,r}[v](x)|\leq \frac{r}{\pi}\|v\|_{C^{1,1}(\R)}.\eeq
The long range interaction can be bounded as follows
$$| \I^{2,r}[v](x)|\leq \frac{4}{r\pi}\|v\|_{\infty}.$$
In particular, choosing $r=2$, we see that 
\beq\label{boundfractionalpreliminaries}
|\I[v](x)|\leq \frac{4}{\pi}\|v\|_{C^{1,1}(\R)}. 
\eeq
\subsection{The functions $\phi$ and $\psi$}

 In what follows we denote by  $H(x)$ the Heaviside function. Let $\alpha:=W''(0)>0$. 
\begin{lem}
\label{phiinfinitylem}Assume that  \eqref{Wass} holds, then there exists a unique solution $\phi$ of \eqref{phi}. Furthermore $\phi\in C^{2,\beta}(\R)$ and  there exist constants $K_0,K_1 >0$ such that
\begin{equation}\label{phiinfinity}\left|\phi(z)-H(z)+\frac{1}{\al \pi
z}\right|\leq \frac{K_1}{z^2},\quad\text{for }|z|\geq 1,
\end{equation}and for any $z\in\R$
\begin{equation}\label{phi'infinity}0<\frac{K_0}{1+z^2}\leq
\phi'(z)\leq\frac{K_1}{1+z^2}.\end{equation}
\end{lem}
\begin{proof} The existence of a unique solution of \eqref{phi} and estimate \eqref{phi'infinity} are proven in \cite{csm,psv}. Estimate \eqref{phiinfinity} is proven in \cite{gonzalezmonneau}.  
\end{proof}
Let $c_0$ be defined as in \eqref{c0}. 
Let us introduce the function  $\psi$ to be the solution of
\begin{equation}\label{psi}
\begin{cases}\I[\psi]=W''(\phi)\psi+\frac{L}{\alpha}(W''(\phi)-W''(0))+c_0L\phi'&\text{in}\quad \R\\
\lim_{z\rightarrow{\pm}\infty}\psi(z)=0.
\end{cases}
\end{equation} 
For later purposes, we recall the following decay estimate
on the solution of~\eqref{psi}:
\begin{lem}
\label{psiinfinitylem}
Assume that  \eqref{Wass} holds, then there exists a unique solution $\psi$ to \eqref{psi}. Furthermore 
$\psi\in C^{1,\beta}(\R)$ 
and  for any $L\in\R$  there exist constants $K_2$
and $K_3$, with $K_3>0$, depending on $L$ such that
\begin{equation}\label{psiinfinity}\left|\psi(z)-\frac{K_2}{
z}\right|\leq\frac{K_3}{z^2},\quad\text{for }|z|\geq 1,
\end{equation}and for any $z\in\R$
\begin{equation}\label{psi'infinity}-\frac{K_3}{1+z^2}\leq
\psi'(z)\leq \frac{K_3}{1+z^2}.
\end{equation}
\end{lem}
\begin{proof}
The existence of a unique solution of  \eqref{psi} is proven in  \cite{gonzalezmonneau}. Estimates \eqref{psiinfinity}  and \eqref{psi'infinity} are shown in \cite{mp2}.
\end{proof}
The results of Lemmas \ref{phiinfinitylem} and  \ref{psiinfinitylem} have been generalized  in \cite{cs, dpv,dfv,psv,pv} to the case 
when the fractional operator is $-(-\Delta)^s$ for any $s\in(0,1)$.

For $\psi$ solution of ~\eqref{psi} and $b\in\{-1,1\}$, $z\in\R$,  we define
\beq\label{psibi}
\psi(z,b):=\psi(bz)\eeq
\subsection{Definition of viscosity solution}\label{viscositysec}
We first recall the definition of viscosity solution for a general
first order non-local equation 
\beq\label{generalpbbdd}
\partial_tu=F(t,x,u,\partial_xu,\I[u])\quad\text{in}\quad (0,+\infty)\times\Om
\eeq
where $\Om$ is an open subset of $\R$ and $F(t,x,u,p,L)$ is continuous and 
non-decreasing in $L$. 
\begin{de}\label{defviscositybdd}A function $u\in USC_b((0,+\infty)\times\R)$ (resp., $u\in LSC_b((0,+\infty)\times\R)$) is a
viscosity subsolution (resp., supersolution) of
\eqref{generalpbbdd}  if for any $(t_0,x_0)\in(0,+\infty)\times\Om$, 
and any test function $\varphi\in
C_b^2((0,+\infty)\times\R)$ such that $u-\varphi$ attains a global maximum
(resp., minimum) at the point $(t_0,x_0)$, 
then 
\beqs\begin{split}&\p_t\varphi(t_0,x_0)-F(t_0,x_0,u(t_0,x_0),\partial_x\varphi(t_0,x_0),\I[\varphi(t_0,\cdot)](x_0))\leq0\\&\text{(resp., }\geq 0).\end{split}\eeqs  
A function $u\in
C_b((0,+\infty)\times\R)$ is a viscosity solution of
\eqref{generalpb} if it is a viscosity sub and supersolution
of \eqref{generalpbbdd}.
\end{de}
\begin{rem}
It is classical that 
the maximum (resp., the minimum) in Definition \ref{defviscositybdd}   can be assumed to be strict  and that
$$\varphi(t_0,x_0)=u(t_0,x_0).$$
This will be used later.
\end{rem}
Next, let us  consider the initial value problem
\begin{equation}\label{generalpb}
\begin{cases}
\partial_tu=F(t,x,u,\partial_xu,\I[u])&\text{in}\quad (0,+\infty)\times\R\\
u(0,x)=u_0(x)& \text{on}\quad \R,
\end{cases}
\end{equation} where  $u_0$  is a continuous function.

\begin{de}\label{defviscosity}A function $u\in USC_b([0,+\infty)\times\R)$ (resp., $u\in LSC_b([0,+\infty)\times\R)$) is a
viscosity subsolution (resp., supersolution) of the initial value problem 
\eqref{generalpb} if $u(0,x)\leq (u_0)(x)$ (resp., $u(0,x)\geq
(u_0)(x)$) and $u$ is viscosity subsolution (resp., supersolution)  of the equation 
$$\partial_tu=F(t,x,u,\partial_xu,\I[u])\quad\text{in}\quad (0,+\infty)\times\R.$$
 A function $u\in
C_b([0,+\infty)\times\R)$ is a viscosity solution of
\eqref{generalpb} if it is a viscosity sub and supersolution
of \eqref{generalpb}.
\end{de}
It is a  known  result that smooth solutions are also viscosity solutions. We provide here the proof for completeness. 
\begin{prop}
If $u\in C^1((0,+\infty);C^{1,\beta}_{loc}(\Om)\cap L^\infty(\R))$ for some $0<\beta\le 1$,  and $u$ satisfies pointwise 
\begin{equation}\label{eqlemmasmoothviscosity}\partial_tu- F(t,x,u,\partial_xu,\I[u])\leq 0 \text{ (resp., $\geq0$)}\quad\text{in}\quad (0,+\infty)\times\Om,\end{equation}then 
$u$ is a viscosity subsolution (resp., supersolution) of \eqref{generalpbbdd}.
\end{prop}
\begin{proof} Assume that $u$ satisfies \eqref{eqlemmasmoothviscosity}. Fix  any $(t_0,x_0)\in(0,+\infty)\times\Om$, 
and let  $\varphi\in
C_b^2((0,+\infty)\times\R)$ be such that $u-\varphi$ attains a global maximum
at the point $(t_0,x_0)$. 
Then  $\partial_t u(t_0,x_0)= \partial_t \varphi(t_0,x_0)$, $\partial_x u(t_0,x_0)= \partial_x \varphi(t_0,x_0)$ and since $u(t_0,x)-u(t_0,x_0)\leq \varphi(t_0,x)-\varphi(t_0,x_0)$, we have that 
$$\I[u(t_0,\cdot)](x_0)]\leq \I[\varphi(t_0,\cdot)](x_0)].$$
Therefore, using that $F(t,x, u,p,\cdot)$ is non-decreasing, we get 
\begin{equation*}\begin{split}&\partial_t\varphi(t_0,x_0)-F(t_0,x_0, u(t_0,x_0), \partial_x\varphi(t_0,x_0), \I[\varphi(t_0,\cdot)](x_0)])\\&\leq 
\partial_t u(t_0,x_0)-F(t_0,x_0, u(t_0,x_0), \partial_x u(t_0,x_0), \I[u(t_0,\cdot)](x_0)])\\&\leq 0,\end{split}\end{equation*}
as desired. 
\end{proof}

\subsection{Comparison principle and existence results}
In this subsection, we successively give com\-pa\-ri\-son
principles and existence results for \eqref{uepeq} and
\eqref{ubareq}. The following comparison theorem is shown in
\cite[Theorem 3.1]{jk} for more general parabolic integro-PDEs.
\begin{prop}[Comparison Principle for \eqref{uepeq}]\label{comparisonuep} Consider
 $u\in USC_b([0,+\infty)\times\R)$ subsolution
and $v\in LSC_b([0,+\infty)\times\R)$ supersolution of \eqref{uepeq},
then $u\leq v$ on $[0,+\infty)\times\R$.
\end{prop}
Following \cite{jk} it can  also be proven the comparison
principle for \eqref{uepeq} in bounded domains. Since we deal with a
non-local equation, we need to compare the sub and the
supersolution everywhere outside the domain.
\begin{prop}[Comparison Principle on bounded domains for
\eqref{uepeq}]\label{comparisonbounded} Let $\Om$ be an open interval 
 of $\R$ and let $u\in USC_b([t_1,t_2]\times\R)$
and $v\in LSC_b([t_1,t_2]\times\R)$ be respectively a sub and a
supersolution of $$\delta\p_{t}
u=\I[u(t,\cdot)]-\frac{1}{\delta}W'\left(\frac{u}{\epsilon}\right)\quad\text{in } (t_1,t_2)\times \Om,
$$
for some $0\leq t_1<t_2$. If
$u\leq v$ on  $(\{t_1\}\times \Om)\cup([t_1,t_2]\times (\R\setminus\Om))$, then $u\leq v$ in $[t_1,t_2]\times \R$. 
\end{prop}

\begin{prop}[Existence for \eqref{uepeq}]\label{existuep}For $\ep,\,\delta>0$ there exists
$u^{\ep}\in C_b([0,+\infty)\times\R)$ (unique) viscosity solution of
\eqref{uepeq}. 
\end{prop}
\dim 
We can construct a
solution by Perron's method if we construct sub and supersolutions
of \eqref{uepeq} which are equal to $u_0(x)$ at $t= 0$. Since $u_0\in C^{1,1}(\R)$, by \eqref{boundfractionalpreliminaries} the two functions
$u^{\pm}(t,x):=u_0(x){\pm} C t$ are respectively a super and a
subsolution of \eqref{uepeq}, if
$$C\geq \frac{4}{\pi\delta}\|u_0\|_{C^{1,1}(\R)}+\frac{1}{\delta^2}\|W'\|_\infty.$$
Moreover $u^+(0,x)=u^-(0,x)=u_0(x)$.
\finedim We next recall the
comparison and the existence results for \eqref{ubareq}, see e.g. \cite{imr}, Proposition 3.
\begin{prop}\label{existHeff}
If $u\in USC_b([0,+\infty)\times\R)$ and  $v\in LSC_b([0,+\infty)\times\R)$
are respectively a sub and a supersolution of  \eqref{ubareq}, 
then $u\leq v$ on $(0,+\infty)\times\R$. Moreover, under assumption \eqref{u^0ass},  there exists a
(unique) viscosity solution of \eqref{ubareq}.
\end{prop}

\section{Strategy and heuristic  proofs}\label{strategysec}
In this section we  explain the steps that we will follow  to prove Theorem \ref{mainthm}
 and the heuristics of the main proofs. 
 \subsection{Approximation of $\I$}
 The first result is a discrete  approximation formula for the fractional Laplace $\I$. 
Let $v$ be any function satisfying \eqref{vass}. Let $x_i$ and $b_i\in\{-1,1\}$, $i=1,\ldots, N_\ep$,  be defined as in \eqref{xi} and \eqref{bidef} respectively. 
Then,   we show (see Proposition \ref{apprIcor} and Proposition \ref{apprIcorall}) that for any fixed $i_0\in \{1,\ldots, N_\ep\}$, 
\beq\label{heuristicI1xi0}\I[v](x_{i_0})\simeq\frac{1}{\pi}\sum_{i\not=i_0} \dfrac{b_i\ep }{x_i-x_{i_0}},\eeq
 and 
 for any $x$,  
$$ \I[v](x)\simeq \frac{1}{\pi}\sum_{\stackrel{}{|x_i-x|> r}} \dfrac{b_i\ep}{x_i-x}, $$ for some $r=o_\ep(1)$, 
where  the error goes to 0 as $\ep\to0$  uniformly over $\R$. On the other hand,  the sum 
$$\sum_{\stackrel{i\neq i_0}{|x_i-x|\leq r}} \dfrac{b_i\ep}{x_i-x},$$ where $x_{i_0}$ is the closest particle to $x$ may not be zero but depends on   the distance of $x$ from 
$x_{i_0}$ (see Lemma \ref{lemmaerrorshortdistance}).
For the proof of these results we follow the proof of the analogous results given in \cite{patsan} in the case $v$ is monotone non-decreasing, i.e. $b_i=1$ for all $i$.
We refer to Section 2.1 there for the heuristic of \eqref{heuristicI1xi0} in the monotone case.

\subsection{Approximation of $v$}\label{heurconsub}
Let $\phi(x,b_i)$ be defined as in \eqref{phibi}. Then, we show (see Proposition \ref{approxpropfinal}) that any function $v$  satisfying \eqref{vass} can be approximated 
in $L^\infty(\R)$ as follows,
 \beq\label{vappintr}v(x)\simeq \sum_{i=1}^{N_\ep}\ep\phi\left(\frac{x-x_i}{\ep\delta},b_i\right),\eeq
with $x_i$ and $b_i\in\{-1,1\}$, $i=1,\ldots, N_\ep$,  defined as in \eqref{xi} and \eqref{bidef} respectively.
We refer to Section 2.2 in \cite{patsan} for the heuristic proof of \eqref{vappintr} in the case $v$ monotone non-decreasing.
\subsection{Heuristic of the proof of Theorem \ref{mainthm}}\label{heuristicsec} 
Let $u$ be the limit solution (that here we suppose  to exist and  be smooth). Fix a point $(t_0,x_0)\in(0,+\infty)\times\R$.
We need to distinguish two cases: $\partial_x u(t_0,x_0)\neq 0$ and $\partial_x u(t_0,x_0)=0$.

\medskip
\noindent{\bf Case 1: $\partial_x u(t_0,x_0)\neq 0$.}

We are going to give  an ansatz for $u^\ep$ in a small box $Q_R$ of size $R$ centered 
at $(t_0,x_0)$.  
 For small $R$,  all the derivatives of $u$ can be considered constant in $Q_R$:
$$\partial_t u(t,x)\simeq \partial_t u(t_0,x_0),\quad \partial_x u(t,x)\simeq \partial_x u(t_0,x_0)$$
and 
$$\I[u(t,\cdot)](x)\simeq \I[u(t_0,\cdot)](x_0)=: L_0.$$
For $t$ close to $t_0$, we define the points $x_i(t)$  as in \eqref{xi} and for $v=u(t,\cdot)$. Since $u$ is monotone in $Q_R$, the $b_i$ of the particles inside that box,  defined as in \eqref{bidef},  
have all the same value.   Moreover, for those points, by differentiating in $t$ the equation
\beq\label{xiintrou}u(t,x_i(t))=const.,\eeq
we get
\beqs \partial_t u(t,x_i(t))+\partial_xu(t,x_i(t))\dot{x}_i(t)=0,\eeqs
from which
\beq\label{xidotintro}\dot{x}_i(t)=-\frac{ \partial_t u(t,x_i(t))}{\partial_xu(t,x_i(t))}\simeq -\frac{ \partial_t u(t_0,x_0)}{\partial_xu(t_0,x_0)}. \eeq
Notice that since particles in the box have the same speed, they never collide there. 
Next we consider as   ansatz for  $u^\ep$ the approximation of $u$ given by  \eqref{vappintr} plus a small correction:
\beqs \Phi^\ep(t,x):= \sum_{i=1}^{N_\ep}\ep\left(\phi\left(\frac{x-x_i(t)}{\ep\delta},b_i\right)+\delta\psi\left(\frac{x-x_i(t)}{\ep\delta},b_i\right)\right),\eeqs
where $\phi(\cdot,b_i)$ is defined as in \eqref{phibi} and  $\psi(\cdot,b_i)$ as  in \eqref{psibi}, with $\psi$ the solution of     \eqref{psi} with $L=L_0$.
 For a detailed heuristic motivation of
this correction,
see Section 3.1 of \cite{gonzalezmonneau}.
By \eqref{vappintr}, $\Phi^\ep(t,x)\to u(t,x)$ as $\ep\to 0$.  Fix $(t,x)\in Q_R$ and let $x_{i_0}(t)$ be the closest point among  the $x_i(t)$'s to $x$ and $z_i(t)=(x-x_i(t))/(\ep\delta)$. Plugging into \eqref{uepeq}, we get  (see proof of \eqref{mainlem2} in Section \ref{convergence})
\beqs\begin{split} 0&=\delta\partial_t \Phi^\ep(t,x)-\I [\Phi^\ep(t,\cdot)](x)+\frac{1}{\delta}W'\left(\frac{\Phi_\ep(t,x)}{\ep}\right)
\\&\simeq 
-\phi'(z_{i_0})(b_{i_0}\dot{x}_{i_0}(t)+c_0L_0)+
(W''(\phi(z_{i_0})) - W''(0))\left ( \frac{1}{\delta}\sum_{\stackrel{}{i\not=i_0}}^{} \tilde{\phi}(z_i,b_i) -\frac{L_0}{\alpha}\right)
\end{split}
\eeqs
where $\tilde\phi(\cdot,b_i)=\phi(\cdot,b_i)-H(\cdot,b_i)$, with $H(\cdot,b_i)$  defined as in \eqref{hbi}. 
 Suppose for simplicity that 
$x=x_{i_0}(t)$, then 
by \eqref{phiinfinity} and \eqref{heuristicI1xi0}
\beqs \frac{1}{\delta}\sum_{\stackrel{}{i\not=i_0}}^{} \tilde{\phi}(z_i,b_i) -\frac{L_0}{\alpha}\simeq \frac{1}{\alpha\pi}\sum_{i\not=i_0} \dfrac{b_i\ep}{x_i-x_{i_0}}-\frac{L_0}{\alpha}\simeq 0. \eeqs
Since $\phi'>0$, we must have 
$$\dot{x}_{i_0}(t)\simeq-c_0b_{i_0}L_0 $$ that is, by \eqref{xidotintro},
$$ \partial_t u(t_0,x_0)\simeq c_0 b_{i_0} \partial_x u(t_0,x_0)\I[u(t_0,\cdot)](x_0)=c_0  |\partial_x u(t_0,x_0)|\I[u(t_0,\cdot)](x_0).$$
To formalize the argument we will construct from the ansatz local in space and time sub and supersolutions of \eqref{uepeq} to compare with $u^\ep$. 

Notice that if we define $$y_i(\tau):=\frac{x_i(\ep\tau)}{\ep}$$then the $y_i$'s  solve
\beqs \dot{y}_i(\tau)=\dot{x}_i(\ep\tau)\simeq -c_0b_i L_0 \simeq \frac{c_0}{\pi}\sum_{j\not=i} \dfrac{b_ib_j\ep}{x_i-x_{j}}=\frac{c_0}{\pi}\ \sum_{j\not=i} \dfrac{b_ib_j}{y_i-y_{j}},\eeqs
which is the discrete dislocations dynamics given in \eqref{DDD}.

\medskip
\noindent{\bf Case 2: $\partial_x u(t_0,x_0)=0$.}

When $\partial_x u(t_0,x_0)=0$, we cannot obtain formula \eqref{xidotintro}. However the ODE system \eqref{DDD} and the approximation formula \eqref{heuristicI1xi0} suggests that, at least locally, 
\beqs u^\ep(t,x)\sim \sum_{i=1}^{N_\ep}\ep\phi\left(\frac{x-x_i(t)}{\ep\delta},b_i\right),\eeqs
with $x_i(t)$ solution of 
$$\dot{x}_{i}(t)\simeq-c_0b_{i_0}\I[u(t_0,\cdot)](x_i(t)),$$ and $x_i(t_0)=x_i^0$, with $x_i^0$ the level set points of the function $u(t_0,\cdot)$.  Therefore, we proceed as follows.
Assume that in a box $Q_\rho$ around $(t_0,x_0)$ $u$ has the form 
\beq\label{uparaboladeintro}u(t,x)=a(x-x_0)^2+g(t)\eeq for some $a>0$ and $g$ smooth. 
Notice that level set points of $u$ in $Q_\rho$ which are  smaller than $x_0$ are associated to $b_i=-1$, and those  bigger than $x_0$ are associated to $b_i=1$.
Fix any $0<\sigma<<\rho$ independent of $\ep$. We construct a smooth approximation, $u^\sigma$,  of $u$ which is constant in $x$ for $|x-x_0|\leq \sigma$. 
  Precisely,  $u^\sigma$  is such that 
\beq\label{parabolabisintro}
\begin{cases}
u\le u^\sigma \\
u^\sigma \le u+C\sigma^2\text{ if }|x-x_0|\le \sigma\\
u^\sigma \text{ is constant in }x \text{ if }|x-x_0|\le \sigma\\
u^\sigma \text{ is non-increasing  in }x \text{ if }x\in (-\infty,x_0)\\
u^\sigma \text{ is non-decreasing  in }x \text{ if } x\in (x_0,+\infty).\\
\end{cases}\eeq
Next, we set  
 \beqs\label{therightcintro}c:= \frac{ 1}{4c_0L},\eeqs
with $L>0$ to be determined.
We then define $x_i^0$ and $b_i$, $i=1,\ldots, N_\ep$,  as in \eqref{xi} and  \eqref{bidef} for the function $u^\sigma(t_0-c\sigma,\cdot)$.
By \eqref{vappintr}, for all $x\in\R$, 
\beq\label{initialconscase2intro}u^\sigma(t_0-c\sigma,x)=  \sum_{i = 1}^{N_\ep}\ep \phi\left (\dfrac{x-x_i^0}{\ep \delta}, b_i\right ) +\ep M_\ep+o_\ep(1),\eeq
where the constant $M_\ep:=\lceil u^\sigma(t_0-c\sigma,-\infty)/\ep\rceil$ is a normalization so that $u^\sigma(t_0-c\sigma,-\infty)-\ep M_\ep=o_\ep(1)$. 
Define the function
\beq\label{phiintro} H^\ep(t,x):= \sum_{i=1}^{N_\ep}\ep\left(\phi\left(\frac{x-x_i(t)}{\ep\delta},b_i\right)+\delta\psi\left(\frac{x-x_i(t)}{\ep\delta},b_i\right)\right)+\ep M_\ep+\ep\left\lceil\frac{o_\ep(1)}{\ep}\right\rceil,\eeq
with $x_i(t)$ the solution of the ODE system  \eqref{levelsetode} with initial condition  $x_i(t_0-c\sigma)=x_i^0$,
that is 
$$x_i(t)=x_i^0-b_ic_0L[t - (t_0 - c\sigma)].$$
Since $ \sum_{i=1}^{N_\ep}\ep\delta\psi\left(\frac{x-x_i(t)}{\ep\delta},b_i\right)=o_\ep(1)$, from \eqref{initialconscase2intro} we can 
 choose $o_\ep(1)$ in \eqref{phiintro} in such a way 
$$u(t_0-c\sigma,x)\le u^\sigma(t_0-c\sigma,x)\leq H^\ep (t_0-c\sigma,x).$$
Notice that particles  $x_i(t)$ and $x_{i+1}(t)$ with the   same orientation ($b_i b_{i+1}=1$) move in parallel, while opposite oriented particles,  ($b_i b_{i+1}=-1$) move each toward the other. However, since $u^\sigma$ is constant in $x\in [x_0-\sigma,x_0+\sigma]$ and monotonic in $(-\infty,x_0)$ and in $(x_0,+\infty)$, 
 particles with opposite orientation are at distance larger than $2\sigma$ at time $t_0-c\sigma$. This guarantees that no collision occurs in the interval 
 $[t_0-c\sigma, t_0+c\sigma]$.  Then, we are able to show that setting 
 $$L:=\frac{C_0}{\sigma^\frac12}$$ for some  $C_0>0$ large enough but independent of $\ep$ and $\sigma$, $H^\ep$ is supersolution of \eqref{uepeq} in $[t_0-c\sigma, t_0+c\sigma]\times \R$,
and by the comparison principle,
\beqs\label{comparsionproofconograd}H^\ep(t,x) \geq u^\ep(t,x)\quad\text{for any }(t,x) \in [t_0-c\sigma, t_0+c\sigma] \times \R.\eeqs
This yields,  
 \beqs
\begin{split}
    u^\ep(t_0,x_0) 
    &\leq H^\ep(t_0,x_0)\\
      &= \sum_{i=1}^{N_\ep} \ep \phi \left ( \dfrac{x_0-x_i(t_0)}{\ep \delta}, b_i \right ) +\ep M_\ep+ o_\ep(1)\\
    &= \sum_{i=1}^{N_\ep} \ep \phi \left ( \dfrac{(x_0 +b_ic_0Lc\sigma)- x_i^0)}{\ep \delta}, b_i \right ) +\ep M_\ep+ o_\ep(1)\\
    &=  \sum_{i=1}^{N_\ep} \ep \phi \left ( \dfrac{(x_0 +c_0Lc\sigma)- x_i^0)}{\ep \delta}, b_i \right ) +\ep M_\ep+ o_\ep(1),
    \end{split}
\eeqs
    where the last equality needs to be justified (see  Lemma \ref{lastlem}).  Then, by \eqref{initialconscase2intro} and the second inequality in \eqref{parabolabisintro}, we obtain
     \beqs
\begin{split}
 u^\ep(t_0,x_0)&\leq   \sum_{i=1}^{N_\ep} \ep \phi \left ( \dfrac{(x_0 +c_0Lc\sigma)- x_i^0)}{\ep \delta}, b_i \right ) +\ep M_\ep+ o_\ep(1)\\
    &=u^\sigma(t_0-c\sigma, x_0 +c_0Lc\sigma) + o_\ep(1)\\
    &\leq u(t_0-c\sigma, x_0 +c_0Lc\sigma) + o_\ep(1)  + C\sigma^2.
\end{split}
\eeqs
 Passing to the limit as $\ep\to0$ and recalling  the definitions of $c$ and $L$ we get 
$$u(t_0,x_0)- u\left(t_0-k_0\sigma^\frac32,x_0+k_1\sigma\right)\leq C\sigma^2,$$
for some $k_0,\,k_1$ independent of $\sigma$. 
Dividing both sides by $k_0\sigma^\frac32$ and letting $\sigma\to0$, we finally obtain (recall  \eqref{uparaboladeintro})
$$\partial_t u(t_0,x_0)\leq0.$$ Similarly, one can prove that 
$\partial  u_t(t_0,x_0)\geq0.$

\subsection{Viscosity sub and supersolutions}
To formally prove the convergence result we show the functions
$u^+:={\limsup_{\ep\rightarrow0}}^*u^\epsilon$ and 
$u^-:={\liminf_{\ep\rightarrow0}}_*u^\epsilon$, which  are
 everywhere finite, are respectively sub and supersolution of 
\eqref{ubareq}. Moreover, $u^+(0,x)\leq u_0(x)\leq u^-(0,x)$. The comparison principle then implies that $u^+\leq \us\le u^-.$ Since the reverse 
 inequality $u^-\leq u^+$ always holds true, we
conclude that the two functions coincide with the continuous  viscosity  solution of \eqref{ubareq}. 

\section{Approximation Results}

In this section, we present several approximation results, which are similar to those in \cite{patsan}. In this paper, however, we consider a function $v$ satisfying \eqref{vass}, which is not necessarily monotonic. Since the proofs of some results are similar, they will be omitted. Readers may consult \cite{patsan} if necessary. 
The following lemma is proven in \cite{patsan}, see Lemma 4.1.
\begin{lem}\label{lemdistxi} Assume that $v$ satisfies \eqref{vass}. Let  $\|v_x\|_\infty\leq L$, and 
let $x_i$ be defined as in \eqref{xi}.
Then, 
\beq\label{proplippart1}x_{i+1}-x_i\geq \ep L^{-1}\quad\text{ for all }i=1,\ldots, N_\ep-1.\eeq
Moreover,  there exists $c>0$ independent of $v$ such that  for any $\xs\in \R$
\beq\label{i/k^2sum}\sum_{\stackrel{i=1}{i\neq i_0}}^{N_\ep}\frac{\ep^2}{(x_i-\xs)^2}\leq cL^2.\eeq
In addition, if $|v_x|\ge  a>0$ on an interval $I$, then for all $x_{i+1},\,x_i\in I$, we have
\beq\label{proplippart2}x_{i+1}-x_i\leq \ep a^{-1}.\eeq
\end{lem}


\begin{lem}[Short range interaction]\label{approxIshortlem}
Assume that $v$ satisfies \eqref{vass} and 
let $x_i$ and $b_i$ be defined as in \eqref{xi} and  \eqref{bidef}.
Let  $r=r_\ep=o_\ep(1)$ and $\ep/r=o_\ep(1)$. For any $\rho\ge r$ and $\xs\in (x_1+\rho,x_{N_\ep}-\rho)$, then
\beq\label{I1rhoapplemeq} \frac{1}{\pi}\sum_{\stackrel{i\neq i_0}{ r\le |x_i-\xs|\leq\rho}}\frac{b_i\ep}{x_i-\xs}=\I^{1,\rho}[v](\xs)+\frac{1}{\pi}\frac{v(\xs+\rho)+v(\xs-\rho)-2v(\xs)}{\rho}+
o_\ep(1).
\eeq
\end{lem}
\begin{proof}
Since $v\in C^{1,1}(\R)$ and $r=o_\ep(1)$, by \eqref{I11boundprelim} there exists $C>0$ such that 
$$|\I^{1,r}[v](\xs)|\leq Cr=o_\ep(1).$$
Therefore, we have 
\beq\label{approxshortlemprima}\I^{1,\rho}[v](\xs)=\frac{1}{\pi}\int_{\xs-\rho}^{\xs-r} \dfrac{v(x) - v(\xs)}{(x-\xs)^2} dx+
\frac{1}{\pi}\int_{\xs+r}^{\xs+\rho} \dfrac{v(x) - v(\xs)}{(x-\xs)^2} dx+o_\ep(1).
\eeq
We write the first term in \eqref{approxshortlemprima} as
\beqs \int_{\xs-\rho}^{\xs-r} \dfrac{v(x) - v(\xs)}{(x-\xs)^2} dx 
=\int_{\xs-\rho}^{\xs-r} \dfrac{v(x)}{(x-\xs)^2} dx - \int_{\xs-\rho}^{\xs-r} \dfrac{v(\xs)}{(x-\xs)^2} dx,
\eeqs
and notice that we can integrate the second term in \eqref{approxshortlemprima} as follows
\beq \label{ndeq}
\int_{\xs-\rho}^{\xs-r} \dfrac{v(\xs)}{(x-\xs)^2} dx = v(\xs) \int_{\xs-\rho}^{\xs-r} \dfrac{1}{(x-\xs)^2} dx =  \dfrac{v(\xs) }{r} - \dfrac{v(\xs) }{\rho}.
\eeq
Let us first assume that there are particles $x_i$ in the interval $[\xs-\rho,\xs-r]$. 
Let us denote by $M_\rho$ and $M_r$ respectively the smallest and the largest integer $i$ such that $x_i\in[\xs-\rho,\xs-r]$, that is
\beq \label{xmrho} x_{M_\rho-1}<\xs-\rho\leq x_{M_\rho}\leq  x_{M_r}\leq\xs-r<x_{M_r+1}.\eeq
Then, we have that
\beq \label{intsplit}
\int_{\xs-\rho}^{\xs-r} \dfrac{v(x)}{(x-\xs)^2} dx = \int_{\xs-\rho}^{x_{M_\rho}}\dfrac{v(x)}{(x-\xs)^2} dx + \sum_{i=M_\rho}^{M_r-1} \int_{x_i}^{x_{i+1}} \dfrac{v(x)}{(x-\xs)^2} dx + \int_{x_{M_r}}^{\xs-r} \dfrac{v(x) }{(x-\xs)^2} dx. 
\eeq
By \eqref{oscvxixi+1}, $v(x)\leq v(\xs-\rho) +2 \ep$ for $x\in [\xs-\rho, x_{M_\rho}]$. Hence, we obtain 
\beq \label{t1}
\int_{\xs-\rho}^{x_{M_\rho}}\dfrac{v(x)}{(x-\xs)^2} dx \leq \int_{\xs-\rho}^{x_{M_\rho}}\dfrac{v(\xs-\rho) + 2\ep}{(x-\xs)^2} dx = \dfrac{v(\xs-\rho) + 2\ep}{-\rho} - \dfrac{v(\xs-\rho) + 2\ep}{x_{M_\rho}-\xs}.
\eeq
Similarly, $v(x)\leq v(\xs-r) +2 \ep$ for $x\in [x_{M_r},\xs-r]$, which gives us 
\beq \label{t3}
\int_{x_{M_r}}^{\xs-r} \dfrac{v(x) }{(x-\xs)^2} dx \leq \int_{x_{M_r}}^{\xs-r} \dfrac{v(\xs-r) +2 \ep}{(x-\xs)^2} dx = \dfrac{v(\xs-r) +2 \ep}{x_{M_r}-\xs} - \dfrac{v(\xs-r) +2 \ep}{-r}.
\eeq
Also, for $x\in [x_i,x_{i+1}]$, we have $v(x)\leq v(x_i) + 2\ep$. Thus, we obtain 
\beq \label{t2} \begin{split}
   \sum_{i=M_\rho}^{M_r-1} \int_{x_i}^{x_{i+1}} \dfrac{v(x)}{(x-\xs)^2} dx &\leq \sum_{i=M_\rho}^{M_r-1} \int_{x_i}^{x_{i+1}} \dfrac{v(x_i) + 2\ep}{(x-\xs)^2} dx\\
   &= \sum_{i=M_\rho}^{M_r-1} \left [ \dfrac{v(x_i) + 2\ep}{x_i-\xs} -\dfrac{v(x_i) + 2\ep}{x_{i+1}-\xs} \right ]\\
   &= \sum_{i=M_\rho}^{M_r-1} \dfrac{v(x_i) + 2\ep}{x_i-\xs} - \sum_{i=M_\rho+1}^{M_r} \dfrac{v(x_{i-1}) +2 \ep}{x_i-\xs}\\
   &=-\dfrac{v(x_{M_r}) + 2\ep}{x_{M_r} - \xs} + \sum_{i=M_\rho}^{M_r} \dfrac{v(x_i) - v(x_{i-1})}{x_i-\xs} + \dfrac{v(x_{M_\rho-1}) + 2\ep}{x_{M_\rho}-\xs}\\
   &= \dfrac{v(x_{M_\rho-1}) +2 \ep}{x_{M_\rho}-\xs} -\dfrac{v(x_{M_r}) + 2\ep}{x_{M_r} - \xs} + \sum_{i=M_\rho}^{M_r} \dfrac{b_i\ep}{x_i-\xs},
\end{split} \eeq
using $v(x_i) = v(x_{i-1}) + b_i\ep$ in the last equality. Finally, we combine \eqref{ndeq}, \eqref{t1}, \eqref{t3} and \eqref{t2} to get
\beqs \begin{split}
\int_{\xs-\rho}^{\xs-r} \dfrac{v(x) - v(\xs)}{(x-\xs)^2} dx  &\leq \dfrac{v(\xs-\rho) + 2\ep}{-\rho} - \dfrac{v(\xs-\rho) + 2\ep}{x_{M_\rho}-\xs}\\
&+ \dfrac{v(\xs-r) + 2\ep}{x_{M_r}-\xs} - \dfrac{v(\xs-r) + 2\ep}{-r}\\
&+ \dfrac{v(x_{M_\rho-1}) + 2\ep}{x_{M_\rho}-\xs} -\dfrac{v(x_{M_r}) + 2\ep}{x_{M_r} - \xs} + \sum_{i=M_\rho}^{M_r} \dfrac{b_i\ep}{x_i-\xs} - \dfrac{v(\xs)}{r} + \frac{v(\xs)}{\rho},
\end{split}
\eeqs
which simplifies to 
\beqs \begin{split}
\int_{\xs-\rho}^{\xs-r} \dfrac{v(x) - v(\xs)}{(x-\xs)^2} dx  &\leq \sum_{i=M_\rho}^{M_r} \dfrac{b_i\ep}{x_i-\xs} + \dfrac{v(\xs)-v(\xs-\rho)}{\rho} - \dfrac{2\ep}{\rho} + \dfrac{v(\xs-r)-v(\xs)}{r} + \dfrac{2\ep}{r}\\
&+ \dfrac{v(x_{M_\rho-1})-v(\xs-\rho)}{x_{M_\rho}-\xs} + \dfrac{v(\xs-r) - v(x_{M_r})}{x_{M_r} - \xs}.
\end{split}
\eeqs
Note that, by \eqref{oscvxixi+1} and \eqref{xmrho}, $|v(x_{M_\rho-1})-v(\xs-\rho)|\leq 2\ep$, $|v(\xs-r) - v(x_{M_r})|\leq 2\ep$, $|x_{M_\rho}-\xs|\geq r$ and $|x_{M_r} - \xs| \geq r$. Hence, the following holds: 
\beq \label{est} 
\dfrac{v(x_{M_\rho-1})-v(\xs-\rho)}{x_{M_\rho}-\xs} \leq \dfrac{2\ep}{r} \quad \text{and} \quad
\dfrac{v(\xs-r) - v(x_{M_r})}{x_{M_r} - \xs} \leq \dfrac{2\ep}{r}. \eeq
Using \eqref{est}, we therefore obtain 
\beq \label{leftleqbound1}
\int_{\xs-\rho}^{\xs-r} \dfrac{v(x) - v(\xs)}{(x-\xs)^2} dx  \leq \sum_{i=M_\rho}^{M_r} \dfrac{b_i\ep}{x_i-\xs} + \dfrac{v(\xs)-v(\xs-\rho)}{\rho} + \dfrac{v(\xs-r)-v(\xs)}{r} + \dfrac{6\ep}{r}.
\eeq
If there are no particles $x_i$ in the interval $[\xs-\rho,\xs-r]$, then by \eqref{oscvxixi+1} $|v(x)-v(\xs-\rho)|\le2\ep$ and  $|v(x)-v(\xs-r)|\le2\ep$ for any $x\in[\xs-\rho,\xs-r]$. 
Therefore, we can simply estimate
\begin{equation*}\begin{split} \int_{\xs-\rho}^{\xs-r} \dfrac{v(x) }{(x-\xs)^2} dx&\leq (v(\xs-\rho)+2\ep)\int_{\xs-\rho}^{\xs-r} \dfrac{1}{(x-\xs)^2} dx=(v(\xs-\rho)+2\ep)\left(\frac1r-\frac1\rho\right)\\&
\leq\frac{v(\xs)}{r}-\frac{v(\xs-\rho)}{\rho}+\frac{4\ep}{r}, 
\end{split}\end{equation*}
which combined with \eqref{ndeq} gives again \eqref{leftleqbound1} with the summation on the right-hand side equal to 0. 

Similarly, for the second term in \eqref{approxshortlemprima}, one can show that 
\beq \label{rightleqbound2}
\int_{\xs+r}^{\xs+\rho} \dfrac{v(x) - v(\xs)}{(x-\xs)^2} dx  \leq \sum_{i=N_r}^{M_\rho} \dfrac{b_i\ep}{x_i-\xs} + \dfrac{v(\xs)-v(\xs+\rho)}{\rho} + \dfrac{v(\xs+r)-v(\xs)}{r} + \dfrac{6\ep}{r}.
\eeq
By \eqref{approxshortlemprima}, \eqref{leftleqbound1} and \eqref{rightleqbound2}, we conclude that 
\beq \label{315} \begin{split}
   \I^{1,\rho}[v](\xs) &\leq \frac{1}{\pi}\sum_{\stackrel{i\neq i_0}{ r\le |x_i-\xs|\leq\rho}}\frac{b_i\ep}{x_i-\xs}+ \dfrac{12\ep}{r} + o_\ep(1)\\
   &+ \dfrac{1}{\pi} \dfrac{v(\xs+r)+ v(\xs-r)-2v(\xs) }{r} - \dfrac{1}{\pi} \dfrac{v(\xs+\rho)  + v(\xs-\rho)-2v(\xs)}{\rho}. 
\end{split}
\eeq
Since $v\in C^{1,1}(\R)$, there exists a constant $C>0$ such that 
\beq \label{c11} \left|\dfrac{v(\xs+r) +v(\xs-r)- 2v(\xs)}{r}\right|\leq Cr=o_\ep(1).\eeq 
Therefore, using $\ep/r = o_\ep(1)$, we finally obtain
\beq \label{leqbound}
\I^{1,\rho}[v](\xs) \leq \frac{1}{\pi}\sum_{\stackrel{i\neq i_0}{ r\le |x_i-\xs|\leq\rho}}\frac{b_i\ep}{x_i-\xs} - \dfrac{1}{\pi} \dfrac{v(\xs+\rho) + v(\xs-\rho)-2v(\xs) }{\rho}  + o_\ep(1).
\eeq
For the lower bound, we apply the following inequalities to \eqref{intsplit}.
\beqs \begin{split}
    v(x) &\geq v(x_i) -2\ep \quad \quad  \text{for} \quad x\in [x_i,x_{i+1}], i = M_\rho,...,M_r-1.\\
    v(x) &\geq v(\xs-\rho) - 2\ep \quad \text{for} \quad x\in[\xs-\rho, x_{M\rho}]\\
    v(x) &\geq v(\xs-r) - 2\ep \quad \text{for} \quad x\in[x_{M_r}, \xs-r].
\end{split} \eeqs
Then, we follow the same steps as above to eventually obtain 
\beq \label{leftgeqbound1}
\int_{\xs-\rho}^{\xs-r} \dfrac{v(x) - v(\xs)}{(x-\xs)^2} dx  \geq \sum_{i=M_\rho}^{M_r} \dfrac{b_i\ep}{x_i-\xs} + \dfrac{v(\xs)-v(\xs-\rho)}{\rho} + \dfrac{v(\xs-r)-v(\xs)}{r} - \dfrac{6\ep}{r}.
\eeq
Similarly, one can show that 
\beq \label{rightgeqbound2}
\int_{\xs+r}^{\xs+\rho} \dfrac{v(x) - v(\xs)}{(x-\xs)^2} dx  \geq \sum_{i=N_r}^{M_\rho} \dfrac{b_i\ep}{x_i-\xs} + \dfrac{v(\xs)-v(\xs+\rho)}{\rho} + \dfrac{v(\xs+r)-v(\xs)}{r} - \dfrac{6\ep}{r}.
\eeq
By combining \eqref{leftgeqbound1} and \eqref{rightgeqbound2}, and using \eqref{c11}, we obtain 
\beq \label{geqbound}
\I^{1,\rho}[v](\xs) \geq \frac{1}{\pi}\sum_{\stackrel{i\neq i_0}{ r\le |x_i-\xs|\leq\rho}}\frac{b_i\ep}{x_i-\xs} - \dfrac{1}{\pi} \dfrac{v(\xs+\rho) + v(\xs-\rho)-2v(\xs) }{\rho}  + o_\ep(1).
\eeq
 By \eqref{leqbound} and \eqref{geqbound}, we have proven \eqref{I1rhoapplemeq}.
\end{proof}


\begin{lem}[Long range interaction]\label{approxIlonglem}
Under the assumptions of Lemma \ref{approxIshortlem} and for $r$ as in the lemma, 
 for any $\rho\ge r$ and $\xs\in(x_1+\rho,x_{N_\ep}-\rho)$,
\beq\label{approxIlonglemeq} \frac{1}{\pi}\sum_{|x_i-\xs|>\rho}\frac{b_i\ep}{x_i-\xs}=\I^{2,\rho}[v](\xs)
-\frac{1}{\pi}\frac{v(\xs+\rho)+v(\xs-\rho)-2v(\xs)}{\rho}+ o_\ep(1).
\eeq
\end{lem}
\begin{proof}
First, we consider the following decomposition of $\I^{2,\rho}[v](\xs)$.
\beq\label{middletwoterms} \I^{2,\rho}[v](\xs) = \int_{-\infty}^{x_{1}} \frac{v(x) - v(\xs)}{(x-\xs)^2} dx + \int_{x_{1}}^{\xs-\rho} \frac{v(x) - v(\xs)}{(x-\xs)^2} dx + \int_{\xs+\rho}^{x_{N_\ep}} \frac{v(x) - v(\xs)}{(x-\xs)^2} dx + \int_{x_{N_\ep}}^{+\infty} \frac{v(x) - v(\xs)}{(x-\xs)^2} dx.
\eeq
Define the following integrals.
\beqs
\begin{split}
    T_1 &:= \int_{-\infty}^{x_{1}} \frac{v(x) - v(\xs)}{(x-\xs)^2} dx \quad \quad
    T_2 := \int_{x_{1}}^{\xs-\rho} \frac{v(x) - v(\xs)}{(x-\xs)^2} dx\\
    T_3 &:= \int_{\xs+\rho}^{x_{N_\ep}} \frac{v(x) - v(\xs)}{(x-\xs)^2} dx \quad \quad
    T_4 := \int_{x_{N_\ep}}^{+\infty} \frac{v(x) - v(\xs)}{(x-\xs)^2} dx.
\end{split}
\eeqs
Proceeding as in the proof of Lemma \ref{approxIshortlem} with  $\xs - \rho, \xs-r, \xs+r, \xs+\rho$ replaced by  $x_1, \xs-\rho, \xs-\rho, x_{N_\ep}$ respectively in \eqref{315}, we obtain
\beq \label{leqlong}
\begin{split}
    T_2 +T_3 &\leq \sum_{|x_i-\xs|\geq \rho} \dfrac{b_i\ep}{x_i-\xs} + \dfrac{v(\xs+\rho)+ v(\xs-\rho)-2v(\xs) }{\rho} \\
    &+\dfrac{v(\xs) - v(x_1)}{\xs-x_1} - \dfrac{v(\xs)-v(x_{N_\ep})}{\xs-x_{N_\ep}}+ \dfrac{12\ep}{\rho}.
\end{split}
\eeq
Similarly, by applying the same changing of variables to \eqref{geqbound}, one can show that 
\beq \label{geqlong}
\begin{split}
    T_2 + T_3 &\geq \sum_{|x_i-\xs|\geq \rho} \dfrac{b_i\ep}{x_i-\xs} + \dfrac{v(\xs+\rho) + v(\xs-\rho)-2v(\xs)}{\rho} \\
    &+\dfrac{v(\xs) - v(x_1)}{\xs-x_1} - \dfrac{v(\xs)-v(x_{N_\ep})}{\xs-x_{N_\ep}}- \dfrac{12\ep}{\rho},
\end{split}
\eeq
where the summation on the right-hand side is zero if there are not points $x_i$ such that  $|x_i-\xs|\ge \rho$. 
Next, using that
\beqs
\begin{split}
    \inf_{ (-\infty, x_1]}v \leq v(x) \leq v(x_1) + 2\ep & \quad \text{for} \ x\in (-\infty, x_1]\\
    v(x_{N_\ep}) - 2\ep \leq v(x) \leq \sup_{ [x_{N_\ep}, +\infty)} v & \quad \text{for} \ x\in [x_{N_\ep}, +\infty),
\end{split}
\eeqs
we have the following estimates
\beq \label{t1leq}
T_1 = \int_{-\infty}^{x_{1}} \frac{v(x) - v(\xs)}{(x-\xs)^2} dx \leq \int_{-\infty}^{x_{1}} \frac{v(x_1) + 2\ep - v(\xs)}{(x-\xs)^2} dx = -\dfrac{v(x_1) - v(\xs)}{x_1 - \xs} - \dfrac{2\ep}{x_1 - \xs},
\eeq
\beq \label{t1geq}
T_1 = \int_{-\infty}^{x_{1}} \frac{v(x) - v(\xs)}{(x-\xs)^2} dx \geq \int_{-\infty}^{x_{1}} \frac{\inf_{ (-\infty, x_1]}v - v(\xs)}{(x-\xs)^2} dx = -\dfrac{\inf_{ (-\infty, x_1]}v - v(\xs)}{x_1 - \xs},
\eeq 
and
\beq \label{t4leq}
T_4 = \int_{x_{N_\ep}}^{+\infty} \frac{v(x) - v(\xs)}{(x-\xs)^2} dx \leq \dfrac{\sup_{ [x_{N_\ep}, +\infty)} v - v(\xs)}{x_{N_\ep} - \xs},
\eeq
\beq \label{t4geq}
T_4 = \int_{x_{N_\ep}}^{+\infty} \frac{v(x) - v(\xs)}{(x-\xs)^2} dx \geq \dfrac{v(\xs) - v(x_{N_\ep})}{\xs - x_{N_\ep}} - \dfrac{2\ep}{x_{N_\ep}-\xs}.
\eeq
Combining \eqref{leqlong}, \eqref{t1leq} and \eqref{t4leq}, we obtain 
\beq \label{i2leq}
\begin{split}
    \I^{2,\rho}[v](\xs) &\leq \sum_{|x_i-\xs|\geq \rho} \dfrac{b_i\ep}{x_i-\xs} + \dfrac{v(\xs+\rho) + v(\xs-\rho)-2v(\xs)}{\rho}\\
    &+ \dfrac{\sup_{[x_{N_\ep}, +\infty)} v - v(x_{N_\ep})}{x_{N_\ep} - \xs} + \dfrac{2\ep}{\xs-x_1} + \dfrac{12\ep}{\rho}.
\end{split}
\eeq
Similarly, combining \eqref{geqlong}, \eqref{t1geq} and \eqref{t4geq}, we obtain 
\beq \label{i2geq}
\begin{split}
    \I^{2,\rho}[v](\xs) &\geq \sum_{|x_i-\xs|\geq \rho} \dfrac{b_i\ep}{x_i-\xs} + \dfrac{v(\xs+\rho) + v(\xs-\rho)-2v(\xs)}{\rho}\\
    &-\dfrac{v(x_1) - \inf_{ (-\infty, x_1]}v}{\xs-x_1} - \dfrac{2\ep}{x_{N_\ep} - \xs} - \dfrac{12\ep}{\rho}.
\end{split}
\eeq
Since $\xs - x_1 \geq \rho$, $x_{N_\ep} - \xs \geq\rho$, and
\beqs
0 \leq \sup_{ [x_{N_\ep}, +\infty)} v - v(x_{N_\ep}) \leq 2\ep \ \ \text{and} \ \ 0\leq v(x_1) - \inf_{ (-\infty, x_1]}v \leq 2\ep,
\eeqs
we conclude that 
\beq \label{lastetm}
\begin{split}
\dfrac{\sup_{ [x_{N_\ep}, +\infty)} v - v(x_{N_\ep})}{x_{N_\ep} - \xs} + \dfrac{2\ep}{\xs-x_1} &\leq \dfrac{2\ep}{\rho} + \dfrac{2\ep}{\rho} = \dfrac{4\ep}{\rho}\\
-\dfrac{v(x_1) - \inf_{ (-\infty, x_1]}v}{\xs-x_1} - \dfrac{2\ep}{x_{N_\ep} - \xs} &\geq 
-\dfrac{2\ep}{\rho} - \dfrac{2\ep}{\rho} = -\dfrac{4\ep}{\rho}.
\end{split}
\eeq
Finally, combining \eqref{i2leq}, \eqref{i2geq} and \eqref{lastetm} gives \eqref{approxIlonglemeq}, which completes the proof.
\end{proof}
The  following proposition is an immediate consequence of Lemma \ref{approxIshortlem} and Lemma \ref{approxIlonglem}.

\begin{prop}\label{apprIcor}
Assume that $v$ satisfies \eqref{vass} and let 
$x_i$ and $b_i$ be  defined as in \eqref{xi} and  \eqref{bidef}.
Let  $r=r_\ep=o_\ep(1)$ and $\ep/r=o_\ep(1)$. Then, for any  $ \xs\in(x_1+r,x_{N_\ep}-r)$,
\beqs\frac{1}{\pi}\sum_{|x_i-\xs|\ge r}\frac{\ep b_i}{x_i-\xs}=\I[v](\xs)+ o_\ep(1).
\eeqs
\end{prop}
\begin{rem}
Heuristically, the cancellation of the quantity 
$$\frac{1}{\pi}\frac{v(\xs+\rho)+v(\xs-\rho)-2v(\xs)}{\rho}$$ when summing up \eqref{I1rhoapplemeq} and \eqref{approxIlonglemeq} comes from the fact that 
\beq\label{heuristic1sect4}\sum_{r<|x_i-\xs|<\rho}\frac{b_i\ep}{x_i-\xs}\simeq \int_{r<|x-\xs|<\rho}\frac{v_x(x)}{x-\xs}dx\eeq and 
\beq\label{heuristic2sect4}\sum_{|x_i-\xs|>\rho}\frac{b_i\ep}{x_i-\xs}\simeq \int_{|x-\xs|>\rho}\frac{v_x(x)}{x-\xs}dx,\eeq 
and that by performing an integration by parts, 
\beqs\begin{split}\int_{r<|x-\xs|<\rho}\frac{v_x(x)}{x-\xs}dx&=\int_{r<|x-\xs|<\rho}\frac{v(x)-v(\xs)}{(x-\xs)^2}dx+\frac{v(\xs+\rho)+v(\xs-\rho)-2v(\xs)}{\rho}
\\&-\frac{v(\xs+r)+v(\xs-r)-2v(\xs)}{r},
\end{split}
\eeqs
and 
\beqs\begin{split}\int_{|x-\xs|>\rho}\frac{v_x(x)}{x-\xs}dx&=\int_{|x-\xs|>\rho}\frac{v(x)-v(\xs)}{(x-\xs)^2}dx-\frac{v(\xs+\rho)+v(\xs-\rho)-2v(\xs)}{\rho}.\end{split}
\eeqs
For the heuristics of \eqref{heuristic1sect4} and \eqref{heuristic2sect4} we refer to Section 2.1 in \cite{patsan}. 
\end{rem}
\begin{rem}
Notice that in Lemma \ref{approxIshortlem}, Lemma \ref{approxIlonglem} and Proposition \ref{apprIcor}, the error $o_\ep(1)$ satisfies
\beq\label{o1bound}o_\ep(1)=O(r)+O\left(\frac{\ep}{r}\right).\eeq
\end{rem}

\begin{lem}\label{lemmaerrorshortdistance}
Under the assumptions  of Lemma  \ref{approxIshortlem}, let 
$\xs=x_{i_0}+\ep\gamma$, where $x_{i_0}$ is the closest point to $\xs$. Then, there exists $r=r_\ep $ satisfying  $\ep^\frac{5}{8}\le r\le c\ep^\frac{1}{2}$, with $c$ depending on the $C^{1,1}$ norm of $v$, such that 
\beq\label{shortdostaneqlem}\frac{1}{\pi}\sum_{\stackrel{i\neq i_0}{|x_i-\xs|< r}}\frac{\ep b_i}{x_i-\xs}= O(\ep^\frac{1}{8})+O(\gamma).\eeq
\end{lem}
\begin{proof}
The proof is a minor modification of the proof of Lemma 4.6 in \cite{patsan}. 
Let $K>0$ be such that $\|v_{xx}\|_{L^\infty(\R)}\leq K$. We consider  three cases.

\medskip
{\em Case 1: $|v_x({x_{i_0}})|\le 12 K^\frac12\ep^\frac12$.} 

By Taylor expansion and Young's inequality, we have  
\beqs\begin{split} \ep = |b_{i_0+1}\ep|=|v({x_{i_0+1}})-v({x_{i_0}})|&\leq |v_x({x_{i_0}})||x_{i_0+1}-x_{i_0}|+ \frac{K}{2}(x_{i_0+1}-x_{i_0})^2
\\&\leq \frac{v_x({x_{i_0}})^2}{2 (12)^2K}+\left(\frac{12^2K}{2}+\frac{K}{2}\right) (x_{i_0+1}-x_{i_0})^2\\&\leq \frac\ep2 +\frac{12^2+1}{2}K(x_{i_0+1}-x_{i_0})^2,
\end{split}
\eeqs
which implies
$$x_{i_0+1}-x_{i_0}\geq c\ep^\frac{1}{2},$$
for some $c>0$ independent of $\ep$. 
Similarly, one can prove that 
$$x_{i_0}-x_{i_0-1}\geq c\ep^\frac{1}{2}.$$
Since $x_{i_0}$ is the closest point to $\xs$, we must have that $\xs-x_{i_0-1}\geq c\ep^\frac{1}{2}/2$ and 
$x_{i_0+1}-\xs\geq c\ep^\frac{1}{2}/2$.
Therefore, if we choose $r=r_\ep=c\ep^\frac{1}{2}/4$, there is  no index  $i\neq i_0$ for which 
$|\xs-x_i|\leq r$ and thus  \eqref{shortdostaneqlem} is trivially true. 

\medskip

Next,  we show that 
\beq\label{centeredserier}\frac{1}{\pi}\sum_{\stackrel{i\neq i_0}{|x_i-\xs|< r}}\frac{\ep b_i}{x_i-x_{i_0}}= O(\ep^\frac{1}{8}).
\eeq
We consider  two more cases. 
\medskip

{\em Case 2:  $12 K^\frac12\ep^\frac12\leq |v_x({x_{i_0}})|\le \ep^{\frac12-\tau}$, for some $\tau\in (0,1/4)$.}

If $|\xs-x_{i_0}|\geq \ep^\frac{1}{2}/(4K^\frac{1}{2})$, then we choose 
$r=\ep^\frac{1}{2}/(8K^\frac{1}{2})$ and, as in Case 1, there is  no index  $i\neq i_0$ for which 
$|\xs-x_i|\leq r$. Thus \eqref{shortdostaneqlem} holds  true. 

Now, assume $|\xs-x_{i_0}|\leq \ep^\frac{1}{2}/(4K^\frac{1}{2})$ and define
\beq\label{randep}r:=\frac{\ep^\frac{1}{2}}{2K^\frac{1}{2}}\geq 2 |\xs-x_{i_0}|.\eeq

We claim that $v$ is monotone in $(\xs-r,\xs+r)$, where $r$ is defined in \eqref{randep}. To show this, suppose that $v_x(x_{i_0}) \geq 12 K^\frac12\ep^\frac12$. Then, for $x\in (\xs-r,\xs+r)$,
\beqs
|x-x_{i_0}| \leq |x-\xs| + |\xs- x_{i_0}| < r + \dfrac{r}{2} = \dfrac{3r}{2}.
\eeqs
Now, we have that
\beqs
v_x(x) - v_x(x_{i_0}) \geq -K|x-x_{i_0}| \geq -\dfrac{3rK}{2},
\eeqs
which gives us 
\beqs
v_x(x) \geq v_x(x_{i_0})  -\dfrac{3rK}{2} \geq 12K^\frac12 \ep^\frac12 - \dfrac{3K^\frac12\ep^\frac12}{4} = \dfrac{45 K^\frac12 \ep^\frac12}{4}>0.
\eeqs
Hence, $v$ is monotone increasing   in $(\xs-r,\xs+r)$.  
Let $M_r$ and $N_r$ be respectively the smallest and the larger index $i$ such that $x_i\in (\xs-r,\xs+r)$, that is
\beq\label{xnr-xs}\begin{split}
&x_{M_r-1}\le \xs-r< x_{M_r}\\&
x_{N_r}< \xs+r\le x_{N_r+1}.
\end{split}\eeq
By the monotonicity of $v$ in $(x-r,x+r)$,
$$ v(x_{N_r})=v(x_{M_r})+\ep (N_r-M_r) \quad\text{and}\quad  v(x_{i_0})=v(x_{M_r})+\ep (i_0-M_r),$$ from which
\beq\label{allthevxdfejtiypfbis}
 v(x_{N_r})+v(x_{M_r})-2 v(x_{i_0})=\ep(N_r+M_r-2i_0).
\eeq
Again by the monotonicity of $v$ in $(x-r,x+r)$ and by making a Taylor expansion, we get, for $i=M_r,\ldots,N_r$
\beq\label{allthevxdfejtiypftris} \ep(i-i_0)=v(x_i)-v({x_{i_0}})=v_x({x_{i_0}})(x_i-x_{i_0})+O(r^2),\eeq
where $|O(r^2)|\leq K(2r)^2/2=\ep/2$,
from which 
\beq\label{distancexixi0} x_i-x_{i_0}=\frac{\ep(i-i_0)+O(r^2)}{v_x({x_{i_0}})}.\eeq
Therefore, we can write
\beq\label{sumx-xi0smallprima}\begin{split}\sum_{i\neq i_0\atop |x_i-\xs|< r}\frac{\ep}{x_i-x_{i_0}}&=\sum_{i=M_r\atop i\neq i_0}^{N_r}\frac{v_x(x_{i_0})\ep}{\ep(i-i_0)+O(r^2)}\\&=
\sum_{i=M_r\atop }^{i_0-1}\frac{v_x(x_{i_0})\ep}{\ep(i-i_0)+O(r^2)}+\sum_{i=i_0+1}^{N_r}\frac{v_x(x_{i_0})\ep}{\ep(i-i_0)+O(r^2)}.
\end{split}
\eeq
Now, suppose without loss of generality that $N_r-i_0\leq i_0-M_r$.
Then, 
\beq\label{pedalssonno}\begin{split}
&\sum_{i=M_r\atop }^{i_0-1}\frac{\ep}{\ep(i-i_0)+O(r^2)}+\sum_{i=i_0+1}^{N_r}\frac{\ep}{\ep(i-i_0)+O(r^2)}\\&=
\sum_{k=1}^{i_0-M_r}\frac{\ep}{-\ep k+O(r^2)}+\sum_{k=1}^{N_r-i_0}\frac{\ep}{\ep k+O(r^2)}\\&
=\sum_{k=1}^{N_r-i_0}\ep\left(\frac{1}{-\ep k+O(r^2)}+\frac{1}{\ep k+O(r^2)}\right)+\sum_{k=N_r-i_0+1}^{i_0-M_r}\frac{\ep}{-\ep k+O(r^2)}.
\end{split}
\eeq
We can bound the first term of the right hand-side of the last equality  as follows
\beq\label{serveperclaim3}\begin{split}
\left|\sum_{k=1}^{N_r-i_0}\ep\left(\frac{1}{-\ep k+O(r^2)}+\frac{1}{\ep k+O(r^2)}\right)\right|
&=\frac{2|O(r^2)|}{\ep}\left|\sum_{k=1}^{N_r-i_0}\frac{1}{(-k+\frac{O(r^2)}{\ep})(k+\frac{O(r^2)}{\ep})}\right|\\&
\leq\sum_{k=1}^\infty \frac{1}{k^2-\frac14}\\&
=C,
\end{split}
\eeq
where we used that $|O(r^2)|/\ep\leq 1/2$. 
Therefore, 
\beq\label{pedalssonnobis}\begin{split}
v_x(x_{i_0})\left|\sum_{k=1}^{N_r-i_0}\ep\left(\frac{1}{-\ep k+O(r^2)}+\frac{1}{\ep k+O(r^2)}\right)\right|\leq Cv_x(x_{i_0})\leq C\ep^{\frac12-\tau}.
\end{split}
\eeq
Next, by using that $\sum_{k=n}^m1/k\leq (m-n+1)/n$,  \eqref{allthevxdfejtiypfbis} and the first equality in \eqref{allthevxdfejtiypftris}, we get
\beq\label{secondpieceestilemii0}\begin{split}
\left|\sum_{k=N_r-i_0+1}^{i_0-M_r}\frac{\ep}{-\ep k+O(r^2)}\right|&\leq \sum_{k=N_r-i_0+1}^{i_0-M_r}\frac{\ep}{\ep k-|O(r^2)|}\\&\leq
\frac{-(\ep N_r+\ep M_r-2\ep i_0)}{\ep (N_r+1)-\ep i_0-|O(r^2)|}\\&
=\frac{-(v(x_{N_r})+v(x_{M_r})-2v(x_{i_0}))}{v(x_{N_r})-v(x_{i_0})+\ep-|O(r^2)|}\\&
\le \frac{-(v(x_{N_r})+v(x_{M_r})-2v(x_{i_0}))}{v(x_{N_r})-v(x_{i_0})}.
\end{split}
\eeq
By   \eqref{oscvxixi+1}, the fact that $v(x_{M_r})\ge v(\xs-r)$  and the regularity of $v$, 
\beq\label{topestinrmri0}
0\le-( v(x_{N_r})+v(x_{M_r})-2v(x_{i_0}))\leq -(v(\xs+r)+ v(\xs-r)-2v(\xs))+6\ep\leq Kr^2+6\ep\le C\ep.
\eeq 
Now, by using that  $v_x(x_{i_0})\ge 12K^\frac12\ep^\frac12$,  that $|\xs-x_{i_0}|\leq r/2$,  and   \eqref{oscvxixi+1}, we get  
\beq\label{bottomestinrmri0}\begin{split}
v(x_{N_r})-v(x_{i_0})&\geq v(\xs+r)-v(x_{i_0})-2\ep\\&\ge v_x(x_{i_0})(r-|\xs-x_{i_0}|)-\frac{K}{2}(2r)^2-2\ep\\&
\ge  v_x(x_{i_0})\frac r2-\frac32\ep \\&
=v_x(x_{i_0})\frac r2-12K^\frac12\ep^\frac12\frac{r}{4}\\&
\ge v_x(x_{i_0})\frac{r}{4}\\&
= v_x(x_{i_0})\frac{\ep^\frac12}{8K^\frac12}.
\end{split}
\eeq
From \eqref{secondpieceestilemii0}, \eqref{topestinrmri0} and \eqref{bottomestinrmri0}, we infer that
\beq\label{secondpiecefinallemii0}
\left|\sum_{i=i_0-M_r+1}^{N_r-i_0}\frac{\ep v_x(x_{i_0})}{-\ep k+O(r^2)}\right|\leq \frac{v_x(x_{i_0})C\ep 8K^\frac12}{v_x(x_{i_0})\ep^\frac12}\leq
C\ep^\frac12.
\eeq
Finally, \eqref{sumx-xi0smallprima}, \eqref{pedalssonno}, \eqref{pedalssonnobis} and \eqref{secondpiecefinallemii0} imply
\beqs \left|\frac{1}{\pi}\sum_{i\neq i_0\atop |x_i-\xs|< r}\frac{\ep}{x_i-x_{i_0}}\right|\leq C\ep^{\frac12-\tau}\leq C\ep^\frac{1}{4},
\eeqs
which gives \eqref{centeredserier}.

One can similarly show \eqref{centeredserier} when $v_x(x_{i_0}) \leq -12K^\frac12 \ep^\frac12$. 

\medskip

{\em Case 3:  $|v_x({x_{i_0}})|\ge \ep^{\frac12-\tau}$, for some $\tau\in(0,1/4)$.}

Arguing as before,   we can assume that $|\xs-x_{i_0}|\le \ep^{\frac{1+\tau}{2}}$. 
Then, we define
\beq\label{rcase3}r:=2\ep^{\frac{1+\tau}{2}}\ge 2|\xs-x_{i_0}|.\eeq

Notice that $r\ge \ep^\frac{5}{8}$.

Assume that $v_x({x_{i_0}})\ge \ep^{\frac12-\tau}$, then as in Case 2 we can prove that $v$ is monotone increasing in the interval $(x-r,x+r)$. 

Let $M_r$ and $N_r$ be defined as in Case 2.  Assume,    without loss of generality,  that $N_r-i_0\le i_0-M_r$. Then   as before, we  write 
\beq\label{sumx-xi0smallprimacase3}\begin{split}\sum_{i\neq i_0\atop |x_i-\xs|< r}\frac{\ep}{x_i-x_{i_0}}&=
\sum_{k=1}^{N_r-i_0}\ep v_x(x_{i_0})\left(\frac{1}{-\ep k+O(r^2)}+\frac{1}{\ep k+O(r^2)}\right)+ \sum_{k=N_r-i_0+1}^{i_0-M_r}\frac{\ep v_x(x_{i_0})}{-\ep k+O(r^2)}.
\end{split}
\eeq
By \eqref{serveperclaim3} and the definition \eqref{rcase3}  of $r$, 
\beq\label{ccase3step1}
\left|\sum_{k=1}^{N_r-i_0}\ep v_x(x_{i_0})\left(\frac{1}{-\ep k+O(r^2)}+\frac{1}{\ep k+O(r^2)}\right)\right|\leq 
 Cv_x(x_{i_0})\frac{|O(r^2)|}{\ep}\leq C\ep^\tau.
\eeq
By \eqref{secondpieceestilemii0}, \eqref{topestinrmri0}
 and \eqref{bottomestinrmri0},  and by using that $ v_x({x_{i_0}})\ge C\ep^{\frac12-\tau}$ and \eqref{rcase3}, we get
  \beq\label{secondpieceestilemii0case3}\begin{split}
\left|\sum_{k=N_r-i_0+1}^{i_0-M_r}\frac{\ep}{-\ep k+O(r^2)}\right|&\leq\frac{C\ep}{ v_x(x_{i_0})\frac r2-\frac32\ep }\leq C\ep^\frac{\tau}{2},
\end{split}
\eeq
for $\ep$ small enough (independently of $\xs$).
Estimates  \eqref{sumx-xi0smallprimacase3},\eqref{ccase3step1} and  \eqref{secondpieceestilemii0case3} imply
\beqs\left|\sum_{i\neq i_0\atop |x_i-\xs|< r}\frac{\ep}{x_i-x_{i_0}}\right|\leq C\ep^\frac{\tau}{2}\leq C\ep^\frac{1}{8},
\eeqs
which gives \eqref{centeredserier}. One can similarly show \eqref{centeredserier} when $v_x({x_{i_0}})\le -\ep^{\frac12-\tau}$

Finally, to prove \eqref{shortdostaneqlem}, we estimate
\beq\label{sumseriesdiffexsxi0}\begin{split}\left|\sum_{i\neq i_0\atop |x_i-\xs|< r}\frac{\ep}{x_i-\xs}
-\sum_{i\neq i_0\atop |x_i-x_{i_0}|<r}\frac{\ep}{x_i-x_{i_0}}\right|&=\left|
\sum_{i\neq i_0\atop |x_i-\xs|< r}\frac{\ep^2\gamma}{(x_i-\xs)(x_i-x_{i_0})}\right|.
\end{split}\eeq
Assume, without loss of generality that $\xs=x_{i_0}+\ep\gamma$, with $\gamma\ge0$, that is, $\xs\in[x_{i_0},x_{i_0+1})$.
Then, 
\beqs |x_i-\xs|\ge\begin{cases}|x_i-x_{i_0}|&\text{if }i\leq i_0-1\\
\frac{x_{i_0+1}-x_{i_0}}{2}&\text{if }i=i_0+1 \\
x_i-x_{i_0+1}&\text{if }i\geq i_0+2.
\end{cases}
\eeqs
Moreover, by \eqref{proplippart1}, $x_{i_0+1}-x_{i_0}\ge \ep L^{-1}$. Therefore, 
\beq\label{lastlemmagammaxixi0}\left|
\sum_{i\neq i_0\atop |x_i-\xs|< r}\frac{\ep^2\gamma}{(x_i-\xs)(x_i-x_{i_0})}\right|\le \sum_{i\leq i_0-1}\frac{\ep^2\gamma}{(x_i-x_{i_0})^2}
+2L^2\gamma+ \sum_{i\geq i_0+2}\frac{\ep^2\gamma}{(x_i-x_{i_0+1})^2}\le C\gamma,
\eeq
where in the last inequality we used \eqref{i/k^2sum}.
By \eqref{sumseriesdiffexsxi0} and \eqref{lastlemmagammaxixi0} we get 
\beqs
\left|\sum_{i\neq i_0\atop |x_i-\xs|<r}\frac{\ep}{x_i-\xs}
-\sum_{i\neq i_0\atop |x_i-x_{i_0}|<r}\frac{\ep}{x_i-x_{i_0}}\right|\le C\gamma,
\eeqs
which together with \eqref{centeredserier} gives  \eqref{shortdostaneqlem}. 
\end{proof}

The following proposition is an immediate consequence of  Lemma \ref{approxIshortlem}, Proposition \ref{apprIcor}
and Lemma \ref{lemmaerrorshortdistance}.
\begin{prop}\label{apprIcorall}
Assume that $v$ satisfies \eqref{vass} and let 
$x_i$ and $b_i$ be defined as in \eqref{xi} and  \eqref{bidef}. Then, there  exists $c>0$ depending on the $C^{1,1}$ norm of $v$ such that if $\rho\ge c\ep^\frac12$ 
and $\xs\in(x_1+\rho,x_{N_\ep}-\rho)$, 
$\xs=x_{i_0}+\ep\gamma$, where $x_{i_0}$ is the closest point to $\xs$, then
\beq\label{apprIcorallest1}\frac{1}{\pi}\sum_{\stackrel{i\neq i_0}{|x_i-\xs|\le \rho}}\frac{\ep b_i}{x_i-\xs}=\I^{1,\rho}[v](\xs)+O(\gamma)+\frac{1}{\pi}\frac{v(\xs+\rho)+v(\xs-\rho)-2v(\xs)}{\rho}+o_\ep(1),
\eeq
and
\beq\label{apprIcorallest2}\frac{1}{\pi}\sum_{i\neq i_0}\frac{\ep b_i}{x_i-\xs}=\I[v](\xs)+ o_\ep(1)+O(\gamma).
\eeq
\end{prop}
\begin{proof}

Fix $\xs$ and let $r$ and $c$ be given by Lemma  \ref{lemmaerrorshortdistance}. Then $\ep^\frac{5}{8}\le r\le c\ep^\frac{1}{2}\le\rho$.
By Lemma \ref{approxIshortlem} and recalling \eqref{o1bound},
\beqs \frac{1}{\pi}\sum_{i\neq i_0\atop r\le |x_i-\xs|\le \rho}\frac{\ep}{x_i-\xs} =\I^{1,\rho}[v](\xs)+\frac{1}{\pi}\frac{v(\xs+\rho)+v(\xs-\rho)-2v(\xs)}{\rho}+O\left(\ep^\frac{3}{8}\right).
\eeqs
Combining this estimate with \eqref{shortdostaneqlem} yields \eqref{apprIcorallest1}. 

Similarly, by Proposition \ref{apprIcor} and Lemma \ref{lemmaerrorshortdistance}, we get \eqref{apprIcorallest2}.
\end{proof}

\begin{rem}\label{vcopntstanti1rem}
If $\ep|\gamma|=|\xs-x_{i_0}|>c\ep^\frac{1}{2}\ge r$, then $|\xs-x_{i}|>r$ for all $i$ and 
\beqs\label{vcopntstanti1}\begin{split} \frac{1}{\pi}\sum_{\stackrel{i\neq i_0}{|x_i-\xs|\le \rho}}\frac{\ep}{x_i-\xs}&=
\frac{1}{\pi}\sum_{r< |x_i-\xs|\le \rho}\frac{\ep}{x_i-\xs}\\&=
\I^{1,\rho}[v](\xs)+ \frac{1}{\pi}\frac{v(\xs+\rho)+v(\xs-\rho)-2v(\xs)}{\rho}+o_\ep(1).
\end{split} \eeqs
Therefore, we can assume 
\beq\label{Ogammabpund}\ep^\frac{1}{2}|\gamma|\leq C\eeq
in \eqref{apprIcorallest1} and \eqref{apprIcorallest2}.
\end{rem}



\begin{lem}\label{vapproxphisteps}
Assume that $v$ satisfies \eqref{vass} and let
$x_i$ be defined as in \eqref{xi}. Let $\phi$ be defined by \eqref{phi}. Let 
$1\le M<N\leq N_\ep$ and $ R\ge c\ep^\frac{1}{2}$, with $c>0$  given by Proposition \ref{apprIcorall}.  Then, 
for all  $x\in (x_{M}+R,x_N-R)$,
\beq \label{shortapprox}
\left|\sum_{i=M}^{N}\ep \phi\left(\frac{x-x_i}{\ep\delta},b_i\right)+v(x_M)
 -v(x)\right|\leq 
o_\ep(1) + \frac{C\ep^2\delta N_\ep}{R},\eeq
with $o_\ep(1)$ independent of $R$ and $x$.
\end{lem}
\begin{proof}
Fix $x\in (x_{M}+R,x_N-R)$, and let $x_{i_0}$ be the closest point to $x$. Then, we have 
\beq \label{342}
x-x_i >0 \quad \text{for} \quad i\leq i_0-1, \quad x-x_i<0 \quad \text{for}  \quad i\geq i_0+1
\eeq
and by \eqref{oscvxixi+1}
\beq \label{vxuplow}
v(x_{i_0}) -2\ep \leq v(x) \leq v(x_{i_0}) + 2\ep.
\eeq
By splitting the sum and using \eqref{vxuplow}, we obtain
\beqs
    \sum_{i=M}^{N} \ep \phi\left(\frac{x-x_i}{\ep\delta},b_i\right) -v(x)
   \leq \sum_{\stackrel{i=M}{i\not= i_0} }^{N} \ep \phi\left(\frac{x-x_i}{\ep\delta},b_i\right) 
    + \ep \phi\left(\frac{x-x_{i_0}}{\ep\delta},b_{i_0}\right)-v(x_{i_0}) +2\ep.
\eeqs
Using the asymptotic estimate \eqref{phiinfinity}, \eqref{vxidef}, $\phi\le1$, and \eqref{342}, we have that 
\beqs \begin{split}
    \sum_{i=M}^{N} \ep & \phi\left(\frac{x-x_i}{\ep\delta},b_i\right) -v(x)\\
     &\leq  \sum_{\stackrel{i=M}{i\not= i_0} }^{N} \ep \left (H\left ( \frac{x-x_i}{\ep\delta},b_i\right) +\dfrac{\ep \delta}{\alpha \pi}\dfrac{b_i \ep}{x_i-x} + \dfrac{K_1 \ep^2 \delta^2}{(x_i-x)^2} \right ) - v(x_{i_0})+ 3\ep \\
    &= \ep( n^+_{M,i_0-1} - n^-_{M,i_0-1}) -
   v(x_{i_0}) + 3\ep + \dfrac{\ep \delta}{\alpha \pi} \sum_{\stackrel{i=M}{i\not= i_0} }^{N} \dfrac{b_i \ep}{x_i-x} + \sum_{\stackrel{i=M}{i\not= i_0} }^{N} \dfrac{K_1 \ep^3 \delta^2}{(x_i-x)^2}.
    \end{split}
\eeqs
Notice that by \eqref{stronzo}, 
\beqs
\ep n^+_{M,i_0-1} - \ep n^-_{M,i_0-1} - v(x_{i_0})  = -v(x_M)+b_M\ep - b_{i_0}\ep.
\eeqs
We conclude that 
\beq \label{sqt}
    \sum_{i=M}^{N} \ep \phi\left(\frac{x-x_i}{\ep\delta},b_i\right)+v(x_M) -v(x)
     \leq   \dfrac{\ep \delta}{\alpha \pi} \sum_{\stackrel{i=M}{i\neq i_0}}^N \dfrac{b_i \ep}{x_i-x} +C\ep,
\eeq
where we also used \eqref{i/k^2sum}. By decomposing the sum in the right-hand side of the above inequality as follows
\beq \label{splitt}
\sum_{\stackrel{i=M}{i\neq i_0}}^N \dfrac{b_i \ep}{x_i-x} = \sum_{\stackrel{i\not= i_0}{|x_i-x|\leq R}}  \dfrac{b_i \ep}{x_i-x} 
+ \sum_{\stackrel{i=M}{|x_i-x|> R}}^N  \dfrac{b_i \ep}{x_i-x},
\eeq
applying \eqref{apprIcorallest1}  and recalling  \eqref{Ogammabpund}, we obtain
\beq \label{leqrt}\begin{split}
 \dfrac{\ep \delta}{\alpha \pi}\left|\sum_{\stackrel{i\not= i_0}{|x_i-x|\leq R}}  \dfrac{b_i \ep}{x_i-x} \right|
 &\leq C\ep\delta \left(\I^{1,R}[v](x)+\gamma+\frac{v(\xs+R)+v(\xs-R)-2v(\xs)}{R}\right)\\&
\leq C\ep\delta(1+\gamma) \leq C\ep^\frac{1}{2}\delta=o_\ep(1).
\end{split}\eeq
For the second term in \eqref{splitt}, we have 
\beq \label{secndt}
 \dfrac{\ep \delta}{\alpha \pi} \left| \sum_{\stackrel{i=M}{|x_i-x|> R}}^N \dfrac{b_i \ep}{x_i-x}\right| \leq  \dfrac{\ep \delta}{\alpha \pi}  \sum_{\stackrel{i=M}{|x_i-x|> R}}^N  \dfrac{\ep}{R} 
 \leq  \dfrac{\ep \delta}{\alpha \pi}\frac{\ep N_\ep}{R}= \frac{C\ep^2\delta N_\ep}{R}.
 \eeq
Combining \eqref{sqt}, \eqref{splitt}, \eqref{leqrt} and \eqref{secndt}, we get
\beq 
\sum_{i=M}^{N} \ep \phi\left(\frac{x-x_i}{\ep\delta},b_i\right)+v(x_M) -v(x)
     \leq o_\ep(1) + \frac{C\ep^2\delta N_\ep}{R}.
\eeq
Similarly, one can show that 
\beq 
\sum_{i=M}^{N} \ep \phi\left(\frac{x-x_i}{\ep\delta},b_i\right)+v(x_M) -v(x)
     \geq o_\ep(1) - \frac{C\ep^2\delta N_\ep}{R}.
\eeq
This completes the proof.
\end{proof}


\begin{lem}\label{vapprowhatremains}
Under the assumptions of Lemma \ref{vapproxphisteps},    there exists $C>0$  independent of $\ep$  and $R$ such that 
for all  $x>x_N+R$, 
\beq\label{vapprowhatremainseq1}
\left|\sum_{i=M}^{N}\ep \phi\left(\frac{x-x_i}{\ep\delta},b_i\right)+v(x_M) -v(x_N)\right|\leq o_\ep(1) + \frac{C\ep^2\delta N_\ep}{R}.
\eeq
 Furthermore, 
for all  $x<x_M-R$, 
\beq\label{vapprowhatremainseq2}
\left|\sum_{i=M}^{N}\ep \phi\left(\frac{x-x_i}{\ep\delta},b_i\right) \right|\leq o_\ep(1) + \frac{C\ep^2\delta N_\ep}{R}.
\eeq
\end{lem}
\begin{proof}
First, suppose that $x>x_N+R$. Then, $x-x_i >R$ for all $i = M,\ldots,N$. The sign of $b_i(x-x_i)$ depends on the sign of $b_i$. Then by decomposing  the sum and applying \eqref{phiinfinity}, we have 
\beqs \begin{split}
    \sum_{i=M}^{N}\ep &\phi\left(\frac{x-x_i}{\ep\delta},b_i\right)-v(x_N)\\
    &\leq  \sum_{i=M}^N \ep \left (H \left ( \dfrac{x-x_i}{\ep \delta},b_i \right )+ \dfrac{\ep \delta}{\alpha \pi}\dfrac{b_i \ep}{x_i-x} + \dfrac{K_1 \ep^2 \delta^2}{(x_i-x)^2} \right) -v(x_N)\\
    &=\ep n^+_{M,N} - \ep n^-_{M,N}   -v(x_N)+ \dfrac{\ep \delta}{\alpha \pi} \sum_{\stackrel{i=M}{}}^N \dfrac{b_i \ep}{x_i-x} +K_1\ep \delta^2 \sum_{\stackrel{i=M}{}}^N \dfrac{\ep^2}{(x_i-x)^2}.
\end{split} \eeqs
Notice that by \eqref{stronzo}, 
\beqs
    \ep n^+_{M,N} - \ep n^-_{M,N}  -v(x_N)= -v(x_M)+b_M\ep.
\eeqs
Hence, we have 
\beqs \begin{split}
\sum_{i=M}^{N} \ep \phi\left(\frac{x-x_i}{\ep\delta},b_i\right)+ v(x_M)-v(x_N) &\leq \dfrac{\ep \delta}{\alpha \pi} \sum_{\stackrel{i=M}{}}^N \dfrac{b_i \ep}{x_i-x} +K_1\ep \delta^2 \sum_{\stackrel{i=M}{}}^N \dfrac{\ep^2}{(x_i-x)^2}+\ep \\
&\leq \dfrac{\ep \delta}{\alpha \pi}  \sum_{\stackrel{i=M}{}}^N \dfrac{\ep}{R} +C\ep\delta^2+\ep \\
&\le C\ep^2 N_\ep\frac{\delta}{R} +C\ep\delta^2 +\ep 
\\&=  o_\ep(1) + \frac{C\ep^2\delta N_\ep}{R},
\end{split}
\eeqs 
using \eqref{i/k^2sum}.
One can similarly show that 
\beqs
\sum_{i=M}^{N} \ep \phi\left(\frac{x-x_i}{\ep\delta},b_i\right)+ v(x_M)-v(x_N) \geq o_\ep(1) - \frac{C\ep^2\delta N_\ep}{R},
\eeqs
which completes the proof of \eqref{vapprowhatremainseq1}.

Next, suppose that $x<x_M-R$. Notice that $x-x_i<-R$ for all $i = M,\ldots,N$ and by \eqref{phiinfinity}, we obtain
\beqs \begin{split}
    \sum_{i=M}^{N}\ep \phi\left(\frac{x-x_i}{\ep\delta},b_i\right)
    &\leq  \sum_{i=M}^N \ep \left (H \left ( \dfrac{x-x_i}{\ep \delta},b_i \right )+ \dfrac{\ep \delta}{\alpha \pi}\dfrac{b_i \ep}{x_i-x} + \dfrac{K_1 \ep^2 \delta^2}{(x_i-x)^2} \right) \\
    &= \dfrac{\ep \delta}{\alpha \pi} \sum_{\stackrel{i=M}{}}^N \dfrac{b_i \ep}{x_i-x} +K_1\ep \delta^2 \sum_{\stackrel{i=M}{}}^N \dfrac{\ep^2}{(x_i-x)^2}.
\end{split} \eeqs
Thus, as before, we get
\beqs
\sum_{i=M}^{N} \ep \phi\left(\frac{x-x_i}{\ep\delta},b_i\right) \leq o_\ep(1) + \frac{C\ep^2\delta N_\ep}{R}.
\eeqs
Similarly, we can prove that 
\beqs
\sum_{i=M}^{N} \ep \phi\left(\frac{x-x_i}{\ep\delta},b_i\right) \geq o_\ep(1) - \frac{C\ep^2\delta N_\ep}{R},
\eeqs
which gives \eqref{vapprowhatremainseq2}.
\end{proof}

\begin{prop}\label{approxpropfinal}
Assume that $v$ satisfies \eqref{vass} and let
$x_i$ and $b_i$ be defined as in \eqref{xi} and  \eqref{bidef}. Let $\phi$ be defined by \eqref{phi} and 
let $\delta=\delta(\ep)$ be such that $\ep^2N_\ep\delta=o_\ep(1)$. Then, for all  $x\in\R$, 
\beq\label{approxpripofinaleqbis}
\left|\sum_{i=1}^{N_\ep}\ep \phi\left(\frac{x-x_i}{\ep\delta},b_i\right)-v(x)\right|\leq o_\ep(1),
\eeq
where $o_\ep(1)$ is independent of $x$.
\end{prop}
\begin{proof}
Estimate  \eqref{approxpripofinaleqbis} is a consequence of the following inequality
\beq\label{approxpripofinaleqclaim}
\left|\sum_{i=1}^{N_\ep}\ep \phi\left(\frac{x-x_i}{\ep\delta},b_i\right)-v(x)\right|\leq o_\ep(1) + \dfrac{C\ep^2 \delta N_\ep}{R},
\eeq
for some $C>0$ independent of $\ep$ and $R=R_\ep\geq\ep\sqrt{N_\ep\delta}$. 

Let us denote $a_\ep: =\ep^2N_\ep\delta=o_\ep(1)$. Let $R=R_\ep:=\max\{a_\ep^\frac{1}{2},c\ep^\frac{1}{2}\}=o_\ep(1)$, with  $c$  given in Proposition \ref{apprIcorall}. 
If $x\in (x_1+R, x_{N_\ep}-R)$,  then \eqref{approxpripofinaleqclaim} follows from Lemma \ref{vapproxphisteps} with $M=1$ and $N=N_\ep$.
Next, let us assume $x>x_{N_\ep}+R$. Then, by \eqref{vapprowhatremainseq1}, we have
\beqs\begin{split}
\left|\sum_{i=1}^{N_\ep}\ep \phi\left(\frac{x-x_i}{\ep\delta},b_i\right)-v(x)\right|\leq |v(x_{N_\ep})-v(x)|+o_\ep(1).
\end{split}\eeqs
In the interval $[x_{N_\ep},+\infty)$, the oscillation of $v$ is less or equal than $2\ep$, therefore
\beqs | v(x_{N_\ep})-v(x)|\leq 2\ep,\eeqs
from which  \eqref{approxpripofinaleqclaim}  for   $x>x_{N_\ep}+R$ follows.
By using \eqref{vapprowhatremainseq2}, one can similarly prove 
\eqref{approxpripofinaleqclaim} when  $x<x_1-R$.

Next, assume $x_{N_\ep}-R\le x\le x_{N_\ep}+R$. Define $N$ to be an index such that 
\beqs
x_{N}\leq x_{N_\ep}-2R<x_{N+1}\le x_{N_\ep},
\eeqs
Using that $v(x) = v(x_{N_\ep}) + O(R)= v(x_N)+O(R)$, we get
\beqs\begin{split}
&\left |\sum_{i=1}^{N_\ep}\ep \phi\left(\frac{x-x_i}{\ep\delta},b_i\right)-v(x)\right |\\
&\leq \left |
\sum_{i=1}^{N_\ep}\ep \phi\left(\frac{x-x_i}{\ep\delta},b_i\right)-v(x_{N_\ep})\right |+O(R)\\
&= \left | \sum_{\stackrel{i=1}{x_i\leq x_{N_\ep}-2R}}^{N_\ep}\ep \phi\left(\frac{x-x_i}{\ep\delta},b_i\right) + \sum_{\stackrel{i=1}{x_i> x_{N_\ep}-2R}}^{N_\ep}\ep \phi\left(\frac{x-x_i}{\ep\delta},b_i\right) - v(x_{N\ep})  \right |+O(R)\\
&= \left | \sum_{i=1}^{N}\ep \phi\left(\frac{x-x_i}{\ep\delta},b_i\right) + \sum_{i=N+1}^{N_\ep}\ep \phi\left(\frac{x-x_i}{\ep\delta},b_i\right) - v(x_{N\ep})  \right |+O(R)\\
& \leq \left | \sum_{i=1}^{N}\ep \phi\left(\frac{x-x_i}{\ep\delta},b_i\right)  - v(x_N) \right |+\left | \sum_{i=N+1}^{N_\ep}\ep \phi\left(\frac{x-x_i}{\ep\delta},b_i\right)  \right |+O(R).
\end{split}\eeqs
By \eqref{vapprowhatremainseq1} with $M=1$, we obtain
\beq \label{1_term}
\left | \sum_{i=1}^{N}\ep \phi\left(\frac{x-x_i}{\ep\delta},b_i\right)  - v(x_N) \right | \leq o_\ep(1) + \dfrac{C\ep^2 \delta N_\ep}{R}.
\eeq
Since $0<\phi< 1$, we have
\beq \label{2_term}
  \left | \sum_{i=N+1}^{N_\ep}\ep \phi\left(\frac{x-x_i}{\ep\delta},b_i\right)  \right | 
 \leq  \sum_{i=N+1}^{N_\ep}\ep  = \ep (n_{N+1,N_\ep}^+ + n_{N+1,N_\ep}^-) = O(R),
\eeq
where $n_{N+1,N_\ep}^+ + n_{N+1,N_\ep}^-$ is the total number of particles lying inside the interval 
$$[x_{N+1}, x_{N_\ep}] \subset [x_{N_\ep} -2R, x_{N_\ep}+R],$$
which by \eqref{proplippart1} can be estimated by $CR/\ep$ for some constant $C>0$.

By \eqref{1_term} and \eqref{2_term}, we have proven \eqref{approxpripofinaleqclaim} when $x_{N_\ep}-R\le x\le x_{N_\ep}+R$.
Similarly, one can prove   \eqref{approxpripofinaleqclaim} when $x_{1}-R\le x\le x_{1}+R$, which concludes the proof of the proposition.
\end{proof}

We conclude this section with the following result  that will be used in Section \ref{supersolutionssec}.
The proof is an easy adaptation of the proof of  Lemma 4.13 in \cite{patsan}. 
\begin{lem}\label{sumxialwaysfiniteprop}
Assume that $v$ satisfies \eqref{vass} and let
$x_i$ and $b_i$ be defined as in \eqref{xi} and  \eqref{bidef}. 
 Then, there exists $C>0$ such that for all $x\in\R$, 
\beqs\label{sumxialwaysfiniteeq}\left| \sum_{i\neq i_0}\frac{\ep b_i}{x_i-x}\right|\leq C.
\eeqs
\end{lem}

\begin{rem}\label{importantrem}
Proposition \ref{approxpropfinal} and Lemma \ref{sumxialwaysfiniteprop} hold true if in \eqref{vass}  the assumption $v(-\infty)=0$ is replaced by $v(-\infty)=m$, for any $m\in\R$. 
In this case, formula \eqref{approxpripofinaleq} is replaced by  
\beqs\label{approxpripofinaleq}
\left|\sum_{i=1}^{N_\ep}\ep \phi\left(\frac{x-x_i}{\ep\delta},b_i\right)-v(x)+\ep M_\ep\right|\leq o_\ep(1),
\eeqs
with $M_\ep= \lceil m/\ep\rceil$.
\end{rem}

\section{Supersolutions  of \eqref{uepeq}}\label{supersolutionssec}
In this section, we construct global is space and local in time supersolutions of the equation \eqref{uepeq}.
\begin{lem} \label{hepsupersolution}
Let $v$ be a $C^{1,1}$ function which  is  monotonic non-increasing in $(-\infty,x_0)$, non-decreasing in $(x_0,\infty)$, constant in $(x_0-\sigma,x_0+\sigma)$, for some $x_0\in\R$ and $0<\sigma<1$.   Let $x_i(t)$ be the solution of the ODE system  \eqref{levelsetode} with $L>0$ and  $x_i^0$  and $b_i$ defined as in \eqref{xi} and  \eqref{bidef}.
Then, there exists  $C>0$  depending on $v$ and $W$ such that  if $L=C/\sigma^\frac12$, 
 the function 
\beq \label{HEP}
H^\ep(t,x) = \sum_{i = 1}^{N_\ep}\ep \phi\left (\dfrac{x-x_i(t)}{\ep \delta}, b_i\right ) + \sum_{i = 1}^{N_\ep} \ep \delta \psi\left (\dfrac{x-x_i(t)}{\ep \delta}, b_i\right ) +\dfrac{\ep \delta L}{\alpha}
\eeq
is a supersolution of \eqref{uepeq} in $ (0,\sigma/(2c_0L)]\times \R$, where $\phi$ is the solution of \eqref{phi} and  $\psi$ is the solution of \eqref{psi}.
\end{lem}
\begin{proof}
To simplify the notation, we write $N = N_\ep$. Since $v$ is monotonic in $(-\infty,x_0)$ and  $(x_0,\infty)$,  the number of particles
in each of the intervals  is bounded by 
$(\sup v-\inf v)/\ep$ so that $\ep N\leq C$. 
Notice that the points $x_i^0$ such that $x_i^0<x_0$ are associated to $b_i=-1$ while points $x_i^0>x_0$ to $b_i=1$. Moreover,  since $v$ is constant 
in $(x_0-\sigma,x_0+\sigma)$ we have that 
\beq\label{nocolliinsigmac}x_{i+1}^0 - x_i^0 \geq2\sigma,\quad \text{if }b_i=-1,\,b_{i+1}=1.\eeq
Since $L>0$ the particles on the left of $x_0$ move to the right, while particles on the right of $x_0$ move to the left. However, since the particles with opposite sign are far enough, there is no collision for small times. More precisely, by the ODE system  \eqref{levelsetode}  and \eqref{nocolliinsigmac}, we see that for any $t \in [0,\sigma/(2c_0L)]$.
\beq \label{noncollision}
x_{i+1}(t) - x_i(t) \ge\sigma \quad \text{if }b_i=-1,\,b_{i+1}=1.
\eeq 
Particles with same orientation move in parallel  so they never collide.

Now, for a fixed $(t,x) \in [0,\sigma/(2c_0L)]\times \R,$ define 
$$\Lambda (t,x) = \delta \partial_t H^\ep (t,x) - \I[H^\ep(t,\cdot)](x) + \dfrac{1}{\delta}W'\left (\dfrac{H^\ep(t,x)}{\ep} \right ).$$
Define also $z_i(t) = \dfrac{x-x_i(t)}{\ep\delta}$. Let $i_0$ be the index such that $x_{i_0}(t)$ is the closest point to $x$. By direct computation, we obtain 
\beqs
\begin{split}
    \Lambda(t,x) 
&=- \sum_{i = 1}^N b_i\dot{x}_i \phi'(z_i,b_i) - \sum_{i = 1}^N\delta b_i\dot{x}_i \psi'(z_i,b_i)- \dfrac{1}{\delta}\sum_{i=1}^N \I[\phi](z_i,b_i)-\sum_{i=1}^N \I[\psi](z_i,b_i)\\
&+ \dfrac{1}{\delta} W' \left ( \sum_{i=1}^N [\tilde{\phi}(z_i,b_i) + \delta \psi(z_i,b_i) ] + \dfrac{\delta L}{\alpha} \right ),
\end{split}
\eeqs
where $\tilde{\phi}(z_i,b_i) := \phi(z_i,b_i) - H(z_i,b_i)$. Notice that by the periodicity of $W$, $W'(\phi(z_i,b_i))=W'(\tilde\phi(z_i,b_i))$. Using $\dot{x}_i = -c_0b_iL$ and $ \I[\phi](z_i,b_i) = W'(\tilde\phi(z_i,b_i))$, we have 
\beqs
\begin{split}
    \Lambda(t,x) 
&= \sum_{i = 1}^N c_0L \phi'(z_i,b_i) + \sum_{i = 1}^N\delta c_0L \psi'(z_i,b_i)- \dfrac{1}{\delta}\sum_{i=1}^N W'(\tilde\phi(z_i,b_i))-\sum_{i=1}^N  \I[\psi](z_i,b_i)\\
&+ \dfrac{1}{\delta} W' \left ( \sum_{i=1}^N [\tilde{\phi}(z_i,b_i) +  \delta \psi(z_i,b_i) ] + \dfrac{\delta L}{\alpha} \right )\\
&= c_0L\phi'(z_{i_0},b_{i_0}) - \dfrac{1}{\delta} W'(\tilde\phi(z_{i_0},b_{i_0})) -  \I[\psi](z_{i_0},b_{i_0}) - \dfrac{1}{\delta}\sum_{\stackrel{i=1}{i\not=i_0}}^NW'(\tilde\phi(z_i,b_i))\\
&+\dfrac{1}{\delta}W'\left(\tilde{\phi}(z_{i_0},b_{i_0}) + \sum_{\stackrel{i=1}{i\not=i_0}}^N\tilde{\phi}(z_i,b_i) + \sum_{i=1}^N \delta \psi(z_i,b_i) + \dfrac{\delta L}{\alpha} \right ) + E_0,
\end{split}
\eeqs
where 
\beqs
E_0 := \sum_{\stackrel{i=1}{i\not=i_0}}^N c_0L\phi'(z_i,b_i) + \sum_{i=1}^N \delta c_0L \psi'(z_i,b_i) - \sum_{\stackrel{i=1}{i\not=i_0}}^N  \I[\psi](z_i,b_i).
\eeqs
By Taylor expansion of $W'$ around $\tilde\phi(z_{i_0},b_{i_0})$, we can write
\beqs
\begin{split}
    \Lambda(t,x)
&= c_0L\phi'(z_{i_0},b_{i_0}) - \I[\psi](z_{i_0},b_{i_0}) - \dfrac{1}{\delta}\sum_{\stackrel{i=1}{i\not=i_0}}^NW'(\tilde\phi(z_i,b_i))\\
&+ \dfrac{1}{\delta}W''(\tilde{\phi}(z_{i_0},b_{i_0}))\left (\sum_{\stackrel{i=1}{i\not=i_0}}^N\tilde{\phi}(z_i,b_i) + \sum_{i=1}^N \delta \psi(z_i,b_i) + \dfrac{\delta L}{\alpha} \right ) + E_0 + E_1,
\end{split}
\eeqs
where 
\beqs
E_1 := O\left (\sum_{\stackrel{i=1}{i\not=i_0}}^N\tilde{\phi}(z_i,b_i) + \sum_{i=1}^N \delta \psi(z_i,b_i) + \dfrac{\delta L}{\alpha} \right )^2.
\eeqs
Now, by Taylor expansion of $W'$ around $0$ and using that $\alpha=W''(0)$,  we obtain 
\beqs
\begin{split}
    \Lambda(t,x)
&= c_0L\phi'(z_{i_0},b_{i_0}) -  \I[\psi](z_{i_0},b_{i_0}) + (W''(\tilde\phi(z_{i_0},b_{i_0}))-W''(0))\sum_{\stackrel{i=1}{i\not=i_0}}^N \dfrac{\tilde{\phi}(z_i,b_i)}{\delta} + L \\
&+ (W''(\tilde\phi(z_{i_0},b_{i_0}))-W''(0))\dfrac{L}{\alpha} + W''(\tilde\phi(z_{i_0},b_{i_0})) \sum_{i=1}^N  \psi(z_i,b_i) + E_0 + E_1 + E_2\\
&= c_0L\phi'(z_{i_0},b_{i_0}) -  \I[\psi](z_{i_0},b_{i_0}) + (W''(\tilde\phi(z_{i_0},b_{i_0}))-W''(0))\dfrac{L}{\alpha}  \\
&+ W''(\tilde\phi(z_{i_0},b_{i_0})) \psi(z_{i_0},b_{i_0})+W''(\tilde\phi(z_{i_0},b_{i_0})) \sum_{\stackrel{i=1}{i\not=i_0}}^N  \psi(z_i,b_i)\\
&+ (W''(\tilde\phi(z_{i_0},b_{i_0}))-W''(0))\sum_{\stackrel{i=1}{i\not=i_0}}^N \dfrac{\tilde{\phi}(z_i,b_i)}{\delta} + L + E_0 + E_1 + E_2,\\
\end{split}
\eeqs
where 
\beqs
E_2 :=  \sum_{\stackrel{i=1}{i\not=i_0}}^N O(\tilde{\phi}(z_i,b_i))^2.
\eeqs
By using equation \eqref{psi}, we obtain
\beq \label{lastcal}
\begin{split}
    \Lambda(t,x)
&= (W''(\tilde\phi(z_{i_0},b_{i_0}))-W''(0))\sum_{\stackrel{i=1}{i\not=i_0}}^N \dfrac{\tilde{\phi}(z_i,b_i)}{\delta} + L + E_0 + E_1 + E_2 + E_3,\\
\end{split}
\eeq
where 
\beqs
E_3 := W''(\tilde\phi(z_{i_0},b_{i_0})) \sum_{\stackrel{i=1}{i\not=i_0}}^N  \psi(z_i,b_i).
\eeqs
Similarly to Lemma 5.3 in \cite{patsan},  one can show using the estimates of Lemma \ref{phiinfinitylem} and Lemma \ref{psiinfinitylem} that
\beq\label{errorssubsollem}E_0 + E_1 + E_2 + E_3 = o_\ep(1).\eeq
Now, by \eqref{phiinfinity} and \eqref{i/k^2sum}, we get
\beq\label{sumphilemmsupe}
\left |\sum_{\stackrel{i=1}{i\not=i_0}}^N \dfrac{\tilde{\phi}(z_i,b_i)}{\delta} \right | \leq \dfrac{1}{\delta\alpha \pi} \left | \sum_{\stackrel{i=1}{i\not=i_0}}^N \dfrac{b_i \ep \delta}{x-x_i(t)} \right | + \dfrac{K_1}{\delta} \left | \sum_{\stackrel{i=1}{i\not=i_0}}^N \dfrac{\ep^2 \delta^2}{(x-x_i(t))^2} \right | \leq \dfrac{1}{\alpha \pi} \left | \sum_{\stackrel{i=1}{i\not=i_0}}^N \dfrac{b_i \ep 
}{x-x_i(t)} \right | + C\delta,
\eeq
for some constant $C>0$. Recall that $x_i(t)=x_{i}^0-c_0b_iLt$ and that 
$x_{i_0}(t)$ is the closest point to $x$ among the particles $x_i(t)$. 
Going back at time 0, it is easy to see that $x_{i_0}^0$ is the closest point to $x+c_0b_{i_0}Lt$. 
We write 
\beq\label{splitposnesublemma5}
\sum_{\stackrel{i=1}{i\not=i_0}}^N \dfrac{b_i \ep 
}{x-x_i(t)}=\sum_{\stackrel{i=1}{i\not=i_0, x_i^0>x_0}}^N \dfrac{ \ep 
}{(x+c_0Lt)-x_i^0}-\sum_{\stackrel{i=1}{i\not=i_0, x_i^0<x_0}}^N \dfrac{ \ep 
}{(x-c_0Lt)-x_i^0}.
\eeq
Assume  $x\geq x_0$ so that $b_{i_0}=1$. The case $x<x_0$ can be treated similarly. Denote $y(t):=x+c_0Lt$. Then, by 
  Lemma \ref{sumxialwaysfiniteprop}  we have that 
 \beqs  \left|\sum_{\stackrel{i=1}{i\not=i_0, x_i^0>x_0}}^N \dfrac{ \ep 
}{y(t)-x_i^0}\right|=\left|\sum_{\stackrel{i=1}{i\not=i_0 }}^N \dfrac{ \ep b_i
}{y(t)-x_i^0}- \sum_{\stackrel{i=1}{x_i^0<x_0}}^N \dfrac{ \ep b_i
}{y(t)-x_i^0}\right|\leq C+\left| \sum_{\stackrel{i=1}{x_i^0<x_0}}^N \dfrac{ \ep 
}{y(t)-x_i^0}\right|.
\eeqs
Notice that by \eqref{nocolliinsigmac} and  \eqref{noncollision},  if $x_i^0<x_0$ then $x_i^0<y(t)-\sigma$. 
Now, fix $\rho>\sigma$. By \eqref{proplippart1} the number of particles in $(y(t)-\rho,y(t)-\sigma)$ is bounded by $C\rho/\ep$. This, together with $N\ep\le C$ yields  
\beqs 
\left| \sum_{\stackrel{i=1}{x_i^0<x_0}}^N \dfrac{ \ep 
}{y(t)-x_i^0}\right|\leq \left| \sum_{\stackrel{i=1}{x_i^0\leq y(t)-\rho}}^N \dfrac{ \ep 
}{y(t)-x_i^0}\right|+\left| \sum_{\stackrel{i=1}{ y(t)-\rho<x_i^0<y(t)-\sigma}}^N \dfrac{ \ep 
}{y(t)-x_i^0}\right|\leq \frac{C}{\rho}+\frac{C\rho}{\sigma}.
\eeqs 
Choosing $\rho=\sigma^\frac12$ the two last estimates give
\beq\label{finalestilemmasupe1}
\left|\sum_{\stackrel{i=1}{i\not=i_0, x_i^0>x_0}}^N \dfrac{ \ep 
}{(x+c_0Lt)-x_i^0}\right|\leq \frac{C}{\sigma^\frac12}.
\eeq
If $x_i^0<x_0$, then  by \eqref{nocolliinsigmac} and  \eqref{noncollision},  $x_i^0<(x-c_0Lt)-\sigma/2$ for any $t\in [0,\sigma/(2c_0L)]$, so that similar computations as above yield
  \beq\label{finalestilemmasupe2}\left|\sum_{\stackrel{i=1}{i\not=i_0, x_i^0<x_0}}^N \dfrac{ \ep 
}{(x-c_0Lt)-x_i^0}\right|\leq  \frac{C}{\sigma^\frac12}.
\eeq
 Combining \eqref{lastcal}, \eqref{errorssubsollem},  \eqref{sumphilemmsupe}, \eqref{splitposnesublemma5}, \eqref{finalestilemmasupe1} and \eqref{finalestilemmasupe2}, we finally obtain

$$\Lambda(t,x) \geq -\frac{C}{\sigma^\frac12} + L \geq 0,$$
which implies that $H^\ep$ is a supersolution of \eqref{uepeq} by choosing $L=C/\sigma^\frac12$ with $C>0$ sufficiently large. 
\end{proof}


\section{Proof of Theorem \ref{mainthm}} \label{convergence}
In this section, we prove our main Theorem \ref{mainthm}. We  first show that the functions $u^\ep$ are bounded uniformly in $\ep$. Since $W'(z)=0$ for any $z\in\Z$, integers are stationary solutions to \eqref{uepeq}. Let $\lambda_1, \lambda_2\in \Z$ be such that $\lambda_1\leq \inf_\R u_0\leq\sup_\R u_0\leq \lambda_2$. Then by the comparison principle
we have that for any $\ep>0$
$$\lambda_1\leq u^\ep(t,x)\leq \lambda_2\quad\text{for all }(t,x)\in [0,+\infty)\times\R.$$
In particular, 
$u^+:=\limsup^*_{\ep\rightarrow0}u^\epsilon$ is everywhere finite.
We will prove that 
$u^+$ is a  viscosity  subsolution of \eqref{ubareq}. Similarly, we can prove that
$u^-:={\liminf_*}_{\ep\rightarrow0}u^\epsilon$ is a supersolution
of \eqref{ubareq}. We will then show that 
\beq\label{initialconditionlimit} u^+(0,x)\leq u_0(x)\leq u^-(0,x),\eeq
with $u_0$ the initial condition in  \eqref{ubareq}. The proof of \eqref{initialconditionlimit} is postponed to Section \ref{Initialconditionsection}. Then, if   $\us$ is the viscosity solution of \eqref{ubareq}, by the comparison principle, 
\beq\label{comparisonu+u-}u^+\leq \us\le u^-.\eeq
Since the reverse inequality $u^-\leq u^+$ always holds true, we
conclude that the two functions coincide with $\us$. This implies that $u^\ep\to \us$ as $\ep\to 0$, uniformly on compact sets on $[0,+\infty)\times \R$, see \cite[Remark 6.4]{usersguide}. 

\medskip 

Let us start by proving that $u^+$ is a  viscosity   subsolution of \eqref{ubareq}. Let $\eta$ be a smooth bounded function such that
\beq\label{inequalitycomparison}u^+(t,x)-\eta(t,x)<u^+(t_0,x_0)-\eta(t_0,x_0)=0\quad\text{for all } (t,x)\neq(t_0,x_0).\eeq

We separate the proof into two cases.

\subsection{Case 1: Test function $\eta$ with $\partial_x \eta (t_0,x_0) = 0$}\label{mainthmsubgrad0} \hfill\\
In this case, we want to show that

 \beq\label{mainresultcogl}\partial_t\eta (t_0,x_0) \leq 0.\eeq
Without loss of generality, we may  assume that  $\eta$ 
has the form 
\beq\label{etadefgrad0}\eta(t,x) = h(x) + g(t)\eeq with $g$ any smooth function and $h$ satisfying 
\beq\label{parabola}
\begin{cases}
h(x) = a(x-x_0)^2  \text{ for }|x-x_0|<\rho \\
h \text{ is non-increasing in } (-\infty,x_0)\\
h \text{ is non-decreasing in }  (x_0,+\infty)\\
\end{cases}\eeq 
for some $a,\rho>0$. 
Fix $\sigma>0$ such that $4\sigma<\rho$. 

We are going to construct a global in space  supersolution of  \eqref{uepeq} in an  interval around  $t_0$ 
 by using  Lemma \ref{hepsupersolution}. 
We cannot apply the lemma to the function $\eta$ as the required flat condition  is not satisfied by $\eta$.  
Therefore, we consider the function $\eta^\sigma(t,x):=h^\sigma(x)+g(t)$, where $h^\sigma(x)=\max\{ h(x), a(2\sigma)^2\}$. Notice that $\eta^\sigma\geq\eta$,
$\eta^\sigma\leq\eta+ a(2\sigma)^2$ and  $\eta^\sigma$
is constant in the interval $(x_0-2\sigma,x_0+2\sigma)$.  However, $\eta^\sigma$  is not of class  $C^{1,1}$.  To overcome this problem, we consider any $C^{1,1}$ function $\tilde\eta^\sigma$ such that 
$$\tilde\eta^\sigma(t,x)=\tilde h^\sigma(x)+g(t), $$ with $\tilde h^\sigma$ satisfying 

\beq\label{parabolabis}
\begin{cases}
\tilde h^\sigma\geq h^\sigma\geq h \\
\tilde h^\sigma(x)= h^\sigma(x)=a(2\sigma)^2 &\text{if }|x-x_0|\le \sigma\\
\tilde h^\sigma \text{ is  non-increasing in } (-\infty,x_0)\\
\tilde h^\sigma \text{ is  non-decreasing in }(x_0,+\infty).\\
\end{cases}\eeq 
For example, the function  $\tilde h^\sigma(x)$ defined for $|x-x_0|<\rho$ by
$$\tilde h^\sigma(x)=\begin{cases}2a(x-x_0-\sigma)^2+4a\sigma^2&\text{if }x>x_0+\sigma\\
4a\sigma^2&\text{if }|x-x_0|\le \sigma\\
2a(x-x_0+\sigma)^2+4a\sigma^2&\text{if }x<x_0-\sigma\\
\end{cases}$$
and extended to be monotonic, bounded  and above $h$ outside $ (x_0-\rho,x_0+\rho)$ would work. 
Moreover, since $u^\ep$ is bounded uniformly in $\ep$, without loss of generality we can assume that for all $\ep>0$ and  all $t>0$,
\beq\label{etatildegretuep1}\tilde\eta^\sigma(t,x)\geq u^\ep(t,x)\quad \text{ if  }|x-x_0|>1.\eeq
Finally,  since $u^+-\eta$ attains a strict maximum at $(t_0,x_0)$ and $\tilde\eta^\sigma\ge\eta$, for all $R>0$ there exists $\ep_0=\ep_0(R)$ such that for $\ep<\ep_0$
\beq\label{etatildegretuep2} u^\ep(t,x)-\tilde\eta^\sigma(t,x)<0\quad\text{for all }(t,x)\in Q_{1,1}(t_0,x_0)\setminus Q_{R,R}(t_0,x_0).\eeq
Next, we set  
 \beq\label{therightc}c:=\frac{ 1}{4c_0L},\eeq
with $L$ to be determined. 
Define $x_i^0$ and $b_i$, $i=1,\ldots, N_\ep$  as in \eqref{xi} and  \eqref{bidef} for the function $\tilde\eta^\sigma(t_0-c\sigma,\cdot)$. 
Since $\tilde\eta^\sigma(t,\cdot)$ is monotonic in  $(-\infty,x_0)$ and  $(x_0,\infty)$,  the number of particles in each interval is bounded by 
$(\sup \tilde\eta^\sigma-\inf \tilde\eta^\sigma)/\ep$ so that $\ep N_\ep\leq C$. In particular the condition $\ep^2 N_\ep\delta=o_\ep(1)$ is satisfied and we are in position to apply Proposition \ref{approxpropfinal} (recall Remark \ref{importantrem}) to get
\beq\label{initialconmaingradf0}\tilde\eta^\sigma(t_0-c\sigma,x)= \sum_{i = 1}^{N_\ep}\ep \phi\left (\dfrac{x-x_i^0}{\ep \delta}, b_i\right ) +\ep M_\ep+o_\ep(1),
\eeq
where $M_\ep:=\lceil \tilde\eta^\sigma(t_0-c\sigma, -\infty)/\ep\rceil$. By \eqref{etatildegretuep1}, \eqref{etatildegretuep2} with $R=c\sigma$, and \eqref{initialconmaingradf0} we also have
\beq\label{initialconmaingradf1}u^\ep(t_0-c\sigma,x)\leq  \sum_{i = 1}^{N_\ep}\ep \phi\left (\dfrac{x-x_i^0}{\ep \delta}, b_i\right ) +\ep M_\ep+o_\ep(1).\eeq
Let $x_i(t)$ be the solution of the ODE system  \eqref{levelsetode} with initial condition  $x_i(t_0-c\sigma)=x_i^0$, that is 
$$x_i(t)=x_i^0-b_ic_0L[t - (t_0 - c\sigma)].$$
Define $$H^\ep(t,x) := \sum_{i = 1}^{N_\ep}\ep \phi\left (\dfrac{x-x_i(t)}{\ep \delta}, b_i\right ) + \sum_{i = 1}^{N_\ep} \ep \delta \psi\left (\dfrac{x-x_i(t)}{\ep \delta}, b_i\right ) +\ep M_\ep+\dfrac{\ep \delta L}{\alpha}+\ep\left\lceil \frac{o_\ep(1)}{\ep}\right\rceil,$$
with $\phi$ and $\psi$ the solutions of \eqref{phi} and \eqref{psi} respectively. Notice that 
\beq\label{psiboundedporrfconveta0}\left|\sum_{i = 1}^{N_\ep}\ep \delta \psi\left(\dfrac{x-x_i(t)}{\ep \delta}, b_i\right )\right| \leq C\ep N_\ep\delta\leq C\delta=o_\ep(1).\eeq
By \eqref{initialconmaingradf1} and  \eqref{psiboundedporrfconveta0} we can choose $o_\ep(1)$ in the definition of $H^\ep$ such that 
$$H^\ep(t_0-c\sigma,x)\ge u^\ep(t_0-c\sigma,x).$$
 Now, by Lemma \ref{hepsupersolution}  if 
 \beq\label{L}L=\frac{C_0}{\sigma^\frac12}\eeq with $C_0$  large enough, the function  $H^\ep$
is supersolution of \eqref{uepeq} in $[t_0-c\sigma, t_0+c\sigma]\times \R$.
Therefore, by the comparison principle, we obtain 
\beq\label{comparsionproofconograd}H^\ep(t,x) \geq u^\ep(t,x)\quad\text{for any }(t,x) \in [t_0-c\sigma, t_0+c\sigma] \times \R.\eeq
Consider a sequence $(t_\ep,x_\ep)$ converging to $(t_0,x_0)$ as $\ep \to 0$. By \eqref{comparsionproofconograd} and \eqref{psiboundedporrfconveta0}  we have that
\beq\label{beforelastlem}
\begin{split}
    u^\ep(t_\ep,x_\ep) 
    &\leq H^\ep(t_\ep,x_\ep)\\
      &= \sum_{i=1}^{N_\ep} \ep \phi \left ( \dfrac{x_\ep-x_i(t_\ep)}{\ep \delta}, b_i \right ) +\ep M_\ep+ o_\ep(1)\\
    &= \sum_{i=1}^{N_\ep} \ep \phi \left ( \dfrac{(x_\ep +b_ic_0L(t_\ep - t_0 + c\sigma))- x_i^0}{\ep \delta}, b_i \right ) +\ep M_\ep+ o_\ep(1).\\
    \end{split}
\eeq
Next we use the following result. 
\begin{lem}\label{lastlem}
We have that, 
\beq\label{lastlemeq}
\begin{split}
&\sum_{i=1}^{N_\ep} \ep \phi \left ( \dfrac{(x_\ep +b_ic_0L(t_\ep - t_0 + c\sigma))- x_i^0}{\ep \delta}, b_i \right ) \\&
=\sum_{i=1}^{N_\ep} \ep \phi \left ( \dfrac{(x_\ep +c_0L(t_\ep - t_0 + c\sigma))- x_i^0}{\ep \delta}, b_i \right )+ o_\ep(1).
  \end{split}
\eeq
\end{lem}
We postpone the proof of Lemma \ref{lastlem} to Section \ref{lemmatasec}. 

Now, from \eqref{beforelastlem},  Lemma  \ref{lastlem},  Proposition \ref{approxpropfinal}, the definition of $x_i^0$ and using that $\tilde\eta^\sigma(t,x)\leq \eta(t,x)+C\sigma^2$ if $|x-x_0|\le\sigma$, 
 we infer that

\beqs
\begin{split}
    u^\ep(t_\ep,x_\ep) 
    &\leq \sum_{i=1}^{N_\ep} \ep \phi \left ( \dfrac{(x_\ep +b_ic_0L(t_\ep - t_0 + c\sigma))- x_i^0}{\ep \delta}, b_i \right )+\ep M_\ep+ o_\ep(1)\\
    &= \tilde\eta^\sigma(t_0-c\sigma, x_\ep + c_0L(t_\ep - t_0 + c\sigma)) + o_\ep(1)\\
    &\leq \eta(t_0-c\sigma, x_\ep + c_0L(t_\ep - t_0 + c\sigma)) + o_\ep(1)  + C\sigma^2.
\end{split}
\eeqs
 By passing to $\limsup^*$ as $\ep\to 0$, we obtain 
\beqs
u^+(t_0,x_0)  \leq \eta(t_0-c\sigma, x_0 + cc_0L\sigma) + C\sigma^2.
\eeqs
Since $u^+(t_0,x_0) = \eta(t_0,x_0)$, we also have 
\beqs
    \eta(t_0,x_0) - \eta(t_0-c\sigma,x_0) 
    \leq \eta(t_0-c\sigma, x_0 + cc_0L\sigma) - \eta(t_0-c\sigma,x_0) + C\sigma^2,
\eeqs
by subtracting $\eta(t_0-c\sigma,x_0)$ on both sides. 
Now, recalling the expression of $\eta$ with $h$ as in \eqref{parabola}, \eqref{therightc} and  \eqref{L}, we see that the inequality above yields
\beqs 
 \eta(t_0,x_0) - \eta(t_0-k_0\sigma^\frac32,x_0)\leq a\left(\frac{\sigma}{4}\right)^2 +C\sigma^2,
 \eeqs where $k_0:=1/(4c_0C_0)$. 
By dividing both sides by $k_0\sigma^\frac32$ and  taking the limit as $\sigma \to 0^+$, we finally get \eqref{mainresultcogl}.

\subsection{Case 2: Test function $\eta$ with $\partial_x \eta (t_0,x_0) \not = 0$} \hfill

Without loss of generality we assume that
\beq\label{deretacond}\partial_x\eta(t_0,x_0)>0.\eeq
The goal is  to show that 
\beq\label{goal}\partial_t\eta(t_0,x_0)\leq c_0 \partial_x\eta(t_0,x_0)\,\I[\eta(t_0,\cdot)](x_0).\eeq

We start with the following asymptotic result whose proof is postponed to Section \ref{additional}.
\begin{lem}\label{asymptou-u+}Let $v_1,\,v_2,\,w_1,\,w_2$ be defined as in \eqref{u_01} and \eqref{u_02}. Then, there exists $L>0$ such that for all $(t,x)\in(0,+\infty)\times\R$
\beqs\label{u_03}\max\{v_1(x+c_0Lt),w_1(x-c_0Lt)\}\leq u^-(t,x)\leq u^+(t,x)\leq \min\{v_2(x+c_0Lt),w_2(x-c_0Lt)\}.\eeqs
\end{lem}
Without loss of generality we may assume that $\eta$ satisfies 
\beq\label{etaasympt}\eta(t,x)=v_2(x+c_0Lt) \quad \text{if }x<-K,\qquad
 \eta(t,x)=w_2(x-c_0Lt)\quad \text{if }x>K,\eeq
 for $K$ large enough and $L>0$ given in
Lemma \ref{asymptou-u+}. 
Indeed, assume that  \eqref{goal} holds true for any test function satisfying \eqref{etaasympt}. If $\tilde\eta$ is any test function satisfying \eqref{inequalitycomparison}, by Lemma  \ref{asymptou-u+} we can always build a function $\eta$ such that 
$\eta=\tilde \eta$ in a neighborhood of $(t_0,x_0)$, $\eta\le \tilde \eta$, and $\eta$ satisfies \eqref{etaasympt}. By $\partial_t\eta(t_0,x_0)=\partial_t\tilde\eta(t_0,x_0)$, 
$\partial_x\eta(t_0,x_0)=\partial_x\tilde\eta(t_0,x_0)$ and $\I[\eta(t_0,\cdot)](x_0)\leq\I[\tilde\eta(t_0,\cdot)](x_0)$ and \eqref{goal} we infer that 
$$\partial_t\tilde\eta(t_0,x_0)\leq c_0 \partial_x\tilde\eta(t_0,x_0)\,\I[\tilde\eta(t_0,\cdot)](x_0),$$ as desired. 
Condition \eqref{etaasympt} implies that for any $T>0$ the points  $x_i=x_i(t)$ defined as in \eqref{xi} for the function $\eta(t,\cdot)$ with $t\in[0,T]$ belong to the set $[-K_\ep-c_0LT, K_\ep+c_0LT]$ with $K_\ep$ defined as in Section \ref{Aepsec}. Therefore, the number of such particles $N_\ep=N_\ep(t)$  satisfies 
$N_\ep\leq C (K_\ep+c_0LT)/\ep$. In particular,   by \eqref{Aepo0bound}, 
\beq\label{Nepbouneta}\ep^2N_\ep\delta=o_\ep(1).\eeq
This will allow us to apply Proposition \ref{approxpropfinal} to $v(x)=\eta(t,x)$ with $t$ close to $t_0$.

Next, the proof of \eqref{goal} is an adaptation of the proof given in \cite{patsan} in the monotonic case, therefore we will skip some details and refer to the corresponding results in 
 \cite{patsan}.
 
Suppose by contradiction that 
\beq\label{contradictionhp} \partial_t\eta(t_0,x_0)> c_0 \partial_x\eta(t_0,x_0)\,\I[\eta(t_0,\cdot)](x_0).\eeq 
Denote
$$L_0:=\I[\eta(t_0,\cdot)](x_0).$$
By \eqref{deretacond} and \eqref{contradictionhp}, there exist $0<\rho<1$ and $L_1>0$ such that 
\beq\label{etaxposirho}\partial_x\eta(t,x)\ge \frac{\partial_x\eta(t_0,x_0)}{2}>0\quad\text{for all  }(t,x)\in Q_{2\rho,2\rho}(t_0,x_0),\eeq
and 
\beq\label{contractconserho}\partial_t\eta(t,x)\ge c_0\partial_x\eta(t,x)(L_0+L_1)\quad\text{for all  }(t,x)\in Q_{2\rho,2\rho}(t_0,x_0).\eeq

Define $x_i^0 = x_i(t_0)$  and $b_i$, $i=1,\ldots, N_\ep$,   as in \eqref{xi} and  \eqref{bidef} for the function  $\eta$ at $t=t_0$. 
 For $0<R\ll \rho$ to be determined,  let  $x_{M_\rho}^0$  be the biggest point which is smaller than $x_0-(\rho+R)$, and 
  $x_{N_\rho}^0$  the lowest point bigger than $x_0+(\rho+R)$, that is
\beq \label{xm}
x_{M_\rho}^0 < x_0-(\rho+R) \leq x_{M_\rho+1}^0
\eeq
and 
\beq \label{xn}
x_{N_\rho-1}^0\leq x_0+(\rho+R)<x_{N_\rho}^0.
\eeq
 In other words, $ \{x_{M_\rho}^0, x_{M_\rho+1}^0,...,x_{N_\rho-1}^0,x_{N_\rho}^0\}$ are the particle points in the interval 
 $(x_0-(\rho+R), x_0+(\rho+R))$. 
By definition, there exists $J_0\in \{1,...,N_\ep\}$ such that $\eta(t_0,x_{M_\rho}^0) = J_0\ep,$ and since $\eta(t_0,\cdot)$ is increasing in  $(x_0-(\rho+R), x_0+(\rho+R))$, we have that 
$$\eta(t_0,x_{M_\rho+i}^0) = (i+J_0)\ep, \quad \text{for} \ i = 0,1,...,N_\rho-M_\rho:=K_\rho.$$
Define $B_0:=\partial_x\eta(t_0,x_0)/(2\|\partial_t \eta\|_\infty)$. Now, for any time $t$ such that $|t-t_0|<B_0R$, we define a set
\beq \label{Xi}
X_i(t) := \{x\in (x_0-(\rho+3R),x_0+(\rho+3R)) \ : \ \eta(t,x) = (i+J_0)\ep\},
\eeq 
for $i = 0,1,...,K_\rho$.

\begin{lem}\label{partcilescontrollem}
Let $B_0:=\partial_x\eta(t_0,x_0)/(2\|\partial_t \eta\|_\infty)$ and $X_i(t)$ be defined by  \eqref{Xi}, $i=0,\ldots, K_\rho$.
Then, there exists $\ep_0=\ep_0(\rho)$ such that for $\ep<\ep_0$ and $R<\rho/3$, $X_i(t)$ is a singleton, that is, $X_i(t) = \{\z^i(t)\}$, and $\z^i\in C^1(t_0-B_0R,t_0+B_0R)$ and for $|t-t_0|<B_0R$,
\beq\label{xdotbound}|\dot{\z}^i(t)|\le B_0^{-1},\eeq
\beq\label{x_ncontrolenq}x_0+\rho<\z^{K_\rho}(t)<x_0+\rho+3R,\eeq
\beq\label{x_mcontrolenq}x_0-(\rho+3R) <\z^{0}(t)<x_0-\rho.\eeq
In particular $(t,\z^i(t))\in Q_{2\rho,2\rho}(t_0,x_0)$.
\end{lem}
\begin{proof}
By the monotonicity of $\eta$, $X_i(t)$ is a singleton. The rest of the proof of Lemma \ref{partcilescontrollem} directly follows the proof of Lemma 5.1 in \cite{patsan}.
\end{proof}
Therefore, by choosing $R<\rho/3$, we have that $(t,\zeta^i(t)) \in Q_{2\rho,2\rho}(t_0,x_0)$ and 
\beq \label{mainzetaieq} \eta(t,\zeta^i(t)) = (i+J_0)\ep,\eeq 
for $i = 0,1,...,K_\rho.$ By Lemma \ref{partcilescontrollem}, $\zeta^i(t)$ is of class $C^1(t_0-B_0R, t_0+B_0R)$, allowing us to differentiate \eqref{mainzetaieq} in $t$, which yields
\beqs
\partial_t \eta(t,\z^i(t)) + \partial_x \eta(t,\z^i(t)) \dot{\z}^i(t) = 0.
\eeqs
Using \eqref{contractconserho}, for $|t-t_0|<B_0R$, we obtain 
\beq\label{dotzetaiimport}-\dot \z_i(t)\geq c_0(L_0+L_1),\quad i=0,1,\ldots,K_\rho.\eeq

Next, we will construct a supersolution of \eqref{uepeq} in $Q_{B_0R,R}(t_0,x_0)$ for $R\ll\rho< 1$. Since the maximum of $u^+-\eta$ is strict, there exists $\gamma_R>0$ such that 
\beq\label{u-etastrictmax} u^+-\eta\le- 2\gamma_R<0 \quad\text{in } Q_{2\rho,2\rho}(t_0,x_0)\setminus Q_{B_0R,R}(t_0,x_0).\eeq
Then, we define
\beq\label{Phiep}\Phi^\ep(t,x):=\begin{cases}h^\ep(t,x) +\frac{\ep\delta L_1}{\alpha}-\ep\left\lfloor \frac{\gamma_R}{\ep}\right \rfloor&\text{for }(t,x)\in Q_{B_0R,\frac\rho 2}(t_0,x_0)\\
u^\ep(t,x)& \text{outside}, \\
\end{cases}
\eeq
where
\beq\label{hfunct}\begin{split}
h^\ep(t,x)&=\sum_{i=0}^{K_\rho}\ep\left(\phi\left(\frac{x-\z^i(t)}{\ep \delta},1\right)+\delta\psi\left(\frac{x-\z^i(t)}{\ep \delta},1\right)\right)\\&+
\sum_{i=1}^{M_\rho-1}\ep\phi\left(\frac{x-x_i^0}{\ep \delta},b_i\right)+\sum_{i=N_\rho+1}^{N_\ep}\ep\phi\left(\frac{x-x_i^0}{\ep \delta},b_i\right)
\end{split}
\eeq
with  $\phi$ a solution the  \eqref{phi} and  $\psi$  the solution  of \eqref{psi} with $L=L_0+L_1$.

\begin{lem}\label{supersolutiponlemma}
There exist $0<R\ll\rho$ and $\ep_0=\ep_0(R,\rho)>0$ such that for any $\ep<\ep_0$,  the function $\Phi^\ep$ defined by \eqref{Phiep} satisfies
\beq\label{mainlem1}\Phi^\ep\ge u^\ep\quad\text{ outside }Q_{B_0R,R}(t_0,x_0),\eeq 
\beq\label{mainlem3}\Phi^\ep\le \eta+o_\ep(1)-\ep\left\lfloor \frac{\gamma_R}{\ep}\right \rfloor\quad\text{ in }Q_{B_0R,R}(t_0,x_0),\eeq
and 
\beq\label{mainlem2}\delta\partial_t \Phi^\ep\ge\I [\Phi^\ep]-\displaystyle\frac{1}{\delta}W'\left(\frac{\Phi_\ep}{\ep}\right)\quad\text{in }Q_{B_0R,R}(t_0,x_0),\eeq
\end{lem}

We are now in  position to conclude the proof for Case 2. By \eqref{mainlem1} and \eqref{mainlem2} and the comparison principle, Proposition \ref{comparisonbounded}, we have
$$u^\ep(t,x)\leq \Phi^\ep(t,x)\quad \text{for all }(t,x)\in Q_{B_0R,R}(t_0,x_0).$$
Passing to the  upper limit as $\ep\to0$ and using \eqref{mainlem3}  and  that $u^+(t_0,x_0)=\eta(t_0,x_0)$, we obtain
$$0\leq-\gamma_R,$$
which is a contradiction. This completes the proof of \eqref{goal}.

\noindent \textbf{Proof of Lemma \ref{supersolutiponlemma}.}
We divide the proof of Lemma \ref{supersolutiponlemma} in several steps. To prove \eqref{mainlem1} and \eqref{mainlem3}, we will need   the following lemma
whose proof is postponed   to Section \ref{lemmatasec}.

\begin{lem}\label{spproxetalem1} There exists $\ep_0=\ep_0(R,\rho)>0$ such that for any $\ep<\ep_0$ and for any $(t,x)\in Q_{B_0R, \rho-R}(t_0,x_0)$, we have
$$|h^\ep(t,x)-\eta(t,x)|\leq  o_\ep(1).$$ 
\end{lem}

\medskip
\noindent{\em Proof of \eqref{mainlem1}}.
By definition  \eqref{Phiep} of $\Phi^\ep$, $ \Phi^\ep(t,x)=u^\ep(t,x)$ outside of $Q_{B_0R,\frac\rho2}(t_0,x_0)$. Next, by Lemma \ref{spproxetalem1} and \eqref{u-etastrictmax}, for $(t,x)\in Q_{B_0R,\frac\rho2}(t_0,x_0)\setminus Q_{B_0R,R}(t_0,x_0)$, 
\beqs\begin{split} \Phi^\ep(t,x)&=h^\ep(t,x) +\frac{\ep\delta L_1}{\alpha}-\ep\left\lfloor \frac{\gamma_R}{\ep}\right \rfloor\\
&\ge \eta(t,x)+ o_\ep(1)-\ep\left\lfloor \frac{\gamma_R}{\ep}
\right \rfloor\\
&
\ge u^\ep(t,x).
\end{split}\eeqs
 This concludes the proof of  \eqref{mainlem1}.

\medskip
\noindent{\em Proof of \eqref{mainlem3}}. By Lemma  \ref{spproxetalem1}, for $(t,x)\in  Q_{B_0R,R}(t_0,x_0)$
$$\Phi^\ep(t,x)=h^\ep(t,x)+\frac{\ep\delta L_1}{\alpha}-\ep\left\lfloor \frac{\gamma_R}{\ep}\right \rfloor
\leq \eta(t,x)+o_\ep(1)-\ep\left\lfloor \frac{\gamma_R}{\ep}\right \rfloor,$$
which gives \eqref{mainlem3}.

\medskip
For $|x-x_0|\geq \rho-R$ we obtain a worse approximation result than the one in Lemma \ref{spproxetalem1} as shown below.  This is due to the fact that we have choosen 
the particles $x_i$ to be constant in time, equal to $x_i^0$, for $i<M_\rho$ and $i>N_\rho$.   

\begin{lem}\label{spproxetalem2} There exists $\ep_0=\ep_0(R,\rho)>0$ such that for any $\ep<\ep_0$, 
 if $|t-t_0|<B_0R$,  and $|x-x_0|\geq \rho-R$, then 
 $$|h^\ep(t,x)-\eta(t,x)|\leq o_\ep(1)+O(R).$$
\end{lem}
We postpone  the proof  of Lemma  \ref{spproxetalem2} to Section \ref{lemmatasec}.

\begin{cor} \label{coroIphiep}There exists $\ep_0=\ep_0(R,\rho)>0$ such that for any $\ep<\ep_0$, $R<\rho/4$, 
and any $(t,x)\in Q_{B_0R,R}(t_0,x_0)$, we have
\beq\label{eq2cori1uephfinal} \I[\Phi^\ep(t,\cdot)](x)\leq \I[h^\ep(t,\cdot)](x)+o_\ep(1)+\frac{o_R(1)}{\rho}.\eeq
\end{cor}
\begin{proof}
The corollary is a consequence of Lemma \ref{spproxetalem1}, Lemma \ref{spproxetalem2} and the definition  \eqref{Phiep} of $\Phi^\ep$.
For details, we refer to the proof of Corollary  5.7 in \cite{patsan}.
\end{proof}

Now, we are ready to prove \eqref{mainlem2}.

\noindent{\em Proof of \eqref{mainlem2}}.
Denote 
\beqs \Lambda := \delta\partial_t \Phi^\ep-\I [\Phi^\ep]+\frac{1}{\delta}W'\left(\frac{\Phi_\ep}{\ep}\right).
\eeqs
We want to show that $\Lambda(t,x) \geq 0$ for all $(t,x) \in Q_{B_0R,R}(t_0,x_0).$  
Fix $(\ts,\xs)\in Q_{B_0R,R}(t_0,x_0)$. Let $i_0$ be such that $\z^{i_0}(\ts)$ is the closest point to $\xs$. Then, $\xs = \z^{i_0}(\ts) + \ep \gamma$, with $|\gamma| \leq 2/|\partial_x\eta(t_0,x_0)|$ by \eqref{proplippart2} and \eqref{etaxposirho}. Define 
\beqs
z_i (t):= \dfrac{x-\z^i(t)}{\ep \delta}, \quad \quad z_i^0 := \dfrac{x-x^0_i}{\ep \delta} \quad \text{and} \quad \tilde{\phi}(z,b_i) := \phi(z,b_i) - H(z,b_i),
\eeqs with $H(z,b)$ defined as in \eqref{hbi}. 
Using Corollary \ref{coroIphiep}, equations \eqref{phi} and \eqref{psi}, performing Taylor expansions,  as in the proof of Lemma 5.3 in \cite{patsan}, we obtain 
\beq \begin{split}\label{Lambdafinal}
\Lambda(\ts,\xs) &\ge (W''(\phi(z_{i_0}),1) - W''(0))\left ( \frac{1}{\delta}\sum_{\stackrel{i=0}{i\not=i_0}}^{K_\rho} \tilde{\phi}(z_i,1) +  \frac{1}{\delta}\sum_{i=1}^{M_\rho-1}\tilde{\phi}(z_i^0,b_i) +  \frac{1}{\delta}\sum_{i=N_\rho+1}^{N_\ep} \tilde{\phi}(z_i^0,b_i) - \frac{L_0}{\alpha} \right )\\
&+L_1 + E_0 + E_1 + E_2 + E_3 + E_4, 
\end{split} \eeq
where
\beqs
\begin{split}
    E_0 &= o_\ep(1) + \dfrac{o_R(1)}{\rho}\\
    E_1 &= -\sum_{\stackrel{i=0}{i\not=i_0}}^{K_\rho} \dot{\z}^i(\ts)\phi'(z_i,1) - \delta \sum_{\stackrel{i=0}{i\not=i_0}}^{K_\rho}\dot{\z}^i(\ts) \psi'(z_i,1) -\delta\dot{\z}^{i_0}(\ts)\psi'(z_{i_0},1)\\
    E_2 &= \frac{1}{\delta}O\left (\sum_{\stackrel{i=0}{i\not=i_0}}^{K_\rho}[ \tilde{\phi}(z_i,1) + \delta \psi(z_i,1)]+ \delta \psi(z_{i_0},1) + \sum_{i=1}^{M_\rho-1}\tilde{\phi}(z_i^0,b_i) + \sum_{i=N_\rho+1}^{N_\ep} \tilde{\phi}(z_i^0,b_i) + \frac{\delta L_1}{\alpha} \right )^2\\
    E_3 &= \frac{1}{\delta}\sum_{\stackrel{i=0}{i\not=i_0}}^{K_\rho} O(\tilde{\phi}(z_i,1))^2 + \frac{1}{\delta}\sum_{i=1}^{M_\rho-1} O(\tilde{\phi}(z_i^0,b_i))^2 + \frac{1}{\delta}\sum_{i=N_\rho+1}^{N_\ep} O(\tilde{\phi}(z_i^0,b_i))^2\\
    E_4 &= W''(\tilde{\phi}(z_{i_0},1)) \sum_{\stackrel{i=0}{i\not=i_0}}^{K_\rho} \psi(z_i,1) -\sum_{\stackrel{i=0}{i\not=i_0}}^{K_\rho} \I[\psi](z_i,1).
\end{split}
\eeqs

We have estimates for the error terms $E_1, E_2, E_3$ and $E_4$ as stated in the following lemma.
\begin{lem} \label{errorcontrol} For $i\ge 1$, the error $E_i$ defined as above satisfies 
\beqs 
E_i= O(\delta).
\eeqs
\end{lem}
\begin{proof}
The proof follows directly the proof of Lemma 5.9 in \cite{patsan}.
\end{proof}
Furthermore, we claim the following.
\begin{lem} \label{estclaims} 
\beq\label{estclaims1}
(W''(\phi(z_{i_0},1)) - W''(0))\left ( \frac{1}{\delta}\sum_{\stackrel{i=0}{i\not=i_0}}^{K_\rho} \tilde{\phi}(z_i,1) - \frac{ 1}{\alpha}\I^{1,\rho}[\eta(t_0,\cdot)](x_0)\right ) = o_\ep(1)+ o_R(1) + o_\rho(1) + O\left ( \frac{R}{\rho} \right ) ,
\eeq
and
\beq\label{estclaims2}
\frac{1}{\delta}\sum_{i=1}^{M_\rho-1}\tilde{\phi}(z_i^0,b_i) +  \frac{1}{\delta}\sum_{i=N_\rho+1}^{N_\ep} \tilde{\phi}(z_i^0,b_i) 
-   \frac{ 1}{\alpha}\I^{2,\rho}[\eta(t_0,\cdot)](x_0) =  o_\ep(1)  + o_\rho(1) + O\left ( \frac{R}{\rho} \right ) .
\eeq
\end{lem}
\begin{proof}
By the monotonicity of $\eta$ in $Q_{B_0R,R}(t_0,x_0)$, the proof of \eqref{estclaims1} directly follows the proof of Lemma 5.8 in \cite{patsan}. 
With a slight modification of the proof of (5.32) in Lemma 5.8 in \cite{patsan}, using \eqref{i/k^2sum} and Lemma \ref{approxIlonglem} presented in this paper, we now show the estimate \eqref{estclaims2}. By  \eqref{phiinfinity} and \eqref{i/k^2sum}, we have that
\beq \begin{split} \label{627}
\frac{1}{\delta}\sum_{i=1}^{M_\rho-1}\tilde{\phi}(z_i^0,b_i) +  \frac{1}{\delta}\sum_{i=N_\rho+1}^{N_\ep} \tilde{\phi}(z_i^0,b_i) 
&\leq \frac{1}{\alpha \pi} \sum_{i=1}^{M_\rho-1} \frac{\ep b_i }{x_i^0-x} + K_1\sum_{i=1}^{M_\rho-1} \frac{ \ep^2 \delta}{(x_i^0-x)^2}\\
&+\frac{1}{\alpha \pi} \sum_{i=N_\rho+1}^{N_\ep} \frac{\ep b_i }{x_i^0-x} + K_1\sum_{i=N_\rho+1}^{N_\ep} \frac{ \ep^2 \delta}{(x_i^0-x)^2}\\
&\leq \frac{1}{\alpha \pi} \sum_{i=1}^{M_\rho-1} \frac{\ep b_i }{x_i^0-x} + \frac{1}{\alpha \pi} \sum_{i=N_\rho+1}^{N_\ep} \frac{\ep b_i }{x_i^0-x} + O(\delta).
\end{split} \eeq
Since $|x-x_0|<R$ and $|x_i^0-x_0|> \rho+R$ for $i\le M_\rho-1$ and $i\geq N_\rho+1$, we have that for those indices  $|x-x_i^0| \geq\rho$. However, there may be particles 
$x_i^0$ with $i=M_\rho,\ldots, N_\rho$ for which  $|x-x_i^0| \geq\rho$. Therefore, we can write
\beq \label{628}
\frac{1}{\alpha \pi} \sum_{i=1}^{M_\rho-1} \frac{\ep b_i }{x_i^0-x} + \frac{1}{\alpha \pi} \sum_{i=N_\rho+1}^{N_\ep} \frac{\ep b_i }{x_i^0-x} 
= \frac{1}{\alpha \pi} \sum_{\stackrel{i=M_\ep}{|x-x_i^0| \geq \rho}}^{N_\ep} \frac{\ep b_i }{x_i^0-x} - \frac{1}{\alpha \pi} \sum_{\stackrel{i=M_\rho}{|x-x_i^0| \geq\rho}}^{N_\rho} \frac{\ep b_i }{x_i^0-x}.
\eeq
Notice that, by \eqref{proplippart1}, \eqref{xm},  \eqref{xn}, and $|x-x_0|<R$, the number of particles $x_i^0$,  $i=M_\rho,\ldots,N_\rho$ such that $|x-x_i^0| \geq\rho$ is bounded by
$CR/\ep$. Therefore, 
\beqs
\left |\sum_{\stackrel{i=M_\rho}{|x-x_i^0| \geq\rho}}^{N_\rho} \frac{\ep b_i }{x_i^0-x} \right |\leq \dfrac{CR}{\rho}.
\eeqs
By  Lemma \ref{approxIlonglem} and the estimate above, then \eqref{628} becomes 
\beq \label{term2} \begin{split}
&\frac{1}{\alpha \pi} \sum_{i=1}^{M_\rho-1} \frac{\ep b_i }{x_i^0-x} + \frac{1}{\alpha \pi} \sum_{i=N_\rho+1}^{N_\ep} \frac{\ep b_i }{x_i^0-x}=\I^{2,\rho}[\eta(t_0,\cdot)](x_0) + o_\ep(1) + o_\rho(1)+O\left(\frac{R}{\rho}\right).
\end{split} \eeq
Combining \eqref{627} and \eqref{term2} yields 
$$\frac{1}{\delta}\sum_{i=1}^{M_\rho-1}\tilde{\phi}(z_i^0,b_i) +  \frac{1}{\delta}\sum_{i=N_\rho+1}^{N_\ep} \tilde{\phi}(z_i^0,b_i) 
-   \frac{ 1}{\alpha}\I^{2,\rho}[\eta(t_0,\cdot)](x_0) \leq  o_\ep(1)  + o_\rho(1) + O\left ( \frac{R}{\rho} \right ).
$$
Similarly, one can prove the opposite inequality. This proves \eqref{estclaims2}.

\end{proof}
As a consequence of Lemma \ref{errorcontrol} and Lemma \ref{estclaims}, the inequality \eqref{Lambdafinal} becomes 
\beqs
\Lambda(\ts,\xs) \geq L_1 + o_\ep(1)+  o_R(1) + o_\rho(1) + \frac{o_R(1)}{\rho}.
\eeqs
We choose $R\ll \rho\ll1$ and $\ep_0$ so small that for any $\ep<\ep_0$,
$$ \left|o_\ep(1) +  o_R(1) + o_\rho(1) + \frac{o_R(1)}{\rho}\right|<\frac{L_1}{2}.$$Then, 
$$\Lambda(\ts,\xs) >\frac{L_1}{2}>0.$$
This completes the proof of \eqref{mainlem2}.

\section{Proof of \eqref{initialconditionlimit}}\label{Initialconditionsection}
To prove \eqref{initialconditionlimit} we are going to build supersolutions of \eqref{uepeq} for small times to compare to $u^\ep$.
Fix any point $x_0\in\R$. Since $u_0$ is a $C^{1,1}$  function, there exists a parabola $a(x-y_0)^2+b$ touching from above $u_0$ at $x_0$, for some $y_0,b\in\R$ and $a>0$.
Since $u_0$ is bounded, there exists a bounded smooth  function $g$ touching $u_0$ from above such that 
\beqs\begin{cases}
g\ge u_0,\quad g(x_0)=u_0(x_0)\\
g=a(x-y_0)^2+b\quad\text{in }(x_0-1,x_0+1)\\
g\text{ is non-increasing in }(-\infty,y_0)\\
g\text{ is non-decreasing in }(y_0,+\infty).
\end{cases}
\eeqs
 Finally, following the construction of Section  \ref{mainthmsubgrad0}, for $\sigma>0$ small enough it is easy to see that there exists a $C^{1,1}$ function $g_\sigma$ such that 
\beqs\begin{cases}
g_\sigma\ge g,\quad g_\sigma(x_0)\to g(x_0)\text{ as }\sigma\to0\\
g_\sigma \quad\text{is constant  in }(y_0-\sigma,y_0+\sigma)\\
g_\sigma\text{ is non-increasing in }(-\infty,y_0)\\
g_\sigma\text{ is non-decreasing in }(y_0,+\infty).
\end{cases}
\eeqs
Let $x_i^0$ and $b_i$, $i=1,\ldots,N_\ep$  be defined as in \eqref{xi} and  \eqref{bidef} for the function $g_\sigma$. 

Let $M_\ep:=\lceil g_\sigma(-\infty)/\ep\rceil$. Then, by Lemma \ref{hepsupersolution} 
there exists $L= C/\sigma^\frac12$ such that if $x_i(t)$ is the solution of the ODE system  \eqref{levelsetode} with $x_i(0)=x_i^0$, 
the function 
\beqs 
H^\ep(t,x) = \sum_{i = 1}^{N_\ep}\ep \phi\left (\dfrac{x-x_i(t)}{\ep \delta}, b_i\right ) + \sum_{i = 1}^{N_\ep} \ep \delta \psi\left (\dfrac{x-x_i(t)}{\ep \delta}, b_i\right ) +\ep M_\ep+
\ep\left\lceil \frac{o_\ep(1)}{\ep}\right\rceil+\dfrac{\ep \delta L}{\alpha}
\eeqs
is supersolution of \eqref{uepeq} in $(0,\sigma/(2c_0L)]\times \R$.
Since $u^\ep(0,x)=u_0(x)\le g_\sigma(x)$, by Proposition \ref{approxpropfinal} (recall Remark \ref{importantrem}) and the fact that 
$$ \sum_{i = 1}^{N_\ep} \ep \delta \psi\left (\dfrac{x-x_i(t)}{\ep \delta}, b_i\right ) =o_\ep(1),$$
we can choose $o_\ep(1)$ in the definition of $H^\ep$ such that $u^\ep(0,x)\le H^\ep(0,x)$. Then, by the comparison principle,
$$u^\ep(t,x)\le H^\ep(t,x)\quad\text{for all }(t,x)\in (0,\sigma/(2c_0L)]\times \R.$$
Consider any sequence $(t_\ep,x_\ep)$ converging to $(0,x_0)$ as $\ep \to 0$. As in the proof of Theorem \ref{mainthm}, we have 
\beqs
\begin{split}
    u^\ep(t_\ep,x_\ep) 
&\leq H^\ep(t_\ep,x_\ep)\\
&= \sum_{i = 1}^{N_\ep}\ep \phi\left (\dfrac{x_\ep-x_i(t)}{\ep \delta}, b_i\right )  +\ep M_\ep+o_\ep(1)\\
&= \sum_{i = 1}^{N_\ep} \ep \phi \left ( \dfrac{(x_\ep + b_i c_0Lt_\ep)-x_i^0}{\ep \delta} ,b_i \right )+\ep M_\ep + o_\ep(1)\\
&= \sum_{i = 1}^{N_\ep} \ep \phi \left ( \dfrac{(x_\ep +  c_0Lt_\ep)-x_i^0}{\ep \delta} ,b_i \right ) +\ep M_\ep+ o_\ep(1)\\
&= g^\sigma(x_\ep + c_0L t_\ep) + o_\ep(1).\\
\end{split}
\eeqs
Passing to the $\limsup^*$ we get 
$$u^+(0,x_0)\le  g^\sigma(x_0).$$
Finally, letting $\sigma\to0$ and using that $g^\sigma(x_0)\to g(x_0)=u_0(x_0)$ as $\sigma\to0$, we get 
$$u^+(0,x_0)\leq u_0(x_0)$$ as desired. 

\section{Asymptotic behavior of the limit function: proof of Proposition \ref{uasymptoticprop}} \label{additional}
 
 In this section, we investigate  the  asymptotic behavior  as $x\to \pm \infty$ of the limit function $\us$. 
 We first prove Lemma \ref{asymptou-u+} and then Proposition \ref{uasymptoticprop}.
 
The following result is proven in \cite[Section 6]{patsan}.
\begin{lem}\label{supermonotincr} 
Let $v_0$ be a $C^{1,1}$ non-decreasing function and  let $v^\ep$ be the solution of \eqref{uepeq} with $v^\ep(0,x)=v_0(x)$. Then, there exists $L>0$ independent of $\ep$ such that for all $(t,x)\in (0,+\infty)\times\R$, 
$$v_0(x-c_0Lt)+o_\ep(1)\leq v^\ep(t,x)\leq v_0(x+c_0Lt)+o_\ep(1).
$$
\end{lem}
Similarly, one can prove 
\begin{lem}\label{supermonotincr2} 
Let $w_0$ be a $C^{1,1}$ non-increasing function and  let $w^\ep$ be the solution of \eqref{uepeq} with $w^\ep(0,x)=w_0(x)$. Then, there exists $L>0$ independent of $\ep$ such that for all $(t,x)\in  (0,+\infty)\times\R$, 
$$w_0(x+c_0Lt)+o_\ep(1)\leq w^\ep(t,x)\leq w_0(x-c_0Lt)+o_\ep(1).
$$
\end{lem}
\subsection{Proof of Lemma \ref{asymptou-u+}}
Let $v_2$ and $w_2$ defined as in \eqref{u_02}. Let $v^\ep$ be the solution of \eqref{uepeq} with initial condition $v^\ep(0,x)=v_2(x)$ and let $w^\ep$ be the solution of \eqref{uepeq} with initial condition $w^\ep(0,x)=w_2(x)$. By the comparison principle, $u^\ep(t,x)\leq v^\ep(t,x)$ and $u^\ep(t,x)\leq w^\ep(t,x)$  for all $(t,x)\in(0,+\infty)\times\R$.
The inequality 
$$ u^-(t,x)\leq u^+(t,x)\leq \min\{v_2(x+c_0Lt),w_2(x-c_0Lt)\}$$ then follows 
from Lemma \ref{supermonotincr} and Lemma \ref{supermonotincr2}.
Similarly, one can prove that $$u^-(t,x)\ge \max\{v_1(x+c_0Lt),w_1(x-c_0Lt)\},$$ and this concluded the proof of the lemma.
\subsection{Proof of   Proposition \ref{uasymptoticprop}}
By Theorem \ref{mainthm} we know that $u^-=u^+=\us$ with $\us$ the solution of \eqref{ubareq}. Then, the limits in \eqref{ulimits} immediately follow from  Lemma \ref{asymptou-u+}.
Finally, estimate \eqref{ulesssupinfu0} is a consequence of the comparison principle and the fact that constants are solutions to the equation $\partial_t u=c_0|\partial_x u|\,\I[ u]$.


\section{Appendix} \label{lemmatasec}

\begin{lem}\label{psismall}
There exists $C>0$ independent of $\ep$ and $\rho$ such that, for any $x\in\R$, $$\left|\sum_{i=0}^{K_\rho}\ep \delta\psi\left(\frac{x-\z^i(t)}{\ep \delta},1\right)\right|\leq C\delta.$$
\end{lem}
\begin{proof}
Using \eqref{mainzetaieq} and $\|\psi\|_\infty \leq C$ for some $C>0$, we have
\beqs \begin{split}\left|\sum_{i=0}^{K_\rho}\ep \delta\psi\left(\frac{x-\z^i(t)}{\ep \delta},1\right)\right|&\le \delta \|\psi\|_\infty \ep (K_\rho+1)\\
&= \delta \|\psi\|_\infty \ep (K_\rho + J_0 -J_0+1)\\
&=\delta \|\psi\|_\infty (\eta(t,\z^{K_\rho}(t))- \eta(t,\z^0(t))+\ep)\\
&\leq  C\delta.
\end{split}
\eeqs

\end{proof}

\subsection{ Proof of Lemma \ref{spproxetalem1}} 

To prove the lemma, we will show the following claims.

\medskip

\noindent{\em Claim 1: $\left|\sum_{i=0}^{K_\rho}\ep\phi\left(\frac{x-\z^i(t)}{\ep \delta},1\right)+\ep J_0-\eta(t,x)\right|\leq o_\ep(1) + \frac{C\ep^2\delta N_\ep}{R}.$}

{\em Proof of Claim 1.} If $(t,x)\in Q_{B_0R,\rho-R}(t_0,x_0)$, then by Lemma \ref{partcilescontrollem} 
$x\in (\z^0(t)+R, \z^{K_\rho}(t)-R)$. Then, Claim 1 follows from  Lemma \ref{vapproxphisteps} and  the fact that $\eta(t,\z^0(t)) = J_0\ep $.

\medskip

\noindent{\em Claim 2: $\left| \sum_{i=1}^{M_\rho-1}\ep\phi\left(\frac{x-x_i^0}{\ep \delta},b_i\right)-J_0\ep \right| \le 
o_\ep(1) + \frac{C\ep^2\delta N_\ep}{R}.$}

{\em Proof of Claim 2.} By using \eqref{xm}, if $(t,x)\in Q_{B_0R,\rho-R}(t_0,x_0)$, then 
$x>x^0_{M_\rho-1}+R$. Claim 2 then follows from \eqref{vapprowhatremainseq1} and the fact that $\eta(t_0,x^0_1)=\ep$ and   $\eta(t_0,x^0_{M_\rho-1}) = J_0\ep - \ep$.

\medskip
\noindent{\em Claim 3:   $\left |\sum_{i=N_\rho+1}^{N_\ep}\ep\phi\left(\frac{x-x_i^0}{\ep \delta},b_i\right) \right |\le  o_\ep(1) + \frac{C\ep^2\delta N_\ep}{R}.$}

{\em Proof of Claim 3.} By using \eqref{xn}, if $(t,x)\in Q_{B_0R,\rho-R}(t_0,x_0)$, then 
$x<x^0_{N_\rho+1}-R$. Claim 3 then immediately follows from \eqref{vapprowhatremainseq2}.

\medskip

Finally, the lemma is a consequence of Claims 1-3,  Lemma \ref{psismall} and \eqref{Nepbouneta}.


\subsection{ Proof of Lemma \ref{spproxetalem2}} 

We first consider $|x-x_0|>\rho+4R$. Let us assume that $x>x_0+\rho+4R$. One can similarly prove the lemma for $x<x_0-(\rho+4R)$. We divide the proof into three claims as follows. 

\medskip

\noindent{\em Claim 1: $\left|\sum_{i=0}^{K_\rho}\ep\phi\left(\frac{x-\z^i(t)}{\ep \delta},1\right)-\ep K_\rho \right|\leq o_\ep(1) + \dfrac{C\ep^2\delta N_\ep}{R}.$}

{\em Proof of Claim 1.} If $|t-t_0|<B_0R$ and $x>x_0+\rho+4R$, then  by Lemma \ref{partcilescontrollem} 
$x> \z^{K_\rho}(t)+R$.  Therefore, Claim 1 follows immediately by \eqref{vapprowhatremainseq1} and the fact that $\eta(t, \z^{K_\rho}(t))-\eta(t, \z^{0}(t))=\ep K_\rho$.

\medskip

\noindent{\em Claim 2: $\left| \sum_{i=1}^{M_\rho-1}\ep\phi\left(\frac{x-x_i^0}{\ep \delta},b_i\right)-\ep J_0 \right| \le 
o_\ep(1) + \dfrac{C\ep^2\delta N_\ep}{R}.$}

{\em Proof of Claim 2.}  By \eqref{xm}, if  $x>x_0+\rho+4R$, then 
$x>x^0_{M_\rho}+R$. Claim 2 then follows from \eqref{vapprowhatremainseq1} and the fact that $ \eta(t_0, x^0_{M_\rho-1})= J_0\ep -\ep$.

\medskip
\noindent{\em Claim 3:   $\left|\sum_{i=N_\rho+1}^{N_\ep}\ep\phi\left(\frac{x-x_i^0}{\ep \delta},b_i\right)+\ep (K_\rho+J_0)-\eta(t,x)\right|\le  
o_\ep(1) + \dfrac{C\ep^2\delta N_\ep}{R}+O(R).$}

{\em Proof of Claim 3.}  By \eqref{xn}, if  $x>x_0+\rho+4R$ and in addition $x<x^0_{N_\ep}-R$,  then 
$x\in (x^0_{N_\rho}+R, x^0_{N_\ep}-R)$. By Lemma \ref{vapproxphisteps} and the fact that $\eta(t_0,x^0_{N_\rho+1})=\ep (K_\rho+J_0+1)$, we obtain
\beqs
\begin{split}
    &\left|\sum_{i=N_\rho+1}^{N_\ep}\ep\phi\left(\frac{x-x_i^0}{\ep \delta},b_i\right) +\ep (K_\rho+J_0)-\eta(t,x)\right |\\ &\leq \left|\sum_{i=N_\rho+1}^{N_\ep}\ep\phi\left(\frac{x-x_i^0}{\ep \delta},b_i\right) +\ep (K_\rho+J_0)-\eta(t_0,x)\right | + |\eta(t_0,x) - \eta(t,x)|\\
    &\leq o_\ep(1) + \dfrac{C\ep^2\delta N_\ep}{R}+O(R),
\end{split}
\eeqs
using $|\eta(t_0,x) - \eta(t,x)|\leq O(R).$ This proves Claim 3 with $x<x_{N_\ep}^0-R$.

Next, suppose that $x>x_{N_\ep}^0+R$. In this case, we apply \eqref{vapprowhatremainseq1} to obtain
\beqs
\begin{split}
    &\left|\sum_{i=N_\rho+1}^{N_\ep}\ep\phi\left(\frac{x-x_i^0}{\ep \delta},b_i\right) +\ep (K_\rho + J_0)-\eta(t,x)\right |\\ &\leq \left|\sum_{i=N_\rho+1}^{N_\ep}\ep\phi\left(\frac{x-x_i^0}{\ep \delta},b_i\right) +\ep (K_\rho +  J_0)-\eta(t_0,x_{N_\ep}^0)\right | + |\eta(t_0,x_{N_\ep}^0) - \eta(t,x)|\\
    &\leq o_\ep(1) + \dfrac{C\ep^2\delta N_\ep}{R}+O(R),
\end{split}
\eeqs
where, in the last inequality, we used that
\beq \label{eta_reg}
|\eta(t_0,x_{N_\ep}^0) - \eta(t,x)| \leq |\eta(t_0,x_{N_\ep}^0) - \eta(t_0,x)| + |\eta(t_0,x) - \eta(t,x)| \leq \ep + O(R).
\eeq

Finally, suppose $x_{N_\ep}^0 - R \leq x \leq x_{N_\ep}^0+R$. Define $N$ to be an index such that 
\beqs
x^0_{N}\leq x^0_{N_\ep}-2R<x^0_{N+1}\leq x^0_{N_\ep}.
\eeqs
We have 
\beqs\begin{split}
&\left|\sum_{i=N_\rho+1}^{N_\ep}\ep\phi\left(\frac{x-x_i^0}{\ep \delta},b_i\right)+\ep (K_\rho+J_0)-\eta(t,x)\right|\\
&\leq \left | \sum_{i=N_\rho+1}^{N}\ep \phi\left(\frac{x-x_i^0}{\ep\delta},b_i\right) +\ep (K_\rho+J_0)-\eta(t_0,x_{N}^0)\right|\\
&+\left | \sum_{i=N+1}^{N_\ep}\ep \phi\left(\frac{x-x_i^0}{\ep\delta},b_i\right) \right | + o_\ep(1) + O(R).
\end{split}\eeqs
By \eqref{vapprowhatremainseq1} 
\beqs \begin{split}
\left | \sum_{i=N_\rho+1}^{N}\ep \phi\left(\frac{x-x_i^0}{\ep\delta},b_i\right) +\ep (K_\rho+J_0)-\eta(t_0,x_{N}^0)\right|\leq o_\ep(1) + \dfrac{C\ep^2 \delta N_\ep}{R}.
\end{split}
\eeqs 
By using that  $0<\phi< 1$ and that $\{x_{N+1},\ldots, x_{N_\ep}\}\subset (x_{N_\ep}-2R,x_{N_\ep})$ so that $|\{x_{N+1},\ldots, x_{N_\ep}\}|\le CR/\ep$, 
 we have
\beqs \begin{split}
    &\left | \sum_{i=N+1}^{N_\ep}\ep \phi\left(\frac{x-x_i^0}{\ep\delta},b_i\right) \right | \le  O(R).\\
   \end{split}
\eeqs
 This concludes the proof of Claim 3.

The lemma for $|x-x_0|>\rho+4R$ follows as a consequence of Claims 1-3,  Lemma \ref{psismall}  and \eqref{Nepbouneta}.

Finally, let us consider the case $\rho-R\leq |x-x_0|\leq \rho+4R.$ Assume without loss of generality that $\rho-R\le x-x_0\le \rho+4R$. We will divide the proof into three claims.

\medskip
\noindent{\em Claim 4:   $\left|\sum_{i=1}^{M_\rho}\ep\phi\left(\frac{x-x_i^0}{\ep \delta},b_i\right)-J_0\ep \right|\le  
o_\ep(1) + \dfrac{C\ep^2\delta N_\ep}{R}.$}

{\em Proof of Claim 4.} By \eqref{xm} and $x_0+\rho-R<x$, we have that $x>x_{M_\rho}^0 +R$. Therefore, using \eqref{vapprowhatremainseq1}, the claim immediately follows.

\medskip
\noindent{\em Claim 5:   $\left|\sum_{i=N_\rho+1}^{N_\ep}\ep\phi\left(\frac{x-x_i^0}{\ep \delta},b_i\right) \right|\le  
o_\ep(1) + \dfrac{C\ep^2\delta N_\ep}{R}+O(R).$}

{\em Proof of Claim 5.} Define an index $N_1$ such that 
\beqs
x_0+\rho + 5R \leq x^0_{N_1} < x^0_{N\rho}+6R,
\eeqs
so that   $x<x_0+\rho+4R \leq x^0_{N_1}-R$. By using \eqref{vapprowhatremainseq2}, $0<\phi< 1$ and 
$|\{x_{N_\rho+1},\ldots,x_{N_1-1}\}|\leq CR/\ep$, we obtain 
\beqs
\begin{split}
   \left|\sum_{i=N_\rho+1}^{N_\ep}\ep\phi\left(\frac{x-x_i^0}{\ep \delta},b_i\right) \right| 
   &\leq  \left|\sum_{i=N_\rho+1}^{N_1-1}\ep\phi\left(\frac{x-x_i^0}{\ep \delta},b_i\right)\right|
   + \left|\sum_{i=N_1}^{N_\ep}\ep\phi\left(\frac{x-x_i^0}{\ep \delta},b_i\right)\right|\\
      &= O(R) + o_\ep(1) + \dfrac{C\ep^2\delta N_\ep}{R}.
\end{split}
\eeqs
This completes the proof of Claim 5.

\medskip
\noindent{\em Claim 6:   $\left|\sum_{i=0}^{K_\rho}\ep\phi\left(\frac{x-\z^i(t)}{\ep \delta},1\right)+J_0\ep - \eta(t,x) \right|\le  
O(R).$}

{\em Proof of Claim 6.} 
By Lemma \ref{partcilescontrollem}, $|x-\z^0(t)|$, $|x-\z^{K_\rho}(t)|=O(R)$. Then, by using that $0<\phi<1$, $\eta(t,\z^0(t))=\ep J_0$ and $\eta(t,\z^{K_\rho}(t))=\ep (J_0+K_\rho)$, we get
\beqs\begin{split}
 \sum_{i=0}^{K_\rho}\ep\phi\left(\frac{x-\z^i(t)}{\ep \delta},1\right)+J_0\ep - \eta(t,x)&\leq \ep(J_0+K_\rho+1)- \eta(t,x)\\&=\eta(t,\z^{K_\rho}(t))-\eta(t,x)+\ep\\&=O(R),
\end{split}
\eeqs
and 
\beqs\begin{split}
 \sum_{i=0}^{K_\rho}\ep\phi\left(\frac{x-\z^i(t)}{\ep \delta},1\right)+J_0\ep - \eta(t,x)&\geq J_0\ep- \eta(t,x)=\eta(t,\z^{0}(t))-\eta(t,x)=O(R),
\end{split}
\eeqs

which proves the claim. 

The lemma for $\rho-R \leq |x-x_0|\leq \rho+4R$ follows from  Claims 4-6, Lemma \ref{psismall} and \eqref{Nepbouneta}.

\subsection{Proof of Lemma \ref{lastlem}}
Recalling that if $x_i^0<x_0$ then $b_i=-1$, while if $x_i^0>x_0$ then $b_i=1$, we write
\beq\label{lastlemfirsteq}
\begin{split}
&\sum_{i=1}^{N_\ep} \ep \phi \left ( \dfrac{x_\ep +b_ic_0L(t_\ep - t_0  +c\sigma)- x_i^0}{\ep \delta}, b_i \right ) \\&
= \sum_{\stackrel{i=1}{x_i^0>x_0}}^{N_\ep}\ep \phi \left ( \dfrac{x_\ep +c_0L(t_\ep - t_0 + c\sigma)- x_i^0}{\ep \delta}, 1 \right )\\
&+\sum_{\stackrel{i=1}{x_i^0<x_0}}^{N_\ep}\ep \phi \left ( \dfrac{x_\ep -c_0L(t_\ep - t_0 + c\sigma)- x_i^0}{\ep \delta}, -1\right ).
  \end{split}
\eeq
Let us show that 
\beq\label{inewlastlem1}
\sum_{\stackrel{i=1}{x_i^0<x_0}}^{N_\ep}\ep \phi \left ( \dfrac{x_\ep \pm c_0L(t_\ep - t_0 + c\sigma)- x_i^0}{\ep \delta}, -1\right )=
 o_\ep(1)-\ep N_\ep^-, \eeq 
 where $N_\ep^-$ is the number of negative oriented particles.
  By \eqref{therightc}, 
  \beq\label{inewlastlem2}x_\ep \pm c_0L(t_\ep - t_0 + c\sigma)=x_0 \pm c_0Lc\sigma +o_\ep(1)=x_0\pm \frac{\sigma}{4}+o_\ep(1).\eeq
  Since  $\tilde\eta^\sigma$ is constant in $x$ for $|x-x_0|\leq \sigma$, if $x_i^0<x_0$ then  
  $$x_0-x_i^0\geq \sigma,$$ which combined with \eqref{inewlastlem2} gives
  $$x_\ep \pm c_0L(t_\ep - t_0 + c\sigma)- x_i^0\ge \frac{\sigma}{2}.$$
 Therefore by \eqref{phiinfinity}, 
 \beqs
\begin{split}
& \sum_{\stackrel{i=1}{x_i^0<x_0}}^{N_\ep}\ep \phi \left ( \dfrac{x_\ep \pm c_0L(t_\ep - t_0 + c\sigma)- x_i^0}{\ep \delta}, -1\right )\\
&\le\sum_{\stackrel{i=1}{x_i^0<x_0}}^{N_\ep} \frac{\ep^2\delta C}{\sigma}-\ep N_\ep^-\\
&= o_\ep(1)-\ep N_\ep^-,
\end{split}
\eeqs
where  we used that $\ep N_\ep\leq C$. Similarly, one can show that 
\beqs  \sum_{\stackrel{i=1}{x_i^0<x_0}}^{N_\ep}\ep \phi \left ( \dfrac{x_\ep \pm c_0L(t_\ep - t_0 + c\sigma)- x_i^0}{\ep \delta}, -1\right )\geq o_\ep(1)-\ep N_\ep^-.\eeqs
This concludes the proof of \eqref{inewlastlem1}.  From \eqref{inewlastlem1} in particular we infer that 
 \beqs
\begin{split}&\sum_{\stackrel{i=1}{x_i^0<x_0}}^{N_\ep}\ep \phi \left ( \dfrac{x_\ep -c_0L(t_\ep - t_0 + c\sigma)- x_i^0}{\ep \delta}, -1\right )\\
&=
\sum_{\stackrel{i=1}{x_i^0<x_0}}^{N_\ep}\ep \phi \left ( \dfrac{x_\ep +c_0L(t_\ep - t_0 + c\sigma)- x_i^0)}{\ep \delta}, -1\right )+ o_\ep(1),
\end{split}\eeqs
which  combined with \eqref{lastlemfirsteq} yields \eqref{lastlemeq}.

This concludes the proof of the lemma.

\section*{}


\begin{thebibliography}{10}






\bibitem{cs}{\sc X. Cabr\'{e} and Y. Sire}, Nonlinear equations for fractional Laplacians II: existence, uniqueness,
and qualitative properties of solutions, {\em Trans. Amer. Math. Soc.},
{\bf 367} (2015), no. 2, 911-941.

\bibitem{csm}{\sc X. Cabr\'{e} and J. Sol\`{a}-Morales}, Layer
solutions in a half-space for boundary reactions, {\em Comm. Pure
Appl. Math.}, {\bf 58} (2005), no. 12, 1678-1732.








\bibitem{cddp}{\sc M. Cozzi, J. D\'{a}vila and M. del Pino,} {Long-time asymptotics for evolutionary crystal dislocation models},
{\em Adv. Math.}, {\bf  371} (2020), 107242.


\bibitem{usersguide}{\sc M. G. Crandall, H.  Ishii and P. L. Lions}, User's guide to viscosity solutions of second order partial differential equations,
{\em  Bull. Amer. Math. Soc. (N.S.)}, {\bf 27} (1992), no. 1, 1-67. 


\bibitem{Denoual}
{\sc C. Denoual},
Dynamic dislocation modeling by combining Peierls Nabarro and Galerkin methods,
{\em Phys. Rev. B}, {\bf 70} (2004), 024106.


\bibitem{dfv}{\sc S. Dipierro, A. Figalli and E. Valdinoci},
Strongly nonlocal dislocation dynamics in crystals,  
{\em Commun. Partial Differ. Equations}, {\bf 39}
(2014) no. 12, 2351-2387.

\bibitem{dpv}{\sc S. Dipierro, G. Palatucci and E. Valdinoci}, Dislocation 
dynamics in crystals: a macroscopic
theory in a fractional Laplace setting, 
{\em Comm. Math. Phys.}, 
{\bf 333} (2015) no. 2, 1061-1105.

\bibitem{dnpv}{\sc E. Di Nezza, G. Palatucci and E. Valdinoci}, Hitchhiker's guide to fractional Sobolev spaces,
{\em Bull. Sci. Math.}, {\bf 136} (2012), no. 5, 521-573. 

 
 \bibitem{fino}{\sc A. Z. Fino, H. Ibrahim and R. Monneau},
The Peierls-Nabarro model as a limit of a Frenkel-Kontorova model
solutions in a half-space for boundary reactions, {\em
J. Differential Equations}, {\bf 252} (2012), no. 1, 258-293.



\bibitem{fim} {\sc N. Forcadel, C. Imbert and  R.  Monneau}, Homogenization of some particle systems with two-body interactions and of the dislocation dynamics, {\em  Discrete Contin. Dyn. Syst.}, {\bf 23} (2009), no. 3, 785-826.

\bibitem{gm}{\sc A. Garroni and S. M\"{u}ller}, $\Gamma$-limit of
a phase-field model of dislocations {\em SIAM J. Math. Anal.},
{\bf 36} (2005), no. 6, 1943-1964.


\bibitem{gmps}{\sc A. Garroni, P. van Meurs, M. Peletier and L.  Scardia}, Convergence and non-convergence of many-particle evolutions with multiple signs., {\em  Arch. Ration. Mech. Anal.},  {\bf 235} (2020), no. 1, 3-49.


\bibitem{glp}{\sc  A. Garroni, G. Leoni and M. Ponsiglione},  Gradient theory for plasticity via homogenization of discrete dislocations, 
{\em J. Eur. Math. Soc. (JEMS)}, {\bf 12} (2010), no. 5, 1231-1266.


\bibitem{gonzalezmonneau}{\sc M. Gonz\'{a}lez and R. Monneau},
Slow motion of particle systems as a limit of a reaction-diffusion
equation with half-Laplacian in dimension one, {\em Discrete Contin. Dyn. Syst.}, {\bf  32} (2012),
no. 4, 1255-1286.


\bibitem{groma}{\sc I. Groma and P. Balogh}, Investigation of dislocation pattern formation in a two-dimensional self-
consistent field approximation, {\em Acta Mater.}, {\bf 47} (1999), no.13, 3647-3654.



\bibitem{hl}{\sc J. R. Hirth and L. Lothe}, Theory of
dislocations, Second Edition. Malabar, Florida: Krieger, 1992.

\bibitem{imr}{\sc C. Imbert, R. Monneau and E. Rouy}, Homogenization
of first order equations with $u/\epsilon$-periodic Hamiltonians.
Part II: application to dislocations dynamics, {\em Communications
in Partial Differential Equations}, {\bf 33} (2008), no. 1-3,
479-516.



\bibitem{jk}{\sc E. R. Jakobsen and K. H. Karlsen}, Continuous
dependence estimates for viscosity solutions of integro-PDEs. {\em
J. Differential Equations}, {\bf 212} (2005), 278-318.





\bibitem{mm}{\sc P. J. P. van Meurs and A. Muntean}, Upscaling of the dynamics of dislocation walls,
{\em Adv. Math. Sci. Appl.}, {\bf 24} (2014), no. 2, 401-414.

\bibitem{mmp}{\sc P. J. P. van Meurs, A. Muntean and M. A. Peletier},  Upscaling of dislocation walls in finite domains,
{\em European J. Appl. Math.}, {\bf 25} (2014), 749-781.


\bibitem{mp}{\sc R. Monneau and S. Patrizi}, Homogenization of the Peierls-Nabarro model for dislocation dynamics, {\em 
J. Differential Equations}, {\bf 253} (2012), no. 7, 2064-2105.

\bibitem{mp2}{\sc R. Monneau and S. Patrizi}, Derivation of the Orowan's law from the Peierls-Nabarro model,  {\em Comm.  Partial Differential Equations}, {\bf 37} (2012), no. 10, 1887-1911. 

\bibitem{mora}{\sc M. G. Mora, M. A. Peletier and L. Scardia}, Convergence of Interaction-Driven Evolutions of Dislocations with Wasserstein Dissipation and Slip-Plane Confinement, {\em SIAM J. Math. Anal.}, {\bf 49} (2017), no. 5, 4149-4205.



\bibitem{MBW}
{\sc A.B. Movchan, R. Bullough and J.R. Willis}, Stability of a
dislocation: discrete model, {\em Eur. J. Appl. Math.} {\bf 9}
(1998), 373-396.




\bibitem{n} {\sc F.R.N. Nabarro}, Dislocations in a simple cubic
lattice, {\em Proc. Phys. Soc.}, {\bf 59} (1947), 256-272.









\bibitem{nabarro} {\sc F. R. N. Nabarro}, 
Fifty-year study of the Peierls-Nabarro stress. {\em Mat. Sci. Eng. A},
{\bf 234-236} (1997), 67-76.

\bibitem{psv}{\sc G. Palatucci, O. Savin and  E. Valdinoci}, Local and global minimizers for a variational energy
involving a fractional norm,
{\em  Ann. Mat. Pura Appl.},  {\bf 192} (2013), no. 4, 673-718.

\bibitem{patsan}{\sc S. Patrizi and T. Sangsawang}, From  the Peierls-Nabarro model to  the equation of motion of the dislocation continuum,
{\em Nonlinear Analysis}, {\bf 202} (2021),  112096.



\bibitem{pv2}{\sc S. Patrizi and  E. Valdinoci}, Crystal dislocations with different orientations 
and collisions, {\em Arch. Rational Mech. Anal.}, \textbf{217} (2015), 231-261.

\bibitem{pv}{\sc S. Patrizi and  E. Valdinoci}, Homogenization and Orowan's law for anisotropic fractional operators of any order, 
{\em Nonlinear Analysis: Theory, Methods and Applications}, \textbf{119} (2015), 3-36.

\bibitem{pv4}{\sc S. Patrizi and  E. Valdinoci}, Long-time behavior for crystal dislocation dynamics, {\em Math. Models Methods Appl. Sci.}, {\bf 27} (2017), no. 12, 2185-2228.


\bibitem{pv3}{\sc S. Patrizi and  E. Valdinoci}, Relaxation times for atom dislocations in crystals, 
{\em Calc. Var. Partial Differential Equations}, {\bf 55} (2016), no. 3, 1-44.

\bibitem{p} {\sc R. Peierls}, The size of a dislocation, {\em Proc. Phys. Soc.}, {\bf
52} (1940), 34-37.


\bibitem{sppg}{\sc L. Scardia, R. H. J. Peerlings, M. A. Peletier and M. G. D. Geers,} Mechanics of dislocation pile-ups:
A unification of scaling regimes, {\em J. Mech. Phys. Solids}, {\bf 70}  (2014), 42-61.





\bibitem{s}{\sc L. Silvestre}, {\em Regularity of the obstacle problem for a fractional power of the Laplace operator}, PhD thesis, University of Texas at Austin (2005).




\end{thebibliography}
\end{document}